\documentclass[11pt,a4paper]{amsart}
\usepackage[utf8]{inputenc}
\usepackage[english]{babel}
\usepackage[left=2cm,right=2cm,top=2cm,bottom=2cm]{geometry}
\usepackage{graphicx}
\usepackage{amssymb}
\usepackage{amsmath}
\usepackage{amsthm}
\usepackage{algorithm}
\usepackage{algorithmic}
\usepackage{comment}
\usepackage{pdfpages}
\usepackage{appendix}
\usepackage{stmaryrd}
\usepackage{tikz}
\usepackage{tikz-cd}
\usetikzlibrary{positioning}
\usepackage{makecell}
\usepackage{mathtools}
\usepackage{hyperref}

\usepackage{bbm}
\usepackage[maxbibnames=50, style=alphabetic,backend=bibtex]{biblatex}

\usepackage{xcolor}
\definecolor{darkgreen}{RGB}{0,100,0}
\definecolor{pslblue}{RGB}{36, 56, 141}
\definecolor{bordeau}{rgb}{0.5,0,0}

\usepackage{csquotes}
\usepackage{hyperref}
\hypersetup{
  colorlinks   = true, 
  linkcolor    = pslblue,
  citecolor    = bordeau,
  urlcolor = bordeau      
}

\providecommand{\ed}{\mathrm e}
\providecommand{\diff}{\mathrm d}

\providecommand{\R}{\mathbb R}

\providecommand{\U}{\mathrm U}
\providecommand{\SU}{\mathrm {SU}}
\providecommand{\Vol}{\operatorname{Vol}}
\providecommand{\Leb}{\operatorname{Leb}}
\providecommand{\honey}{\texttt{\textsc{HONEY}}}
\providecommand{\dualhive}{\texttt{\textsc{DH}}}
\providecommand{\indic}{\mathbbm 1} 
\DeclareMathOperator{\interior}{int}

\numberwithin{equation}{section}
\newtheorem{theorem}{Theorem}[section]
\newtheorem{proposition}[theorem]{Proposition}
\newtheorem{corollary}[theorem]{Corollary}
\newtheorem{lemma}[theorem]{Lemma}
\theoremstyle{definition}
\newtheorem{definition}[theorem]{Definition}

\newtheorem{remark}[theorem]{Remark}

\title{A positive formula for volumes of moduli spaces of flat unitary connections on compact surfaces}
\author{Quentin François, David García-Zelada, Thierry Lévy, Pierre Tarrago}

\begin{document}

\begin{abstract}
We provide a manifestly positive expression for the volume of moduli spaces of flat $\U(n)$-valued connections on punctured compact oriented surfaces. This volume is obtained by summing volumes of explicit polytopes describing coloured honeycombs on a polygon, in the spirit of the work of Knutson and Tao describing the spectrum of the sum of two hermitian matrices. As a corollary, we also provide a positive formula for marginals of the $\U(n)$-valued Yang-Mills measure on a compact oriented surface in terms of the probability distribution of an explicit path process.
\end{abstract}

\maketitle

\section{Introduction}

Moduli spaces of flat unitary connections have been introduced by Narasimhan and Seshadri \cite{Narasimhan_Seshadri,donaldson1983new} 
as a way to classify vector bundles on a compact Riemann surface. As highlighted in the work of Atiyah and Bott \cite{Atiyah_Bott}, 
these  moduli spaces are deeply related to the $\U(n)$-valued Yang-Mills measure on the corresponding surface. 
The latter is a random $\U(n)$-valued connection weighted by its curvature and depending on a temperature parameter. As the temperature vanishes, the only surviving connections are the ones with zero curvature and the zero-limit temperature of the Yang-Mills functional yields a natural symplectic volume form on the moduli 
space of flat connections.

Since the paper of Atiyah and Bott, there have been several works pursuing the goal of a better understanding of either the Yang-Mills 
measure or its zero temperature limit, the volume form on flat connections. Let us mention in particular two pioneering directions that appeared around 1990. 
The first one, initiated by Verlinde \cite{Verlinde}, expresses the volume as the limit of properly normalized dimensions of vector spaces of holomorphic sections on the moduli space. 
This expression is clearly positive as a limit of positive numbers. 
There is an explicit formula 
to compute these integers, originally due to Verlinde and later proved by Bismut and Labourie \cite{bismut5symplectic}, involving sums of complex numbers in the case of the three-holed sphere. However, until recently \cite{Buch_Kresch_Purbhoo_Tamvakis}, no manifestly positive formula existed for the dimension of such spaces.

The second work, due to Witten \cite{witten1992two}, reduces the case of an arbitrary compact oriented surface to a three-holed sphere 
by exploiting a decomposition of pants of the surface. 
Such gluing procedure has since been rigorously proved several times, see \cite{jeffrey1998intersection} and \cite{Meinrenken_Woodward}. 
Using this method, Witten successfully obtained a general formula for the volume of flat connections as a series of characters of the unitary group. 
This first result, which is a series of complex numbers, has been improved to an alternating sum of positive numbers by Jeffrey \cite{jeffrey1994extended} 
and then Meinrenken and Woodward \cite{Meinrenken_Woodward}.

In the present paper, we provide a positive formula for the volume of such moduli spaces as  a sum of volumes of explicit finite dimensional polytopes. Denote by $M_{g,p}(\alpha_1,\ldots,\alpha_p)$ the moduli space of flat unitary connections on an oriented punctured compact surface of genus $g$ with $p$ removed points (in a necessarily trivial $\U(n)$-principal bundle), such that the holonomy around the $i$-th point belong to an orbit $\alpha_i\subset \U(n)$ which is supposed to be non-degenerated. Then, our result may be informally stated as follows, see Theorem \ref{th:Z_g_p_0} for a precise statement:
$$\Vol\left[M_{n,p}(\alpha_1,\ldots,\alpha_p)\right]=\left(\frac{2^{(n+1)[2]}(2\pi)^{(n-1)(n-2)}}{n!n}\right)^{2g+p-2}\frac{n}{\prod_{j=1}^{p} \Delta(\alpha_j)}\sum_{G\in \mathcal{G}^{(g,p)}}\Vol\left[P^G_{\alpha_1,\ldots,\alpha_p}\right],$$
where $\Delta$ is an explicit function involving a Vandermonde determinant and $\{P^G, G\in \mathcal{G}^{(g,p)}\}$ is a collection of explicit polytopes indexed by a family of graphs $\mathcal{G}^{(g,p)}$ and whose boundaries depend on the values of $\alpha_1,\ldots,\alpha_p$. These polytopes consist of coordinates of a collection of segments on a polygonal surface, see Figure \ref{fig:ex_g2_p_3} for a typical example of such configurations. The surface itself simply reflects the pants decomposition 
of the original punctured oriented surface, while the collection of segments is a colored generalization of the honeycomb model introduced by 
Knutson and Tao \cite{knutson2004honeycomb} to describe the spectrum of the sum of two hermitian matrices with prescribed spectra. 
The fact that our model is a generalization of \cite{knutson2004honeycomb} should not be a surprise, since the solution of the eigenvalue problem 
for the sum of hermitian matrices yields a particular subexample of a flat connection on a three-holed sphere, see \cite{Klyachko, belkale2008quantum}. 
Actually, this subexample has been used to get the previously mentioned signed formulas for the volume of the moduli spaces on the three-holed sphere, see \cite{Meinrenken_Woodward}.

The proof of the present result is mainly based on the positive formula given in \cite{françois2024positiveformulaproductconjugacy} in the case of the three-holed sphere. 
Roughly speaking, this formula had been obtained by following the initial method of Verlinde together with the recent results on the computation 
of the structure constants of the quantum cohomology of the Grassmannians \cite{Buch_Kresch_Purbhoo_Tamvakis}. 
In \cite{françois2024positiveformulaproductconjugacy}, the formula was given in terms of volumes of a so-called coloured hive model, 
which consists of a union of some explicit but degenerate polytopes. 

Note that volumes are always defined up to a constant and when polytopes of a Euclidean space are degenerate, there are several ways 
to choose this constant from the volume measure of the ambient space: one can either consider the volume induced by the restriction of 
the scalar product or consider the pull-back of the canonical Lebesgue measure by a bijective projection on a subset of a canonical basis. 
Both definitions agree in the case of a convex body (that is, a convex polytope of the same dimension as the ambient space), but they start to differ when degeneracies appear. 
The choice of a volume form was relatively inconsequential in the case of \cite{knutson2004honeycomb}, since there was only one polytope 
considered and the change of volume form only amounted to a change by an overall constant. 
However, we are considering here a sum of volumes of polytopes which are degenerate and lying each of them in different subspaces of the ambient space, so that we must ensure that the chosen volumes are coherent from one polytope to the other. 
In \cite{françois2024positiveformulaproductconjugacy}, the correct choice was to use the pull-back corresponding to particular bijective projections. 

It appears that the formulation of \cite{françois2024positiveformulaproductconjugacy} in terms of hives does not fit easily in the surgery formulas of Witten. 
In order to generalize the case of the three-holed sphere to arbitrary punctured compact oriented surfaces, our first task has been to reformulate 
the latter hive model in terms of a colored version of honeycombs. As a byproduct of this reformulation, 
we also obtain a new volume formula in the case of the three-holed sphere which is coordinate-free : we provide an explicit relation between 
the volume form induced by the scalar product of the ambient space and any volume induced by a bijective projection onto some coordinates. 
This relation is given by a simple combinatorial number, the number of spanning trees of a graph, in a similar spirit to \cite{kassel2022determinantal}. 
Thanks to the new honeycomb model, the surgery formulas of \cite{witten1992two} to compute the volume of flat connections easily translate into a patching of honeycombs, 
and we thus get the desired positive formula for surfaces of arbitrary genus.

Note that this formula suggests a simple geometric description of a flat $\U(n)$-connection on a compact oriented surface, 
 up to an overall conjugation by a unitary matrix, by the corresponding honeycomb arrangement. It would be very interesting 
 to provide an explicit measure preserving bijection between the two models. However up to now, such a bijection has been out of reach already 
 for the simplest particular case of the sum of hermitian matrices.

As suggested by \cite{Atiyah_Bott}, volumes of flat connections are the building blocks for the computations of the partition function 
of the two-dimensional Yang-Mills measure on compact oriented surfaces, see \cite[Section 2]{witten1992two}. 
Using the new positive formulas for flat connections, one can thus easily deduce new formulas for marginals of the Yang-Mills measure 
which are again manifestly positive. Remark that our formulas are only valid for the holonomy on a set of non-crossing curves, 
and this thus gives only a partial description of the full Yang-Mills measure. Once again, our result suggests the existence of a 
measure preserving bijection between a projection of the Yang-Mills measure and random path processes consisting of Dyson Brownian motions 
on the circle and certain frozen path configurations in polygons, see Remark \ref{rmk:Yang-Mills_volume_conditioned_process}.

\section{Notations and statement of the main result}
\subsection{Conjugacy classes of the unitary group}

Throughout this paper, we fix an integer 
$n \geqslant 3$ and denote
$\mathcal{H}=(\mathbb{R}/\mathbb{Z})^n/S_n$,
where $S_n$ is the symmetric
group of degree $n$ acting on
$(\mathbb R/\mathbb Z)^n$ by permutations
of the coordinates.
As a set, we identify $\mathcal H$ with $\{\theta=(\theta_1,\dots,\theta_n) 
\in [0,1[^n \colon \theta_1 \geqslant \dots \geqslant \theta_n\}$ in the usual way.
For our purposes, $\mathcal H$ represents
the set of conjugacy classes
of $\U(n)$, where
the conjugacy class of
$\theta \in \mathcal H$ is
\begin{equation*}
	\mathcal{O}(\theta) = 
	\left\lbrace Ue^{2\pi i\theta}U^{-1} \mid U\in \U(n)\right\rbrace
	\quad \text{ with } \quad e^{2\pi i\theta}=
	\begin{pmatrix} 
		e^{2\pi i \theta_1}&0&\dots& 0
		\\0&e^{2\pi i\theta_2}&& \vdots\\ 
		\vdots&&\ddots &\\
		0 & \hdots &&e^{2\pi i\theta_n}
	\end{pmatrix}.
\end{equation*}

\noindent
Let us use the notation $\mathcal{H}_{reg}=
\{\theta\in \mathcal{H} \mid
\theta_1>\theta_2>\ldots>\theta_n\}$ which
represents the set of regular 
conjugacy classes of $\U(n)$, namely the ones of maximal 
dimension in $\U(n)$. Finally, let $\mathcal{H}_{reg}^0$ denote the subset of $\mathcal{H}_{reg}$ corresponding to the regular conjugacy classes of $\SU(n)$. Namely,
$$\mathcal{H}_{reg}^0=\left\{\theta\in \mathcal{H} \; 
\mid  \; 
\theta_1>\theta_2>\ldots>\theta_n \text{ and }\sum_{i=1}^n \theta_i\in\mathbb{N}\right\}.$$
Consider the moduli space $\mathcal{M}_{g, n}(\alpha_1, \dots, \alpha_p)$ of flat $\U(n)$-valued connections 
on a compact oriented surface with genus $g$ having $p$ 
boundary components $\mathcal{L}_1,\ldots,\mathcal{L}_p$, for which the holonomy around 
each $\mathcal{L}$, $1\leqslant i\leqslant p$, belongs to 
$\mathcal{O}_{\alpha_i}$. Then, there is a symplectic structure on the quotient space $M_{g, n}(\alpha_1, \dots, \alpha_p)=\mathcal{M}_{g, n}(\alpha_1, \dots, \alpha_p)/\mathcal{G}_n$, where $\mathcal{G}_n$ is the $\U(n)$-valued gauge group on the surface, see \cite{meinrenken1999moduli}. The space $M_{g, n}(\alpha_1, \dots, \alpha_p)$ inherits thus a canonical volume form from the symplectic form. 

The main goal of the present paper is to provide an explicit and positive formula for the corresponding volume of $M_{g, n}(\alpha_1, \dots, \alpha_p)$. This formula is given in terms of certain configurations of lines, called honeycomb, which are a generalization of the honeycomb model introduced in \cite{knutson2004honeycomb} and are described in the next two paragraphs.

\subsection{Triangular honeycomb}\label{subsec:triangular_honey} Let us a first define a triangular honeycomb, which will contribute to the volume of the moduli space in the case $g=0,p=3$, see Theorem \ref{th:volume_flat_connection_0_3}. 

\noindent
Consider $S\subset \mathbb{R}^2$, a convex subdomain of the plane. Recall that a segment of $S$ is any set of the form $\{xv+(1-x)v', 0\leq x\leq 1\}$ for some $v,v'\in S$. For any such segment $e=\{xv+(1-x)v', 0\leq x\leq 1\}$, we then set $\partial e=\{v,v'\}$ 
and denote by $\overset{\circ}{e} =e\setminus \{v,v'\}$ its interior. A segment is then called non-trivial if $\overset{\circ}{e}\not=\emptyset$. Denote by  $\mathbb{G}$ the set of non-trivial segments of $S$.

If $e,e'$ are two non-trivial segments, we define the angle from $e$ to $e'$ as 
$$\widehat{(e,e')}=\arccos(\langle u,u'\rangle) \ ,$$
where $u,u'$ are unit tangent vectors of respectively $e$ and $e'$ such that $\det(u,u')>0$.

\begin{definition}[Non-degenerate honeycomb]
\label{def:toric_honeycomb}
A \textit{honeycomb} $h=(\mathcal{E},c)$ is a finite subset $\mathcal{E} \subset \mathbb{G}$ 
together with a color map $c:\mathcal{E}\rightarrow \{0,1,3\}$ such that :
\begin{enumerate}
\item $\overset{\circ}{e}\cap \overset{\circ}{e}'\not=\emptyset$ 
is only possible if, up to a transposition of $e$ and $e'$, 
$c(e)=0, \,c(e')=1$ and $\widehat{(e,e')}=\pi/3$,
\item if $\partial e\cap \partial e'=\{v\}$, 
then $(c(e),c(e'))\in\{(0,0),(1,1),(0,1),(1,3),(3,0)\}$ and $\widehat{(e,e')}=2 \pi/3$.
\end{enumerate}
\end{definition}
\noindent
From the definition that $\mathcal{E}$ is unambiguously defined from the representation $\bigcup_{e\in \mathcal{E}}e$ of $h$ as a subset of $S$. In the sequel, we set 
 $$t(h)=\bigcup_{e\in \mathcal{E}}e.$$
\\
\noindent
Figure \ref{fig:segments} lists the possible angles between segments in non-degenerate 
honeycombs of Definition \ref{def:toric_honeycomb}.

\begin{figure}[ht]
    \centering
    \includegraphics[scale=1]{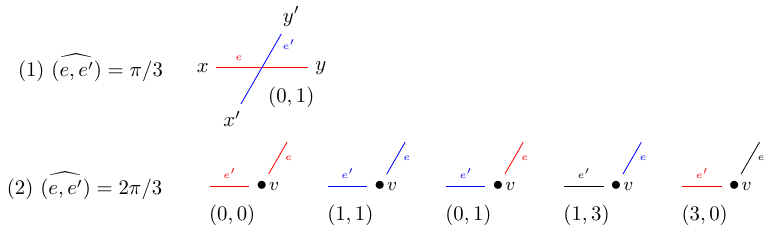}
    \caption{The possible angles between segments. The colors of the 
    edges are given in the couple below the configuration.}
    \label{fig:segments}
\end{figure}

\noindent

Remark that our choice of colors is coherent with the colors coming from \cite{Buch_Kresch_Purbhoo_Tamvakis} and \cite{françois2024positiveformulaproductconjugacy} : the colors $0$ and $1$ play a symmetric role that differs from the one played by the color $3$.

\begin{definition}
The \textit{structure graph} of a honeycomb 
$h$ is the finite graph $G(h)=(V,E)$ with colored edges such that $V=\bigcup_{e\in \mathcal{E}}\partial e$, $E=\{\partial e, e\in \mathcal{E}\}$, and the color map $c:E\rightarrow\{0,1,3\}$ is defined by $c(v,v')=c(e)$ if $\partial e=\{v,v'\}$.
\end{definition}

\noindent
By the angle condition, $G = G(h)$ has only vertices of degree less than $3$ 
and the sequence of colors around any trivalent 
(resp. bivalent) vertex of $G(h)$ belongs to 
$\{(0,0,0),(1,1,1),(0,3,1)\}$ (resp. is equal to $(0,1)$) 
in the clockwise order. 
\\
The \textit{color number} $c(G)$ of a colored graph $G$ is defined as 
the number of edges 
colored $1$ and adjacent to a univalent vertex. By abuse of definition, we speak of edges and vertices 
of a honeycomb to denote edges and vertices of its 
canonical structure graph. 
\\
\\
Let us denote by 
$T \coloneqq \{x+ye^{i\pi/3} \mid 0\leqslant x,y\leqslant 1, \ x+y\leqslant 1\} 
\subset \mathbb{C}$ the equilateral triangle with vertices $0, 1$ 
and $e^{i\pi/3}$. To each point $v \in T$ we associate the 
triple $(v_0,v_1,v_2)$ such that $v=v_1+v_2e^{i\pi/3}$ 
and $v_0=1-v_1-v_2$.
Then, the boundary $\partial T$ can be decomposed as
$$\partial T=\bigsqcup_{i\in\{0,1,2\}}\partial_iT \text{ , where } 
\partial_iT=\{v\in T \mid \,v_i=0\} \ .$$

\begin{definition}[Triangular honeycomb]
    \label{def:triangular_honey}
    For $n\geq 1$, a \textit{triangular honeycomb} $h$ of size $n$ is a non-degenerate 
    honeycomb $h=(\mathcal{E},c)$ on the surface $T$ endowed 
    with the Euclidian metric such that
    \begin{enumerate}
    \item $G(h)$ has only univalent and trivalent vertices, 
    and $t(h)\cap \partial T=\partial V$, where 
    $\partial V$ denotes the set of univalent vertices in $G(h)$,
    \item if $e\in \mathcal{E}$, then $e \subset \left\{x+\mathbb{R} \ed^{2\pi i(\ell(e) +1) / 3}\right\} $ 
    for some $\ell(e)\in\{0,1,2\}$ and $x\in T$. 
    The integer $\ell(e)$ is then called the \textit{type} of $e$ 
    and $L(e)=x_{\ell(e)}$ is called the \textit{height} of $e$ 
    which is independent of the choice of $x$, 
    \item for $0\leqslant i\leqslant 2$, $\#t(h)\cap \partial_i T=n$ and if $e$ 
    is adjacent to a boundary vertex belonging to $\partial_i T$, 
    then either $c(e)=0$ and $\ell(e)=i+1$ or 
    $c(e)=1$ and $\ell(e)=i+2$. Moreover, 
    the color is increasing along each edge : 
    namely, if $e^1$ (resp. $e^2$) meets 
    $\partial_iT$ at $x^1$ (resp. $x^2$) with 
    $x_{i+1}^2>x_{i+1}^1$, then $c(e^2)\geqslant c(e^1)$. 
    \end{enumerate}
\end{definition}

\noindent
A triangular honeycomb 
has always $n^2+3n$ vertices and $\frac{3n(n+1)}{2}$ edges, 
see Remark \ref{rmk:graph_vertices_edges}.
This definition implies that $G(h)=(V,E)$ has only trivalent 
vertices in $\overset{\circ} T$. Moreover, the condition on 
the boundaries yields a natural choice of root 
$v^0$ of $G$, corresponding to the univalent vertex lying on 
$\partial_0T$ and whose coordinate $v^1_1$ is maximal. 
Then, the cyclic counter-clockwise order on the boundary vertices 
given by the orientation of $T$ yields an order on the $3n$ boundary vertices 
$\{v^1,\ldots,v^{3n}\}$ of $\partial V$, with the vertices $\{v^{ni+j},1\leqslant j\leqslant n\}$ located on $\partial_iT$.
\\
In the particular case where all edges have the same color, 
our definition is the original definition of a generic honeycomb 
from \autocite{Knutson_Tao_honeycomb}. Remark that, besides the coloring, 
our definition of triangular honeycombs differs slightly from the original 
one since we impose vertices to be trivalent inside $\interior(T)$. 
The set of honeycombs in their original definition can then be seen as the closure 
of the ones from the present manuscript (in the case where all colors are the same).
\\
\\
The \textit{boundary}, or the \textit{boundary values}, 
of a triangular honeycomb $h$ of size $n \geqslant 3$ 
is the $3n$ tuple 
\begin{equation}
    \label{eq:boundary_honeycomb_order}
    \partial h \coloneqq 
    \left((\alpha^0_n \leqslant \dots\leqslant  \alpha^0_1), 
    (\alpha^1_n\leqslant \dots\leqslant \alpha^1_1), 
    ( \alpha^2_n\leqslant \dots\leqslant \alpha^2_1)\right),
\end{equation}
where $(\alpha^i_n\leqslant \dots\leqslant \alpha^i_1)$ is the ordered 
tuple of the $i+1$-coordinates of boundary points of 
$t(h)$ on $\partial_iT$. Hence, 
$\alpha^i_j=v^{ni+j}_{i+1}$ for $0\leqslant i\leqslant 2$ and $1\leqslant j\leqslant n$.
\\
\\
For $\alpha,\beta,\gamma\in \mathcal{H}_{reg}$ let us 
denote by $\honey_{n, d}(\alpha, \beta, \gamma)$ 
the set of triangular honeycombs $h$ having boundary 
values $ \partial h = (\beta, \alpha, \gamma)$, 
see Figure \ref{fig:boundary_toric_honey}, and such that 
there are $d$ edges colored $1$ meeting one (and thus each) 
boundary component of $T$. 
For any colored graph $G$ with an order on the boundary vertices, let us denote by 
\begin{equation*}
    \honey_{n,d}^G({\alpha,\beta,\gamma})
\end{equation*}
the set of triangular honeycombs of $\honey_{d, n}(\alpha, \beta, \gamma)$ with canonical graph 
structure isomorphic to $G$ as colored graph with ordered boundary. Then, if $\honey_{n,d}^G({\alpha,\beta,\gamma})$ is non-empty,  necessarily the graph $G$ has $3d$ univalent vertices adjacent to edges colored $1$. We then define the color of $G$ as $c(G)=d$.
Let us denote by $\mathcal{G}_d$ the set of isomorphism classes of colored 
graphs with ordered boundary appearing in $\{G(h) \mid h\in\honey_{n,d}\}$.

\begin{figure}[h]
    \centering
    \includegraphics[scale=0.6]{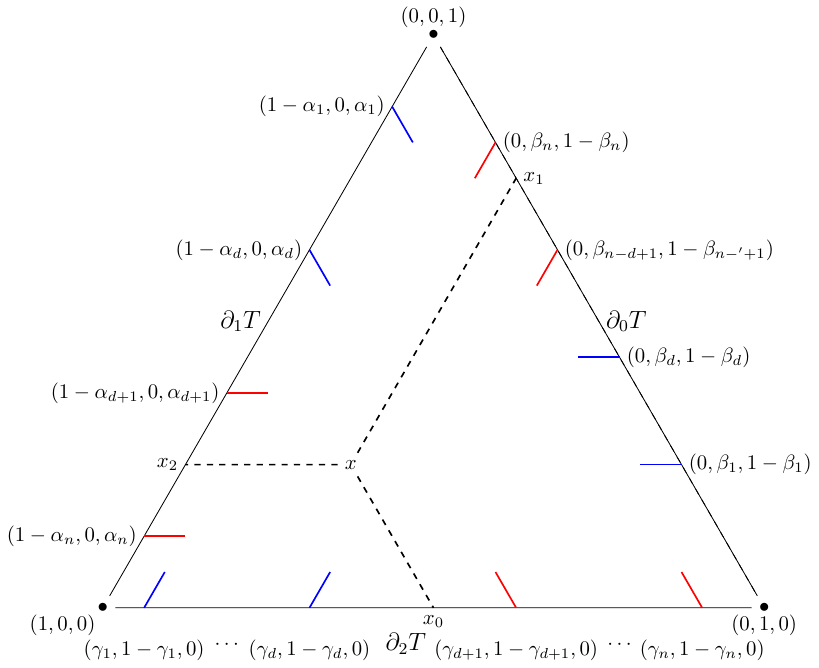}
    \caption{Boundaries of a honeycomb in $\honey_{n, d}(\alpha, \beta, \gamma)$.}
    \label{fig:boundary_toric_honey}
\end{figure}

\subsection{ \texorpdfstring{$(g,p)-$}{(g, p)-}honeycombs}\label{subsec:(g,p)_honey} We next turn to a generalization of triangular honeycombs which will provide a formula for arbitrary surfaces. Let $g,p\geqslant 0$ be integers 
such that $N=p + 2g-2 \geqslant 1$ and let $S$ be a 
connected compact orientable surface of genus $g$ with 
$p$ boundary components $\mathcal{L}_1, \ldots, \mathcal{L}_{p}$. 
Let $\mathcal{M}=(M_1\ldots, M_{N})$ be a pant decomposition 
of this surface. 
We build from $\mathcal{M}$ a surface $\mathcal{T}$ as follows: 
\begin{itemize}
    \item take $N$ oriented equilateral triangles 
    $(T^1,\ldots,T^{N})$, each $T^i$ having three oriented 
    boundaries $\partial_{j}T^i$ of size $1$. 
    \item For each pair of boundary components $\partial_jM^i, 
    \partial_{j'}M^{i'}$ which are identified in the pair of pants 
    decomposition, write $(i,j)\sim(i',j')$ and identify the boundaries $\partial_jT^i$ and 
    $\partial_{j'}T^{i'}$ in an isometric and orientation reversing way. 
\end{itemize}
Then, the resulting surface $\mathcal{T}$ is an oriented surface 
with $p$ boundaries edges denoted by
$L_1,\ldots,L_p$. If $L_i$ a boundary component, 
there exists a triangle $T_{j_i}$ and $\ell_i\in\{0,1,2\}$ 
such that $L_i=\partial_{\ell_i} T^{j_i}$. 

For each $1\leq i\leq p+2g-2$, let $f_i:T\rightarrow \mathcal{T}$ be an orientation preserving isometry such that $f_i(T)=T^i$ (where $T^i$ is seen as a subset of $\mathcal{T}$ by the above construction). For a triangular honeycomb $h=(\mathcal{E},c)$, set $t_i(h)=\bigcup_{e\in\mathcal{E}}f_i(e)$. 
\begin{definition}\label{def:(g,p)_honey}
A \textit{$(g,p)$-honeycomb }is a tuple $(h^i)_{1\leq i\leq N}$ of triangular honeycombs such that\\ $t_i(h^i)_{\vert \partial_j T^i}=t_{i'}(h^{i'})_{\vert \partial_{j'}T^{i'}}$ for any $(i,j)\sim(i',j')$.
\end{definition} 
See Figure \ref{fig:ex_g2_p_3} for the representation of a $(g,p)$-honeycomb for $g=2$ and $p=3$.
\begin{remark}
The natural euclidean metric on each equilateral 
triangle yields a metric on $\mathcal{T}$ which is flat 
except at the vertices of the triangulation belonging to 
the interior of $\mathcal{T}$. Hence,  writing $h^i=(\mathcal{E}^i,c^i)$ for $1\leq i\leq 2p+g+2$, $\bigcup_{i=1}^{N}\mathcal{E}^i$ defines again a non-degenerate honeycomb on $\mathcal{T}$, where we extended Definition \ref{def:toric_honeycomb} from a subdomain of $\mathbb{R}^2$ to any locally flat Riemannian oriented surface.
\end{remark}

\noindent
Let us denote by $\honey^{(g, p)}$ the set of 
$(g,p)$-honeycombs. For $h=(h^i)_{1\leqslant i\leqslant N}\in\honey^{(g,p)}$ and $1 \leqslant i \leqslant N$, let $G(h^i)=(V^i, E^i)$ be 
the structure graph of the honeycomb $h_i$ with corresponding color map $c_i:E^i\rightarrow \{0,1,3\}$. 
We then define the structure graph of $h$ as the graph $G(h) = (V, E)$ where
$$V = \bigcup_{1\leqslant i\leqslant N}f_i(V^i) \ \text{ and } \ 
E = \bigcup_{1\leqslant i\leqslant N} \{f_i(e), e\in E^i\} \ .$$
We define a coloring of the edges of this graph by setting $c(f_i(e))=c_i(e_i)$ if $e\in E^i$.

\noindent
If $h$ is a $(g,p)$ honeycomb and $1\leq i\leq p$, we set $\partial_i h = t(h)\cap L_i$. Each element $v$ of $t(h)\cap L_i$ belongs to a unique triangle $T^{j_i}$ and inherits a triple of coordinate $(v_0,v_1,v_2)$ from this triangle.  Moreover, as in the triangular case, the orientation of $S$ yields an orientation on each boundary component. 
For $h\in \honey^{(g,p)}$, order the boundary points so that 
$$t(h)\cap \bigcup_{1\leqslant i\leqslant p}L_i=\{v^{(i-1)n+j},1\leqslant i\leqslant p,j\leqslant n\},$$ 
where  $\partial_ih=\{v^{(i-1)n+j},\, 1\leqslant j\leqslant n\}$ with $v^{(i-1)n+j}_{\ell_i+1}>  v^{(i-1)n+j'}_{\ell_i+1}$ if $1\leqslant j<j'\leqslant n$. Hence, the structure graph $G(h)$ inherits the previous order on its boundary points. 
\\
We denote by $\honey^{(g,p)} \left(\alpha^1,\ldots,\alpha^p\right)$ 
the set of $(g,p)$-honeycombs with boundary components 
$\partial_ih=\alpha^i$.
\\
Like in the triangular case, we denote by $\honey^{G}$  
the subset of $\honey^{(g,p)}$ 
of honeycombs $h$ with structure graph $G(h)$ isomorphic to $G$ as colored graph with ordered boundary vertices. 
We also denote by $\mathcal{G}^{(p, g)}$ the set of isomorphism classes of 
colored graphs with ordered boundary appearing as structure graph of elements of $\honey^{(g,p)}$.
For $G\in\mathcal{G}^{(g,p)}$, let us set 
$$c(G)=\sum_{i=1}^Nc(G_i) \ ,$$
where $c(G_i)$ is the color number of the structure graph $G_i$ as defined at the end of Section \ref{subsec:triangular_honey}. Let us denote by $\mathcal{G}^{(g,p)}_{d}\subset \mathcal{G}^{(g,p)}$ 
the subset of graphs $G$ such that $c(G)=d$.

\begin{figure}[h!]
    \centering
    \includegraphics[width=\textwidth]{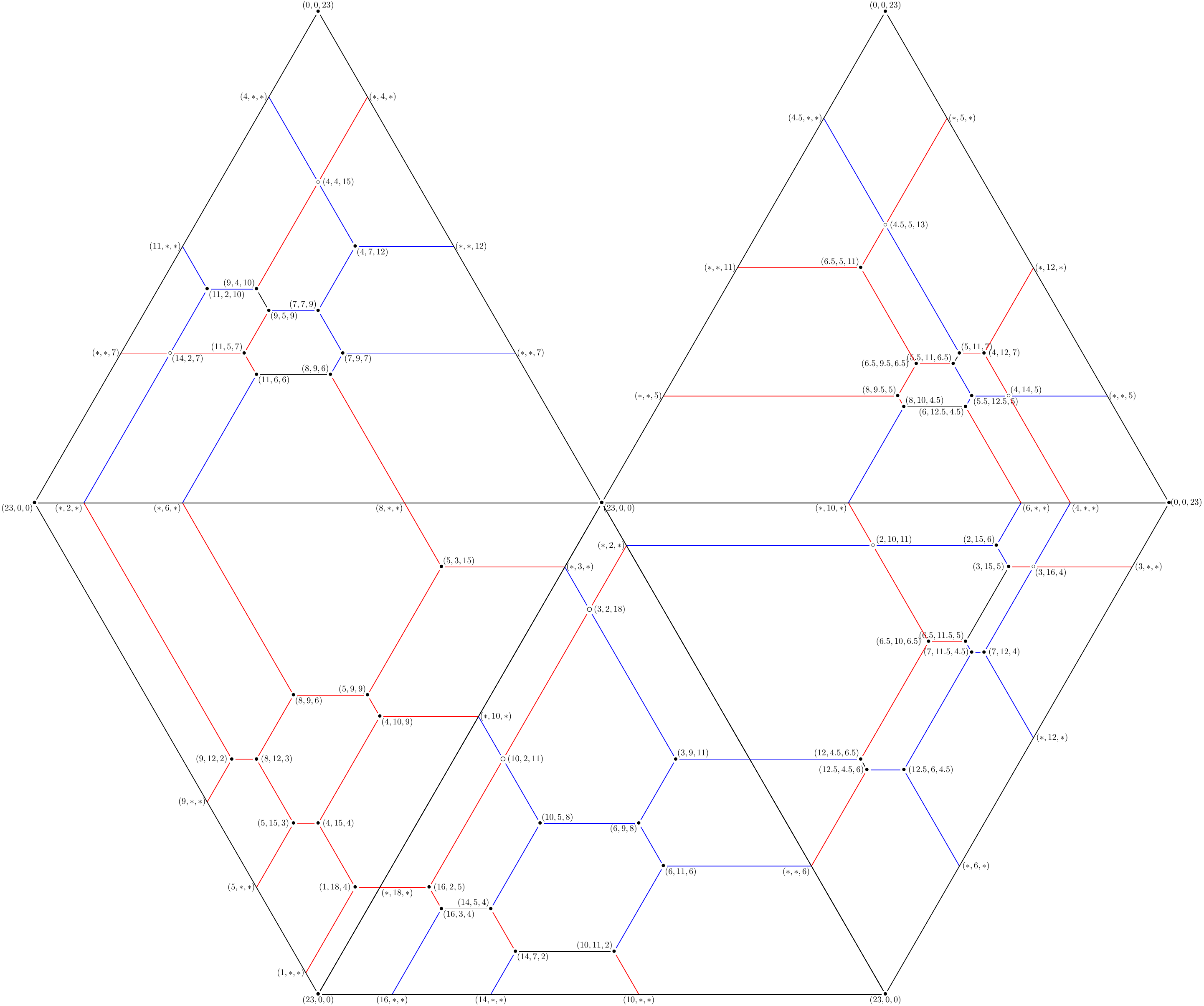}
    \caption{A $(2, 3)$ honeycomb for a surface with genus 2 and 3 boundaries. On each upper triangle, the endpoints of the segments of the honeycomb coincide on the left and on the right edge.}
    \label{fig:ex_g2_p_3}
\end{figure}

\noindent
For $G = (V, E) \in\mathcal{G}^{(p, g)}$, there is a natural parametrization of 
$\honey^G$ constructed as follows. For any edge $e\in E$, there exists $1\leqslant i\leqslant N$ such that 
$e=f_i(\tilde{e})$ with $\tilde{e}\in E^i$. For $h\in\honey^G$, let us set 
$$P[h](e)=L(\tilde{e}) \ ,$$ 
where $L$ is the height map defined on $T$ from 
Definition \ref{def:triangular_honey}. 
Then, introduce the map
\begin{equation}
    \label{eq:parametrization_P}
    P: \honey^{G}\rightarrow \mathbb{R}^{E} \ ,
\end{equation}
sending $h\in \honey^{G}$ to $(P[h](e))_{e\in E}$. 
This map is clearly non-surjective 
since $P \left(\honey^G\right)$ is a bounded subset of 
$\left(\mathbb{R}^{3n(n+1)/2}\right)^N$. 
Moreover, there are several relations among the 
values of $((P[h](e))_{e\in G_h^i})_{1\leqslant i\leqslant N}$
which imply that $P$ is an over-parametrization of $\honey^{G}$. 
However, it will follow from 
Proposition \ref{prop:parametrization_triangular_honey_first} that $P$ is injective. 
In the sequel, we identify $\honey^G$ with its image through $P$. 
\\
\\
As a consequence of the decomposition described in Section \ref{sec:sieving_honeycombs} below, 
all graphs of $\mathcal{G}^{(g,p)}$ have the same number $n_{e} = \frac{3Nn(n+1)}{2}$ 
of edges and $n_{v} = Nn^2+\frac{(3N - (p+2g))n}{2} $ of vertices. 
Hence, up to identifying edges and vertices of these graph, one can assume that there 
is a unique set $E^{(g,p)}$ of edges (resp. set $V^{(g,p)}$ of vertices) 
common to all graphs of $\mathcal{G}^{(g,p)}$ and that the graph structure of 
$G\in \mathcal{G}^{(g,p)}$ is encoded in the map $\partial:E^{(g,p)}\rightarrow 
\mathcal{P}(V^{(g,p)})$ which associates to an edge its endpoints.

\subsection{Volume formula for the moduli space of flat connections}

For $g,p\geqslant 0$, Set $\diff \Vol\in\Omega^{n_{g,p}}\left(\mathbb{R}^{E^{(g,p)}}\right)$ 
for the volume form on $n_{g,p}$-dimensional subspaces induced by the 
canonical scalar product on $\mathbb{R}^{E^{(g,p)}}$. Hence, 
\begin{equation}
    \label{eq:volume_R_E}
    \diff \Vol = \sum_{1\leqslant i_1<\ldots<i_{n_{g,p}} \leqslant n_{e}}\diff x_{i_1} 
    \wedge\ldots \wedge \diff x_{i_{n_{g,p}}}.
\end{equation}
In the following statement, we write $\Vol(K)$ for the integration of $ \diff \Vol$ on a 
$n_{g,p}$-dimensional submanifold and we let $\Delta(x)=2^{n(n-1)/2}\prod_{i<j}\sin(\pi(x_i-x_j))$ 
denote the absolute value of the Vandermonde determinant of $(e^{ix_j})_{1\leqslant j\leqslant n}$.

Let $Z_{g, p}(\alpha_1, \dots, \alpha_p)$ be the volume function for the moduli space of flat unitary connections on a compact oriented surface with genus $g$ and with $p$ boundary components around which the holonomies belong to $\alpha_1,\ldots,\alpha_p$. Since the product of the determinant of the holonomies around the removed point is the identity, $Z_{g, p}(\alpha_1, \dots, \alpha_p)$ is non-zero only if $\sum_{i=1}^p\vert \alpha_i\vert\in \mathbb{N}$, where $\vert x\vert=\sum_{j=1}^n\vert x_j\vert$ for $x\in \mathbb{R}^n$.
\begin{theorem}[Volume formula for $(g, p)$ partition function]
    \label{th:Z_g_p_0}
    
    For $\alpha_1, \dots, \alpha_p\in\mathcal{H}_{reg}$ such that \\ $\sum_{i=1}^p\vert \alpha_i\vert\in \mathbb{N}$, 
    \begin{equation}
        \label{eq:Z_g_p_0}
        Z_{g, p}(\alpha_1, \dots, \alpha_p) = 
        \frac{c_{0,3}^{N}}{n^{2g+p-3}\prod_{j=1}^{p} \Delta(\alpha_j)}
        \sum_{\substack{ G \in \mathcal{G}^{(g,p)}}}
        \frac{\Vol \left[ \honey^G(\alpha_1, 
        \dots, \alpha_p) \right]}{\sqrt{\# \mbox{ Spanning trees of } G}},
    \end{equation}
   where $c_{0,3}=\frac{2^{(n+1)[2]}(2\pi)^{(n-1)(n-2)}}{n!}$.
\end{theorem}

\noindent
The sum on the right hand-side is finite since the set $\mathcal{G}^{(g,p)}$ is finite. 
As it will appear below, for given $\alpha_1, \dots, \alpha_p\in\mathcal{H}_{reg}$ 
the volumes appearing in the sum will be non-zero only for a subset $G\in \mathcal{G}_{d}^{(g,p)}$ 
where $d=\sum_{i=1}^p\sum_{j=1}^n\alpha_j^i-(p-2)n$. 
\\
The formula of Theorem \ref{th:Z_g_p_0} also yields a formula for 
$\SU(n)$-valued connections, since the volume for the $\SU(n)$ case is equal 
to the one of the $\U(n)$ case for $\alpha_1,\ldots,\alpha_p\in\mathcal{H}_{reg}^0$, 
see \eqref{eq:SU_n_equal_Un}. 

There are several other conventions to define the volume of flat connections, depending on the chosen volume on the unitary group (see \cite[Section 4.1]{witten1992two}) and depending on the kind of volume form considered. On the latter convention, remark that we are choosing here the symplectic volume, more adapted to Yang Mills generalizations, rather than the torsion volume introduced in \cite{witten1992two}, see also \cite[Appendix A]{Meinrenken_Woodward} for a symplectic point of view on the relation between both volumes. The volume forms are easily related by the explicit factor $\prod_{j=1}^{p} \Delta(\alpha_j)$, see \cite[Section 4.7]{witten1992two}.

\begin{remark}[Volume and conditioned processes]\label{rmk:_volume_conditioned_process}
For each $G\in\mathcal{G}^{(g,p)}$, choosing a spanning tree $S\subset E^{(g,p)}$ of $G$ yields a natural parametrization of $\honey^G(\alpha_1,\ldots,\alpha_p)$ in terms of  the coordinates $P(e),e\in S$. By construction, each of those coordinates belongs to $[0,1]$. Considering the uniform probability measure on $[0,1]^S$, one can prove that the formula $\frac{\Vol \left[ \honey^G(\alpha_1, 
        \dots, \alpha_p) \right]}{\sqrt{\# \mbox{ Spanning trees of } G}}$ is then exactly the probability that the vector $(L(e))_{e\in S}$ belongs to some convex set $K_G$, see Lemma \ref{lem:differential_gluing_honey} for a definition of $K_G$. After the summation on all graphs $G$, $Z_{g, p}(\alpha_1,\ldots,\alpha_p)$ can be seen as the probability that a random process remains in a union of convex set, up to an explicit constant. Visually, the non-conditioned process would allow segments to exit the domain and to meet at the trivalent vertices with angles being any multiple of $\pi/3$. 
\end{remark}
\subsection{Yang-Mills marginal for disjoint curves}\label{subsec:Yang-mills_result}
As a consequence of Theorem \ref{th:Z_g_p_0}, there is an explicit formula 
for the marginal of $\U(n)$-valued Yang--Mills partition function of an
oriented surface 
of genus $g$ with prescribed non-degenerate holonomies 
(up to conjugation) on a finite set of disjoint loops. This 
formula is given in Corollary \ref{cor:yang_mills_g_p} below.
As it is proven in \cite{levy2003yang}, the partition function 
only depends on the prescribed conjugacy classes and on the 
areas of each connected components delimited by the loops. 
\\
\\
Let $S$ be a connected compact oriented surface of genus $g \geqslant 0$
together with $p$ disjoints Jordan curves 
$\Gamma_1,\ldots,\Gamma_p$ on $S$. For each 
$\Gamma_i, \,1\leqslant i\leqslant p$, let $\alpha_i$ be an element of 
$\mathcal{H}_{reg}$.
\\
We associate to $(S,\Gamma_1,\ldots,\Gamma_p)$ a labeled finite directed graph $\mathcal{T}=(V,E)$ such that vertices are labelled by $\mathbb{R}_{\geqslant 0} \times \mathbb{N}$ and directed edges are labelled by $\mathcal{H}_{reg}$ as follows :
\begin{itemize}
\item the set $V$ of vertices of $\mathcal{T}$ is the set of connected components of $S\setminus \bigcup_{i=1}^p\Gamma_i$. 
Each vertex $v\in V$ is labelled $(A_v,g_v)$ where $A_v$ is the area of the corresponding connected component and $g_v$ is its genus.
\item For $v_1,v_2\in V$, there is a directed edge $e=(v_1,v_2)$ from $v_1$ to $v_2$ 
if $v_1$ and $v_2$ are boundary components of a loop $\Gamma_i$, and we write $v_1=s(e)$. We then label the oriented edge $e$ by $\alpha_e=\alpha_i$ if the loop $\Gamma_i$ is positively oriented on $v_1$ and otherwise by $\alpha_{e}=1-\alpha_e$ (hence, in any case, $\alpha_{(v_1,v_2)}=1-\alpha_{(v_2,v_1)}$).
\end{itemize}

\noindent
We refer to Figure \ref{fig:fig_loops} for an example of such a configuration of loops and its corresponding 
edge-labelled tree. In the following, let us denote by $d_v$ the degree of a vertex $v\in V$.

\begin{figure}[h!]
    \centering
    \begin{minipage}{0.55\textwidth}
        \centering
        \includegraphics[width=\linewidth]{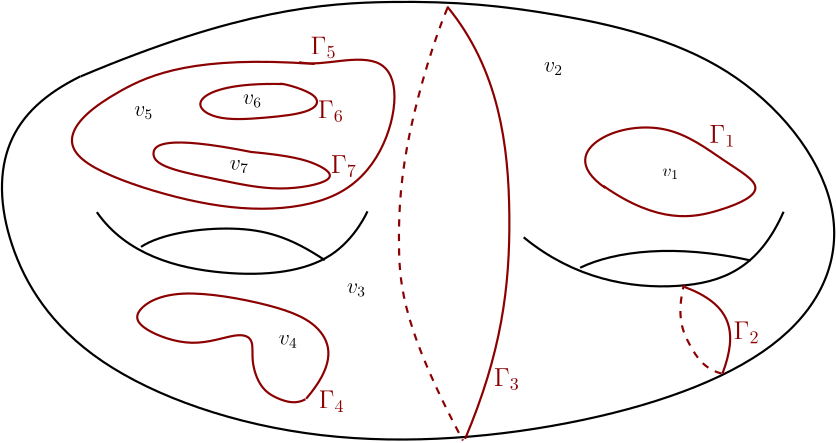}
    \end{minipage}
    \hfill
    \begin{minipage}{0.35\textwidth}
        \centering
    \scalebox{0.7}{\begin{tikzpicture}[
    vertex/.style={minimum size=1cm},
    edge/.style={thick}
]

\node[vertex] (v7) {$v_7$};
\node[vertex, above right=1cm and 0.5cm of v7] (v6) {$v_6$};
\node[vertex, below right=1cm and 0.5cm of v6] (v5) {$v_5$};
\node[vertex, above right=0.5cm and 1cm of v5] (v3) {$v_3$};
\node[vertex, above left=1cm and 0.5cm of v3] (v4) {$v_4$};
\node[vertex, below right=1cm and 0.3cm of v3] (v2) {$v_2$};
\node[vertex, below left=0.5cm and 1cm of v2] (v1) {$v_1$};

\draw[edge] (v2) to [bend right] node[above] {$\Gamma_1$} (v1);
\draw[edge] (v2) edge[loop above] node {$\Gamma_2$} ();
\draw[edge] (v2) to [bend left] node[left] {$\Gamma_3$} (v3);
\draw[edge] (v3) to [bend right] node[above] {$\quad\Gamma_4$} (v4);
\draw[edge] (v3) to [bend left] node[left] {$\Gamma_5\quad$} (v5);
\draw[edge] (v5) to[bend right] node[above] {   $\quad\Gamma_6$} (v6);
\draw[edge] (v5) to[bend left] node[above] {$\Gamma_7$} (v7);

\end{tikzpicture}}
    \end{minipage}
    \caption{A disjoint loops configuration and its corresponding edge-labelled tree}
    \label{fig:fig_loops}
\end{figure}
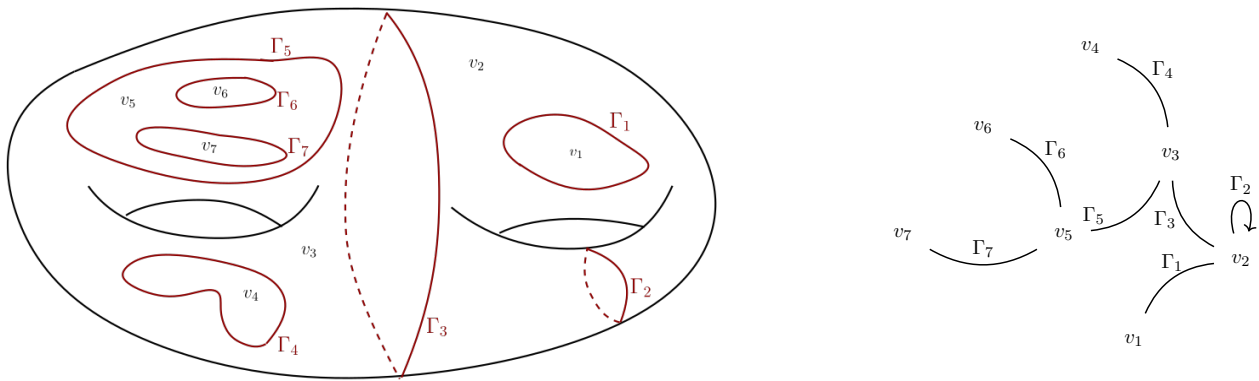
 For $T \geqslant 0$ and $x,y \in \mathcal
H$, let us denote by $k_T(x,\cdot)$ the transition kernel from $x$ to $y$ of the unitary Dyson Brownian motion on $\mathcal{H}$, which is the projection on $\mathcal{H}$ of the $\U(n)$-valued Brownian motion starting at $\mathcal{O}(x)$, see \eqref{eq:kernel_brownian}.

\begin{corollary}[Yang-Mills partition function]
    \label{cor:yang_mills_g_p}
    The Yang-Mills partition function associated 
    to the data $(S,\Gamma_1,\ldots,\Gamma_p)$ is 
    \begin{equation}
        \label{eq:yang_mills_marginal_g_p}
        \operatorname{YM}(\alpha_1,\ldots,\alpha_p) 
        =\int_{\mathcal{H}^{\sum_{i=1}^{\#V}d_v}} 
        \prod_{v\in V}Z_{g_v, d_v}(\mu_1^v, \dots, \mu_{d_v}^v) 
        \prod_{e\in E, v=s( e)}k_{A_v/d_v}(1-\mu_i^v,\alpha_e) 
        \diff \mu_i^v \ ,
    \end{equation}
where $Z_{0, 1}(\mu)=\delta_{0}$ and $Z_{g, p}$ 
is given in \eqref{eq:Z_g_p_0} otherwise.
\end{corollary}

\begin{remark}[Yang-Mills marginals and conditioned processes]\label{rmk:Yang-Mills_volume_conditioned_process}
As a consequence of Theorem \ref{th:Z_g_p_0} and Remark \ref{rmk:_volume_conditioned_process}, the latter formula expresses the partition function $ \operatorname{YM}(\alpha_1,\ldots,\alpha_p)$ as the probability that a random process remains in some domain and has the right boundary values. This random process consists in concatenations of Brownian motions on circle and straight paths in some frozen domains and the domain of restriction corresponds to the condition of being non-crossing for the Brownian motion and yielding $(g,p)-$honeycombs in the frozen domains. It would be interesting to see wether such interpretation allows simulations of a generical process contributing to the formula \eqref{eq:yang_mills_marginal_g_p}.
\end{remark}

\subsection*{Organisation of the paper} 

In Section \ref{subsec:canonical_measure}, 
we introduce a volume form on flows of graphs, together 
with formulas for two operations on graphs which we call  
sieving and contractions in Section \ref{subsec:sieving_contraction}.
Section \ref{sec:three_holed_sphere} is devoted to the proof of 
Theorem \ref{th:Z_g_p_0} in the 
case of the three-holed sphere which corresponds to $(g, p) = (0, 3)$.
We first show that triangular honeycombs can be viewed as flows on 
graphs in Seection \ref{subsec:parametrization_honeycombs}. Section 
\ref{subsec:dual_hive} recalls the \textit{dual hive} model introduced 
in \cite{françois2024positiveformulaproductconjugacy}. 
Sections \ref{subsec:from_dual_hive_to_honey} and \ref{subsec:from_honey_to_dual_hive} 
establish a volume preserving map from dual hives to triangular honeycombs. 
Using the result from \cite{françois2024positiveformulaproductconjugacy} 
which expresses the partition function for the three-holed sphere 
in terms of volume of hives, we prove the volume formula of 
Theorem \ref{th:Z_g_p_0} in terms of triangular honeycombs.
In Section \ref{sec:sieving_honeycombs}, we establish two 
formulas corresponding to Proposition \ref{prop:formula_gluing_honey} 
and Proposition \ref{prop:formula_contracting_honey} which express 
the volume of honeycombs constructed from the operations 
of Section \ref{subsec:sieving_contraction}.
The full proof of Theorem \ref{th:Z_g_p_0} is done in Section 
\ref{sec:proof_of_th_flat_connections}. 
We first give formulas relating the volume of moduli spaces of 
flat connections for surfaces glued together in 
Section \ref{subsec:contraction_formulas_flat_co} before 
proving the theorem in Section \ref{subsec:proof_th_flat_co} 
by relating the previous formulas with the ones of Section 
\ref{sec:sieving_honeycombs}. Finally, we prove Corollary \ref{cor:yang_mills_g_p} 
in Section \ref{sec:proof_positive_area}.

\section{Graphs, divergence of flows and volume measures}
\label{sec:volume_form_flows}

We will see a honeycomb as graph $G$ endowed
with a flow (an antisymmetric function on its
edge set).  The honeycomb
information of $G$, i.e., the precise
drawing on the equilateral triangle of size $1$, 
will be obtained
once we fix the distance of each
edge to the side of the corresponding triangle parallel to it, which are embodied by the map $L$. Is is thus important to give an explicit expression of the volume measure associated to this distances, which is the main objective of this section.

Let us consider a finite graph $G=(V,E)$
and let us denote $\vec E = \{(a,b) \in V\times V \colon
\{a,b\} \in E\}$
its oriented edge set. 
A function $\omega:\vec{E}\to\mathbb{R}$ is called antisymmetric if 
$$\omega(a,b)=-\omega(b,a).$$

Let us denote by $\Omega^0(G)$ 
the vector space of real functions 
on $V$ endowed with the inner product given by
$\langle f_1, f_2 \rangle = \sum_{x \in V} f_1(x) f_2(x)$.
The space of flows is the vector space
$\Omega^1(G)$ of antisymmetric
real functions $\omega$ on $\vec E$ endowed with the inner product
$\langle \omega_1,\omega_2 \rangle
= \sum_{e \in E} \omega_1(e) \omega_2(e) $,
where we are using that, due to the antisymmetry
of $\omega_1$ and $\omega_2$, the product
$\omega_1(e) \omega_2(e)$ does not depend
on the orientating of $e$
as long as we choose the same for both arguments.
\noindent
For any map $\phi: V\rightarrow \mathbb{R}$, set
\begin{equation}
	\label{eq:DivOfLambda}
	\mathcal{F}(\phi) = 
    \left\{\lambda\in\Omega^1(G) \mid \forall v \in V : \sum_{x \sim v} \lambda (v,x) = \phi(v)\right\}.
\end{equation}
We will recall some properties of 
the \emph{divergence} operator
used in the previous equation
and explore the space of solutions
to provide a convenient measure on $\mathcal{F}(\phi)$.

\subsection{Canonical measure}
\label{subsec:canonical_measure}

The analogue of the differential
of a function is the map
$d: \Omega^0(G) \to \Omega^1(G)$,
\[d f (a,b) = f(b) - f(a).\] 
For $x \in V$ define $\delta_{x} \in \Omega^0(G)$ 
that takes the value
 $1$ at $x$ and $0$ at other vertices, and
for $e \in \vec E$ define
$\delta_{e} \in \Omega^1(G)$ that takes the value
$1$ at $e$, $-1$ at the opposite of $e$
and $0$ at other edges. So, we may evaluate
$d$ at the basis
$\{\delta_{x} \}_{x \in V}$ of $\Omega^0(G)$ to obtain
$d \delta_{x} = -\sum_{e \in \vec E,\underline e = x} \delta_{e}$,
where $\underline {(a,b)} = a$. This tells
us that we can write $d$ in terms of 
$\delta_{e}$ and $\delta_{x}$ as
\[d = -\sum_{x \in V} 
\sum_{e \in \vec E, \underline e = x} \delta_{e} \otimes \delta_{x}\]
by identifying
$\Omega^0(G)$ with its dual. Its adjoint $d^*$
would be given by permuting the terms
\[d^* = -\sum_{x \in V} 
\sum_{e \in \vec E, \underline e = x}  \delta_{x} \otimes \delta_{e}\]
or, in a more explicit way, for 
$\omega \in \Omega^1(G)$,
\[d^* \omega (x) = - \sum_{e \in \vec E, \underline e=x}
\omega(e) .\]
We define $\mathrm {div} = -d^*$ so that 
	\eqref{eq:DivOfLambda}
rewrites as
	$\mathrm{div}\lambda = \phi$, yielding to the reformulation of the corresponding affine space 
	 \[\mathcal F(\phi)
	 = 	\{\lambda \in \Omega^1(G):
	  \mathrm{div}\lambda = \phi
	 \}.\]
	 
Let us for now assume that $G$ is connected. Remark first that the solution space may be empty since we
always have, by antisymmetry,
	$\sum_{v \in V} \mathrm{div} \lambda (v)
	= 0$
	which is reminiscent of
	Stokes' theorem. But this is the only
	restriction as explained in Proposition \ref{prop:BijectionTrees}. If
	$\lambda_0 \in \mathcal F(\phi)$ we have
	$\mathcal F(\phi) = 
	\lambda_0 + \mathcal F(0)$ and
	$\mathcal F(0)$ is precisely
	$\mathrm{Ker}(\mathrm{div})$. So, if
	$\mathcal F(\phi)$ is non-empty,
	we can calculate
	\[\mathrm{dim}(\mathcal F(\phi))
	= \mathrm{dim}(\mathrm{Ker}(\mathrm{div}))
	= |E| - \mathrm{dim}(\mathrm{Im}(d))
	= |E| - |V| + 1,\]
	where we used that $\mathrm{Ker}(d) =
	\{\mbox{constant functions}\} $
	 so that
	$ \mathrm{dim}(\mathrm{Im}(d)) = 
	\mathrm{dim}(\Omega^0(G)) - 1 = |V|-1$.

\begin{proposition} \label{prop:BijectionTrees}
	If $G$ is connected, the set
	$\mathcal F(\phi) $ is not empty if and only if
\begin{equation} \label{eq:ExistenceCondition}
\sum_{v \in V} \phi(v)=0.
\end{equation}
	Moreover, if $S \subset E$
	and  $\Omega^1(S)$ denotes the space of antisymmetric
	functions on $\vec S \subset \vec E$, the composition
	\[\varphi_S:\mathcal F(\phi) \xrightarrow[]{
	\mbox{\tiny inclusion}}
	\Omega^1(G) 
	\xrightarrow[]{
		\mbox{\tiny restriction}} \Omega^1(S)\]
	is a bijection if and only if $\mathcal F(\phi)
	\neq \emptyset$
	 and
	$(V,E \setminus S)$ is a spanning tree.
	In case $\varphi_S$ is a bijection, the matrices of $\varphi_S$ and $\varphi_S^{-1}$ in the bases $\{\hat{e},e\in E\},\{\hat{e},e\in S\}$
have integer coefficients.	
\end{proposition}

\begin{proof}[Proof of Proposition
\ref{prop:BijectionTrees}]

We already know that
the condition
	\eqref{eq:ExistenceCondition}
	is necessary.
	Let us show it is sufficient.
	Consider first the case where
	$G=(V,E)$ is a \textbf{tree}. To solve
	$ \mathrm{div} \lambda = \phi$
	we look for
	the edges where $\lambda$
	is most easily determined by the boundary conditions. Let us denote by $\partial V$ the set of leaves of $G$. Then, the divergence condition yields $\lambda(e)=\phi(x)$ if $x\in \partial V$ and $e$ is the unique edge starting from $x$.

	Next, denote by $\interior{V}=V\setminus \partial V$ the vertices of $G$ of degree larger than one. Let us consider edges with an endpoint whose
	other edges connect only to leaves
	or, equivalently, the edges
	in $E\setminus \partial E$ adjacent to leaves
	of $V\setminus \partial V$. Take
	one such leaf $x \in \interior{V}$
	and notice that,
	since $x$ is not a leaf of $G$, the set $V_{1,x}$
	of vertices in $\partial V$ connected to
	$x$ is non-empty.  
	Let us reduce our task to finding
	a solution on $V \setminus V_{1,x}$ as follows.  If $e$ denotes
	the edge of $E \setminus \partial E$
	starting at $x$, using the previous case of $x$ being a leaf yields that a solution should satisfy
	\[ \phi(x)=\bigg(\sum_{y \in V_{1,x}}\lambda(x,y)
	\bigg) +\lambda(e) 
	=-\bigg(\sum_{y \in V_{1,x}}\lambda(y,x)
	\bigg) +\lambda(e)=-\bigg(\sum_{y \in V_{1,x}}\phi(y)
	\bigg) +\lambda(e)\]
	which determines $\lambda$
	at $e$ as linear combinations of different values $\phi$ with integer coefficients. Now, we consider the graph restricted to
	$V\setminus V_{1,x} $, and define
	$\phi(x) =\left(\phi(x)+
	\sum_{y \in V_{1,x}}\phi(y)\right) 
	$
	and notice that
	\eqref{eq:ExistenceCondition}
	is satisfied for $V\setminus V_{1,x}$.
	We may proceed by induction
	until the tree consists solely of leaves
	in which case
	the condition
	\eqref{eq:ExistenceCondition}
	is precisely the
	equation 	 
	$\mathrm{div}\lambda = \phi$. By this procedure
	we have seen that, for a tree, 
	in case there is a solution, it is unique, and its value at $e\in E$ is a linear combination with integer coefficients of the values $\phi(x)$ for $x\in V$.
	
	For a \textbf{general graph} $G$ we
	may take a spanning tree $T$ of $G$
	and try to solve
	the equation for $T$. 
	We may choose
	the values of $\lambda$
	arbitrarily
	at the edges that are not in $T$. More precisely,
	if $S$ is the set of edges not belonging to $T$,
	we may consider any antisymmetric function
	$\widetilde \lambda: \vec S \to \mathbb R $
	and look for a solution
	$\lambda \in \mathcal F(\phi)$
	satisfying $\lambda|_{\vec S} = \widetilde \lambda$.
	To 
	be able to forget the edges in $S$ we change $\phi$ to
	$\phi_{new}:V \to \mathbb R$ given by
		\[\phi_{\tiny new}(x) =\phi(x) -
	\sum_{e \in \vec S, \underline{e} = x}
	\widetilde \lambda(e)\]
	so that if the divergence of $\lambda$
	in $T$ at $x$ is
	$\phi_{new}(x)$, the divergence of the extension
	by $\widetilde \lambda$
	of $\lambda$
	in $G$ at $x$ would be 
	$\phi(x)$.
	Notice that the sum of $\phi_{new}(x)$
	for $x \in V$ is the same as the sum of 
	$\phi(x)$.
	This holds because each edge in $S$ is appears twice in the sum,
	once with each orientation, and thus the contribution to the sum cancels due to the antisymmetry.
	Then 
	\eqref{eq:ExistenceCondition} holds
	for $T$ and $(\phi_{new};\gamma_{new})$, and we may
	find a unique solution $\lambda \in \mathcal F(\phi_{new})$ on $T$ such that $\lambda(x)$ is a linear combination with integer coefficients of the values of $\phi_{new}$, and thus also of the values of $\phi$ and $\tilde{\lambda}$.	We extend $\lambda$ by $\widetilde 
	\lambda$ and remark that it is a solution in
	$\mathcal F(\phi)$. 
	
	On the other hand,
	if $\lambda$ is a solution in $\mathcal F(\phi)$
	its restriction to the directed edge set of $T$
	is a solution in $\mathcal F(\phi_{new})$
	so that there is only 
	one solution $\lambda$ in $\mathcal F(\phi)$
	satisfying $\lambda|_{\vec S} = \widetilde \lambda$.
	Notice that this already
	shows that $\varphi_S$ is a bijection
	when $(V,E\setminus S)$ is a spanning tree of $G$.
	\\
	\\
	It is clear that if $\varphi_S$ is a bijection
	then $\mathcal F(\phi)$ is not empty
	because $\Omega^1(S)$ would be empty which
	is impossible. So, for
	the rest of the proof we may assume that
	$\mathcal F(\phi)$ is not empty.
	By taking any $\lambda_0 \in \mathcal F(\phi)$, using that
	$\mathcal F(\phi)
	= \mathcal F(0) $ 
	and that
	the map
	$\mathcal F(\phi) \to \Omega^1(S)$
	is a composition
	\[\mathcal F(\phi)
	\xrightarrow[]{\tiny \mbox{translation by } - \lambda_0}
	\mathcal F(0)
		\xrightarrow[]{\varphi_S}
	\Omega^1(S) 
		\xrightarrow[]{\tiny \mbox{translation by }  
			\lambda_0|_{\vec S}}
	\Omega^1(S),\]
	we find
	that it is enough to prove the statements
	for $\phi=0$.
	Now, let us show that if $\varphi_S$ is a
	bijection then $(V,E\setminus S)$ is a spanning tree.
	If $(V,E\setminus S)$ had a cycle
	$(v_1,\dots,v_k, v_{k+1})$ with $v_{k+1} = v_1$ we could define
	$\lambda(v_i,v_{i+1}) = - \lambda(v_{i+1},v_i) =1$
	for $i \in \{1,\dots,k\}$ and zero elsewhere.
	Such $\lambda$ would belong to $\mathcal F(0)$
	but its image in $\Omega^1(S)$ would be zero
	so that $\varphi_S$
	would not be injective. To show that
	$(V,E\setminus S)$ is connected we may notice that if
	it were a disconnected forest we could find
	$S' \subset S$ such that
	$(V,E\setminus S')$ is a spanning tree. But this
	would imply that the values at $S'$ determine
	the solution and, thus, determine the
	values at $S \setminus S'$. Then $\varphi_{S}$
	could not be surjective. 
	
\end{proof}
\begin{proposition}
    \label{prop:divergence_volume}
Suppose that $G$ is connected. If $(V,E\setminus S)$ is a spanning tree, then for all $\phi:V\to\mathbb{R}$ satisfying \eqref{eq:ExistenceCondition},
\[(\varphi_S)_* \mathcal Leb_{\mathcal F(\phi)}	= \big(\hspace{-0.5mm} \sqrt{\# \mbox{ Spanning trees of } G}	\big)\,
	\mathcal Leb_{\Omega^1(S)}.\]
\end{proposition}
\begin{proof}
Now, assuming that $(V,E\setminus S)$
	is a spanning tree, let us look 
	for the constant $C>0$ such that
	\[(\varphi_S)_* \mathcal Leb_{\mathcal F(0)}	= C	\mathcal Leb_{\Omega^1(S)}\]
	Denoting the dimension of
	$\mathcal F(0)$ by $k = |E|-|V| +1$,
	we want to study the map $(\varphi_S)_*$ induced
	on $k$-forms. This is equivalent to looking at
	the pushforward map on $k$-vectors
	\[(\varphi_S)_*: \Lambda^k
	\mathcal F(0) \to 
	\Lambda^k \Omega^1(S),\]
	taking the dual and inverting the resulting map.
	The constant $C>0$ would be found
	by taking normalized
	vectors $w_1 \in \Lambda^k
	\mathcal F(0) $, $w_2 \in \Lambda^k \Omega^1(S)$ and 
	solving
	\[(\varphi_S)_* w_1 = \pm C^{-1} w_2\]
	or, what is the same,
	taking $C^{-1} = |\langle (\varphi_S)_* w_1 , w_2\rangle| $.
	The vector $w_2$ can be explicitly obtained
	as $\wedge_{e \in S} \delta_{e}$, where
	we have
	chosen an orientation for each edge $e \in E$ 
	and an order to perform the product.
	We recall that
	 $\delta_{e}$
	takes the value $1$ at $e$ with our chosen orientation,
	$-1$ if we reverse the orientation and
	$0$ at all other edges. To 
	explicitly construct $w_1$ is less obvious since
	we would be dealing with
	$\mathrm{Ker}(\mathrm{div})$. We
	could instead use its orthogonal complement
	$\mathrm{Im}(d)$.
	By fixing a leaf $v_0$
	of $(V,E\setminus S)$ we may consider
	$\{d \delta_{v} \}_{v \in V \setminus \{v_0\}}$,
	where we recall that
	$\delta_{v}$ is the function which is $1$ at $v$
	and $0$ elsewhere.
	Since $\sum_{v \in V \setminus \{v_0\}}
	d\delta_{v} + d\delta_{v_0} = d 1= 0$, this family
	generates $\mathrm{Im}(d)$. Now,
	we obtain an element of
	$\Lambda^k \mathcal F(0;0)$ by taking 
	$*(\wedge_{v \in V\setminus \{v_0\}} d \delta_{v})$,
	where $*$ denotes the Hodge star operator.
	We have not yet normalized this element nor shown it is non-zero but let us calculate,
	using that 
	$ \wedge_{e \in S} \delta_{e} = \pm 
	* (\wedge_{e \in E\setminus S} \delta_{e})  $,
	\[
	\big<  \hspace{-0.5mm} *(\wedge_{v \in V\setminus \{v_0\}} d \delta_{v}),
	 \wedge_{e \in S} \delta_{e} \, \big>
	 = \pm 
	 \big< \hspace{-0.5mm}  \wedge_{v \in V\setminus \{v_0\}} d \delta_{v},
	 \wedge_{e \in E\setminus S} \delta_{e} \, \big>
	 = \pm\mathrm{det}\langle d \delta_{v}, \delta_{e} \rangle
	 _{v \in V\setminus \{v_0\}, e \in E\setminus S}.
	 \]
	This determinant is a sum
	over all bijections $\sigma:E \setminus S \to V \setminus \{v_0\} $ of
	the alternating product $(-1)^\sigma \prod_{e \in E\setminus S} 
	\langle d \widehat {\sigma(e)}, \delta_{e} \rangle$,
	where the sign $(-1)^\sigma$ is only defined
	up to an overall sign. Notice that
	$\langle d \delta_{x}, \delta_{e} \rangle = \pm 1$ if and only
	if $x$ is an endpoint of $e$ and, if not,
	$\langle d \delta_{x}, \delta_{e} \rangle = 0$.
	So, for a bijection to contribute,
	the unique edge adjacent to $v_0$ should correspond
	to the unique other endpoint $v_1$ of
	this edge. Then, for every other edge adjacent to $v_1$
	we do not have a choice but to take the endpoint
	different from $v_1$. If we continue in this way
	we get that there is only
	one bijection
	that contributes and therefore
	\[
	\big< \hspace{-0.5mm}  *(\wedge_{v \in V\setminus \{v_0\}} d \delta_{v}),
	\wedge_{e \in S} \delta_{e} \, \big>
	= \pm 1.\]
	This proves in particular that 
	$*(\wedge_{v \in V\setminus \{v_0\}} d \delta_{v})$ is not
	zero so the explicit formula for $w_1$
	as the normalized $*(\wedge_{v \in V\setminus \{v_0\}} d \delta_{v})$ works.
	If $(V,E\setminus S)$ is not a spanning tree of $G$
	we already know that $\varphi_S$
	is not a bijection so that the previous inner product
	is zero. Using that
	$\{\wedge_{e \in S} \delta_{e} \}_{|S| = k}$ forms an
	orthonormal basis of
	$\Omega^1(G)$ we can calculate the norm
	as the sum of squares of inner products
	to obtain
	\[\|*(\wedge_{v \in V\setminus \{v_0\}} d \delta_{v})\|^2
	= 
	\mbox{\# Spanning trees of } G.
	\]
	This yields that 
	\[\langle (\varphi_S)_* w_1, w_2 \rangle
	=\pm \frac{1}{\sqrt{\mbox{\# Spanning trees of } G}}\]
	which implies the final statement of
	the proposition.
\end{proof}

\noindent
In the following, be set
\begin{equation} 
    \label{eq:definition_volume_form}
d\Vol=\frac{1}{\sqrt{\# \mbox{ Spanning trees of } G}} \mathcal Leb_{\mathcal F(\phi)} \ ,
\end{equation}
so that 
\begin{equation}
    \label{eq:pushforward_phi_S}
    (\varphi_S)_{*}\diff \Vol=\mathcal \Leb_{\Omega^1(S)} \ .
\end{equation}

\subsection{Boundary and sieving of graphs}
\label{subsec:sieving_contraction}

For $G=(V,E)$ a finite graph, recall that we denote by $\partial V$ (resp. $\interior(V)$) 
the set vertices of $G$ of degree $1$ (resp. degree larger than $1$) and by 
$\partial E$ the set of edges adjacent to $\partial V$. If $\phi\in \Omega ^0(G)$, 
we denote by $\partial \phi$ the restriction of $\phi$ to $\partial V$. 
Likewise, we denote by $\partial \lambda$ the restriction of $\lambda\in \Omega^1(G)$ 
to $\{(a,b),\{a,b\}\in\partial E\}$. By the previous section, 
for $R\subset  V$, $\mathcal{F}(\phi) \neq \emptyset$ if and only if 
$$\sum_{v\in R}\phi(v)=-\sum_{v\in V\setminus R}\phi(v) \ .$$

\begin{definition}[Sieving of graphs]\label{def:sieving}
Let $r\geqslant 1$. Let $G=(V,E)$ be a finite graph 
(non necessarily connected) and $W,W'\subset \partial V$ 
with a bijection $g:W\rightarrow W'$. 
The sieving of $G$ along $(R_1,R_2)$ is the graph 
$G_{W*W'}=(\tilde{V},\tilde{E})$ obtained as follows:
\begin{itemize}
\item $\tilde{V}$ is the quotient of $V$ by the equivalence relation generated by $(v,g(v))$ for $v\in W$,
\item $\tilde{E}=\pi(E)$, where $\pi:V\times V\rightarrow \tilde{V}\times \tilde{V}$ is the quotient map.
\end{itemize}
\end{definition}

\noindent
This construction informally amounts to merge $v$ and $g(v)$ for 
$v\in W$ and considering the resulting edge structure inferred by $E$. 
Since $W,W'\subset \partial V$, $\vert \tilde{E}\vert=\vert E\vert$ as long 
as each connected component of $G$ has a size at least $3$ 
(which will always by the case in the present paper). 
There is moreover a canonical bijection between $E$ and $\tilde{E}$. 
Let us closely look at the behavior of the equation \eqref{eq:DivOfLambda} 
with respect to the sieving of graphs. 
\\
In the following, if $G=(V,E)$ is a finite graph, $V=S_1\sqcup\ldots \sqcup S_r$ 
is a partition of $V$ and $\phi\in \Omega^0(G)$, we write by abuse of notation 
$(\phi_{\vert S_1},\ldots,\phi_{\vert S_r})$ instead of $\phi$ to detail the 
decomposition of $\phi$ along this partition. Moreover, we write $\mathcal{F}_{G}(\phi)$ 
instead of $\mathcal{F}(\phi)$ to emphasize that the equation \eqref{eq:DivOfLambda} 
is considered in the graph $G$.

\begin{proposition}[Product formula]\label{prop:divergence_product}
Let $G_1=(V_1,E_1),G_2=(V_2,E_2)$ be two connected finite graphs, 
such that $V_1 \cap V_2 = \emptyset$. Let
$W\subset \partial V_1, W'\subset \partial V_2$ and let
$g:W\rightarrow W'$ be a bijection.
Set $\tilde{G}=(G_1\cup G_2)_{W*W'}$ and 
let $\phi\in \Omega^0(\tilde{G})$ be such that 
$\sum_{v\in \tilde{V}}\phi(v)=0$. 
Let $\lambda\in \Omega^1(\tilde{G})$ and for $w \in W$, let us 
denote by $x_w = \lambda(\tilde{e})$ where $\tilde{e}$ 
is the unique edge of $\tilde{G}$ starting from $\tilde{w}$ 
and ending on $V_1$.
Then, 
\begin{enumerate}
    \item $\lambda\in\mathcal{F}_{\tilde{G}}(\phi)$ if and only if 
\begin{itemize}
    \item[(a)] $\sum_{v\in V_1\setminus W}\phi(v)+\sum_{w\in W} x_w=0$,
    \item[(b)] $\lambda_{\vert E_1}\in 
    \mathcal{F}_{G_1}(\phi_1)$ where 
    $\phi_1 = \left( \phi_{\vert V_1 \setminus W}, 
    \phi_W \right)$ and where $\phi_W(w) = x_w$ for $w \in W$,
    \item[(c)] $\lambda_{\vert E_2}\in 
    \mathcal{F}_{G_2}(\phi_2)$ where $\phi_2 = \left( \phi_{\vert V_2\setminus W'}, 
    \phi_W' \right)$ and where $\phi_W'(w') = \phi(w') - \phi_W(g^{-1}(w'))$ for $w' \in W'$.
\end{itemize} 
\item For $w_0\in W$, $K_1 \subset \Omega^1(G_1) $ and $
K_2\subset \Omega^{1}(G_2)$,
\begin{align*}
\Vol_{\tilde{G}}((K_1\times K_2)\cap\mathcal{F}_{\tilde{G}}(\phi)) = 
&\int_{\mathbb{R}^{\vert W\vert-1}}\Vol_{G_1}\left[K_1\cap \mathcal{F}_{G_1}(\phi_{\vert V_1},(x_w)_{w\in W\setminus W_0},y(x)))\right]\\
&\hspace{1.1cm}\cdot \Vol_{G_2}\left[K_2\cap \mathcal{F}_{G_2}(\phi_{\vert V_2},(\phi(\tilde{w})-x_w)_{w\in W\setminus \{w_0\}}),\phi(\tilde{w})-y(w))\right]dx,
\end{align*}
where $y(w)$ is the unique solution to 
$\sum_{v\in V_1\setminus W} \phi(v) + 
\sum_{w\in W\setminus \{w_0\}}x_w+y(w)=0$.
\end{enumerate}

\end{proposition}
\begin{proof}

(1) Recall that $\lambda\in\mathcal{F}_{\tilde{G}}(\phi)$ if and only if
$$\forall x \in \tilde{V}: \sum_{e,\bar{e}=x}\lambda(e)=\phi(x) \ .$$

\noindent
Let $\lambda\in \mathcal{F}_{\tilde{G}}(\phi)$ and let $x \in V_1$. 
If $x \in V_1 \setminus W$, then 
\begin{equation*}
     \sum_{v \in V_1 : \{v, x\} \in E_1} \lambda_{\vert E_1}(v, x) 
     = \sum_{v \in \tilde{V} : \{v, x\} \in \tilde{E}} \indic_{\{v, x\} \in E_1} \lambda(v, x) 
    = \sum_{v \in \tilde{V} : \{v, x\} \in \tilde{E}} \lambda(v, x) = \phi(x) \ .
\end{equation*}
If $w \in W$, then $ \sum_{v \in V_1 : \{v, w\} \in E_1} \lambda_{\vert E_1}(v, w) = x_w$ 
as required, which gives that $\lambda_{\vert E_1}\in 
\mathcal{F}_{G_1}(\phi_{V_1\setminus W},\phi_W)$. \\
Using the divergence condition for $\phi$, one must have 
$$\sum_{v\in V_1}\phi(v)+\sum_{w\in W}x_w=0 \ .$$
Similarly, $\lambda_{\vert E_2}\in \mathcal{F}_{G_2}(\phi_{V_2\setminus W'},(y_w)_{w\in W'})$ 
with $y_w=\lambda(e)$, where $e$ is the unique edge of $\tilde{E}$ 
starting from $w$ and ending on $V_2$. 
By using the divergence condition on $\tilde{w}=\{w,g(w)\}$ for $w\in W$, 
one has $\phi(\tilde{w})=x_w+y_{g(w)}$, and thus 
$y_{g(w)}=\phi(\tilde{w})-x_w \ .$
\\
\\
Reciprocally, one checks that if $\lambda\in \Omega^1(\tilde{G})$ is such that 
$\lambda_{\vert E_1}\in \mathcal{F}_{G_1}(\phi_{V_1\setminus W},(x_w)_{w\in W})$ and 
$\lambda_{\vert E_2}\in \mathcal{F}_{G_2}(\phi_{V_2\setminus W'},(\phi(\tilde{w})-x_{g^{-1}(w)})_{w\in W})$, 
then $\lambda\in \mathcal{F}(\phi)$.
\\
\\
(2) Set $s+1=\vert W\vert=\vert W'\vert$ and write $W=\{w_0,\ldots,w_s\}$. 
Let $e_i$ be the edge of $E_1$ adjacent to $w_i$. 
Let $T_i$ be a spanning tree of $G_i$ for $i\in \{1,2\}$. 
Since each edge $\tilde{w},w\in W$ is bivalent in $\tilde{G}$, 
removing the edges $e_j$, $1\leqslant j\leqslant s$ from $T_1\cup T_2$ yields a spanning 
tree $T$ of $\tilde{G}$. 
By Proposition \ref{prop:BijectionTrees} applied to $S=T^c$, the restriction map 
$\varphi_{S}$ is a bijection from $\mathcal{F}_{\tilde{G}}(\phi)$ to $\Omega^1(S)$. 
Moreover, by Proposition \ref{prop:divergence_volume}, 
$(\varphi_S)_* \diff\Vol=\mathcal Leb_{\Omega^1(S)}$. 
Let us write $S_i=E_{i}\setminus T_i$ and $R=\{e_1,\ldots,e_{s}\}$. Then, 
$$S = T^c=[E_{1}\setminus  T_1]\cup [E_2\setminus T_2]\cup R=S_1\cup S_2\cup R \ .$$

\noindent
Let $K_1\subset \mathbb{R}^{E_1}$ and $K_2\subset \mathbb{R}^{E_2}$. 
Then, for $(t_1,t_2,x)\in\mathbb{R}^{S_1}\times \mathbb{R}^{S_2}\times \mathbb{R}^{s-1}$, 
$$\lambda = \lambda(t_1, t_2, x) = \varphi_{S}^{-1}(t_1,t_2,x)\in\mathcal{F}_{\tilde{G}}(\phi) \ .$$ 
By the previous statement, this is equivalent to the fact that 
$\lambda_{\vert E_1} \in \mathcal{F}_{G_1}(\phi_{\vert V_1\setminus W},(x_w)_{w\in R},y)$ 
where $y$ is the unique solution to $\sum_{v\in V_1\setminus R}\phi(v)+\sum_{w\in R}x_w+y=0$, 
and $\lambda_{\vert E_2} \in \mathcal{F}_{G_2}(\phi_{\vert V_2\setminus W'},
(\phi(\tilde{w})-x_{w}))_{w\in R},\phi(w_0)-y)$. 
Let $\varphi^x_{S_1}$ be the projection from $\mathcal{F}_{G_1}(\phi_{\vert V_1\setminus W},
(x_w)_{w\in R},y))$ to $\Omega^1(S_1)$ and $\varphi^x_{S_2}$ be the 
projection from $\mathcal{F}_{G_2}(\phi_{\vert V_2\setminus W'},
(\phi(\tilde{w})-x_{w}))_{w\in R},\phi(w_0)-y)$ to $\Omega^1(S_2)$. 
Since $E_i\setminus S_i$ is a spanning tree of $G_i$, each map 
$\varphi^x_{S_i}$ is bijective. 
Moreover, since then $\varphi^x_{S_1}\circ \lambda(t_1,t_2,x)$ 
is well-defined and equal to $t_1$, we have 
$$\lambda_{\vert E_1} = (\varphi^x_{S_1})^{-1}(t_1) \text{ and, likewise, } 
\lambda_{\vert E_2} =  (\varphi^x_{S_2})^{-1}(t_2) \ .$$
Therefore,
$\lambda(t_1, t_2, x) \in K_1\times K_2\cap \mathcal{F}(\phi)$ if and only if 
$(\varphi^x_{S_1})^{-1}(t_1)\in K_1$ and $(\varphi^x_{S_1})^{-1}(t_2)\in K_2$. 
Hence,
\begin{align*}
Vol(K_1\times K_2\cap\mathcal{F}(\phi)) =& 
\int_{\mathbb{R}^{S_1}\times \mathbb{R}^{S_2}\times \mathbb{R}^{s}} 
\indic_{\lambda(t_1,t_2,x)\in K_1\times K_2} \diff t_1 \diff t_2 \diff x \\
=&\int_{\mathbb{R}^s}\left(\int_{\mathbb{R}^{S_1}\times \mathbb{R}^{S_2}} 
\indic_{(\varphi_{S_1}^x)^{-1}(t_1)\in K_1} 
\indic_{(\varphi_{S_2}^x)^{-1}(t_2)\in K_2} \diff t_1 \diff t_2\right) \diff x\\
=&\int_{\mathbb{R}^{s}}\left(\int_{\mathbb{R}^{S_1}} 
\indic_{(\varphi_{S_1}^x)^{-1}(t_1)\in K_1} \diff t_1\right) \cdot 
\left( \int_{\mathbb{R}^{S_2}}\indic_{(\varphi_{S_2}^x)^{-1}(t_2)\in K_2}\diff t_2\right) \diff x\\
=&\int_{\mathbb{R}^{s}}\Vol\left[K_1\cap \mathcal{F}_{G_1}(\phi_{\vert V_1\setminus W},(x_w)_{w\in R},y) \right] \\
&\hspace{1cm}\cdot \Vol\left[K_2\cap\mathcal{F}(\phi_{\vert V_2\setminus W'}, 
(\phi(\tilde{w})-x_{w}))_{w\in R},\phi(w_0)-y)\right] \diff x \ .
\end{align*}

\end{proof}

\begin{proposition}[Contraction formula]\label{prop:divergence_contraction}
Let $G=(V,E)$ be a connected finite graph, $W,W'\subset \partial V$ with $W\cap W'=\emptyset$ and $g:W\rightarrow W'$ a bijection and set $G_{W*W'}=(\tilde{V},\tilde{E})$. Then, for $\phi\in \Omega^0(G_{W*W'})$ such that $\sum_{v\in \tilde{V}}\phi(v)=0$ and $K\subset \Omega^1(\tilde{G})$,
\begin{align*}
Vol(K\cap\mathcal{F}_{\tilde{G}}(\phi))= &\int_{\mathbb{R}^{\vert W\vert-1}} 
\Vol\left[K\cap \mathcal{F}_{G}(\phi_{\vert V\setminus (W\cup W')}, 
(x_w)_{w\in W},(\phi(\tilde{w})-x_{g^{-1}(w')})_{w\in W'})\right] \diff x \ .
\end{align*}
\end{proposition}
\begin{proof}
The proof is similar to the one of Proposition \ref{prop:divergence_product}.
\end{proof}

\section{The three-holed sphere}
\label{sec:three_holed_sphere}

The goal of this section is to establish Theorem \ref{th:Z_g_p_0} 
in the case where $(g, p) = (0, 3)$, that is, for the three-holed sphere. 
Section \ref{subsec:parametrization_honeycombs} gives an 
injection of honeycombs into flows which is 
Proposition \ref{prop:parametrization_triangular_honey_first}. 
The latter thus gives a parametrization of triangular honeycombs. 
In Section \ref{subsec:dual_hive}, we recall a combinatorial model from 
\cite{françois2024positiveformulaproductconjugacy} called \textit{dual hive}. 
Section \ref{subsec:from_dual_hive_to_honey} and Section \ref{subsec:from_honey_to_dual_hive} 
show that there is an linear bijection with integer coefficients between dual hives 
and triangular honeycombs. 
We finally prove Theorem \ref{th:Z_g_p_0} 
for the three-holed sphere in Section \ref{subsec:main_th_three_holed_sphere}.

\subsection{Parametrization of triangular honeycombs}
\label{subsec:parametrization_honeycombs}

The goal of this section is to view triangular honeycombs 
as flows on a graph with prescribed divergence. 
Proposition \ref{prop:divergence_volume} then yields a volume form on this set of 
flows. Recall that for $d \geqslant 0$, $\mathcal{G}_d$ denotes 
the set of isomorphism classes of colored 
graphs with ordered boundary appearing in $\{G(h),\,h \in \honey_{n,d}\}$.
For a honeycomb $h\in \honey^G$ and an edge $e \in E$, let us denote by 
$\ell^h(e)$ the type of $e$, defined in Definition \ref{def:triangular_honey}, (2).

\begin{lemma}[Boundary determine type and colors]
    \label{lem:types_colors_determined_by_boundaries}
    Let $G\in\mathcal{G}_d$ and let $h\in \honey^G$. 
    Then, the type $\ell^h:E\rightarrow \{0,1,2\}$ 
    and color $c^h:E\rightarrow \{0,1,3\}$ are independent of $h$.
\end{lemma}

\begin{proof}
    Remark that such label and color maps can be defined similarly for any honeycomb $h=(\mathcal{E},c)$ of $S\subset T$ such 
    that any segment $e\in \mathcal{E}$ is contained in some ray $\{x+\mathbb{R}e^{2(\ell+1)i \pi/3}\}$ 
    for some $\ell\in\{0,1,2\}$. 
    Let us call such a honeycomb \textit{admissible} and let us prove by 
    induction on the number $M$ 
    of inner vertices (that is, the number of vertices of $G(h)$ which are in the interior of $T$) the following : 
    \textit{for any admissible honeycomb $h$, the induced label and color map on $G(h)$ 
    only depends on the type and color map of the boundary edges and on the 
    order on the boundary vertices}.
    \\
    \\
    If $M=1$, then all edges of $G[h]$ are boundary edges, and the assertion holds. 
    Let $M>1$ and let $h$ be a honeycomb such that $G(h)=(V,E)$ has $M$ inner vertices. Suppose that there is $v\in \interior (V)$ which is adjacent to two boundary edges 
    $e_1=\{v,v_1\},e_2=\{v,v_2\}$ and one non boundary edge $e$. 
    Then, the type $e$ is uniquely determined by the relation 
    $\{\ell(e),\ell(e_1),\ell(e_2)\}=\{0,1,2\}$ 
    and the value of $\ell(e_1)$ and $\ell(e_2)$. 
    Next, since the cyclic order of the boundary vertices is given, 
    by Definition \ref{def:triangular_honey} the color $c(e)$ is uniquely determined 
    by $c(e_1)$ and $c(e_2)$. Hence, on $\tilde{G}=(V\setminus \{v_1,v_2\}, 
    E\setminus \{e_1,e_2\})$ 
    the type and color of the boundary edges is known, and by induction, 
    the type and colors of all edges of $\tilde{G}$ only depends on the graph 
    structure and their value on the boundary. 
    \\
    \\
    Suppose that all vertices of $G$ are adjacent to at most one boundary edge. 
    Let $(v_1,\ldots, v_m)$, $m\geqslant 1$ be the boundary vertices in the cyclic order. 
    By hypothesis, there exist $(w_1,\ldots,w_m)$ such that 
    $\partial E=\{ e^i := \left\{v_i,w_i\right\} ,1\leqslant i\leqslant m\}$ 
    and $w_i\not =w_j$ when $i\not=j$. We claim that there exists $w_{i},w_{i'}$ such that 
    $\{w_{i},w_{i'}\}\in E$ and $\{\ell(\{w_{i},w_{i'}\}),\ell(e_{i}),\ell(e_{i'})\}=\{0,1,2\}$. 
    Let $\tilde{G}=(V\setminus V_{1},E\setminus \partial E)$ and $\tilde{h}=(\mathcal{E}\setminus \partial \mathcal{E},c_{\vert \mathcal{E}\setminus \partial \mathcal{E}})$ where $\partial \mathcal{E}$ are the segments corresponding to edges of $\partial E$. 
    Then, $t(\tilde{h})=\bigcup_{e\in \mathcal{E}\setminus \partial \mathcal{E}}e \subset \overset{\circ}{T}$, and thus 
    there exists a unique connected component $K_0$ in $T\setminus t(\tilde{h})$ which is adjacent to $\partial T$. 
    Let $L$ be a boundary component of $K_0$ in $\overset{\circ}{T}$. Then, $L$ is a close polygonal line with vertices 
    $\{z_1,\ldots,z_p\}$ enumerated in the cyclic order. 
    At each $z_i$, $L$ has an angle $\alpha_i$ so that $\alpha_i=4\pi/3$ if $z_i\in \{w_1,\ldots,w_n\}$.

    Since $L$ is a close polygonal curve, there are at least two consecutive vertices 
    $z_i,z_{i+1}$  such that $\alpha_i=\alpha_{i+1}=4\pi/3$, and thus $z_i=w_{j_i}$ and $z_{i+1}=w_{j_{i+1}}$ for some $1\leqslant i_j\not=i_{j+1}\leqslant m$. 
    Moreover, $[w_{j_i},w_{j_{i+1}}]$ is a segment, and thus $\{w_{j_i},w_{j_{i+1}}\}\in E$. 
    Since the angle from $\{v_{j_i},w_{j_i}\}$ (resp. $\{v_{j_{i+1}},w_{j_{i+1}}\}$) to 
    $\{w_{j_i},w_{j_{i+1}}\}$ 
    is $-2\pi/3$ (resp. $2\pi/3$),
    $$\left\{ \ell(\{w_{j_i},w_{j_{i+1}}\}), \ 
    \ell(\{v_{j_i},w_{j_i}\}), \ 
    \ell(\{v_{j_{i+1}},w_{j_{i+1}}\})\right\} = \{0,1,2\}.$$

    \noindent
    Let $w_{i},w_{i'}$ be such that $e:=\{w_{i},w_{i'}\}\in E$, $e^1:=\{w_{i},w_{i'}\}$ 
    and $e^2:=\{v_i,w_i\}$ satisfy the angle condition $\{\ell(e^1),\ell(e^2),\ell(e^3)\} 
    =\{0,1,2\}$. 
    Then, $\ell(e^3)$ is determined by $\ell(e^1)$ and $\ell(e^2)$. 
    Let $f^1,f^2$ be the third edge around $w_i$ (resp. $w_{i'}$). 
    Then, $\ell(f^i)$ is determined by $\ell(e^i)$ and $\ell(e^i)$ for $i \in \{1,2\}$. 
    By the color condition from Definition \ref{def:toric_honeycomb}, 
    $c(e)=3$ if $c(e^1)\not=c(e^2)$ and otherwise $c(e)=c(e^1)=c(e^2)$. 
    Then, $c(f^1)$ and $c(f^2)$ are uniquely determined by $\{c(e), c(e^2),c(e^2)\}$. 
    Let $\hat{G} = \left(V\setminus \{v_i,v_{i'}\},E\setminus \{e,e^1,e^2\}\right)$. 
    Then, $\hat{h}=(\mathcal{E}\setminus \{e, e^1, e^2\},c_{\vert \mathcal{E}\setminus \{e, e^1, e^2\}})$ is a honeycomb such that
    $G(\hat{h}) = \hat{G}$ has $M-1$ inner vertices, and such that the type and color 
    of the boundary edges are known. By induction, the type and color of all edges of $\hat{G}$, 
    and thus of $G$ are known. 
\end{proof}

\noindent
Let us provide a description of honeycombs with structure graph $G$ in terms of flows. 
Suppose that $G=(V,E)\in \mathcal{G}_d$. 
    By the condition (1) Definition \ref{def:triangular_honey}, 
    $G$ has only vertices of degree $1$ or $3$ and thus $v$ has three adjacent 
    edges $e_\ell,\, \ell\in \{0,1,2\}$. 
    Denote by $\interior(V)$ the set of vertices of degrees $3$ and let $v\in\interior(V)$. 
    By (2) of Definition \ref{def:toric_honeycomb}, 
    the angle between two successive edges at $v$ is $2\pi/3$ and by (2) of 
    Definition \ref{def:triangular_honey}, 
    each edge is oriented along $e^{2(\ell+1)i\pi/3}$ for some $\ell\in\{0,1,2\}$. 
    Hence, there exists a sign $s(v)\in\{-1,+1\}$ such that, up to a relabeling, 
    $e_\ell \subset \left\{v-s(v)\ed^{2(\ell+1)i\pi/3} \mathbb{R}_{\geqslant 0} \right\}$. 
    For a univalent vertex $v\in\partial V$ connected to a unique trivalent vertex 
    $v' \in \interior(V)$, we set $s(v) = -s(v')$.
    \\
    \\
    Let $v,v'\in V$ be such that $e = \{v,v'\} \in E$. 
    By the previous reasoning, there exist $\ell,\ell'$ such that 
    $e\subset \left\{v-s(v)\ed^{2(\ell+1) i\pi/3} \mathbb{R}_{\geqslant 0} \right\}$ and 
    $e\subset \left\{v'-s(v')\ed^{2(\ell'+1)i\pi/3} \mathbb{R}_{\geqslant 0} \right\}$.  
    Necessarily, $\ell=\ell'$ and $s(v)=-s(v')$, see Figure \ref{fig:adjacent_vertices}.

    \begin{figure}[ht]
        \centering
        \includegraphics[scale=1]{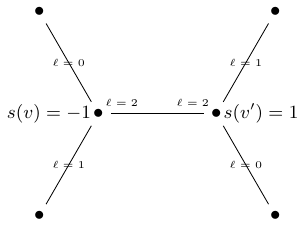}
        \caption{Two adjacent vertices in $\interior(V)$ for which $\ell = \ell'=2$.}
        \label{fig:adjacent_vertices}
    \end{figure}
    \noindent
    In particular, the map $s:v\mapsto s(v)$ only depends on the value of $s$ on the 
    boundary $\partial V$. Since $s(v)$ for $v\in \partial V$ is given by the 
    type and thus the color of the unique adjacent edge, 
    see Definition \ref{def:triangular_honey}.(3), 
    we deduce by Lemma \ref{lem:types_colors_determined_by_boundaries} 
    that $s$ only depends on $G$. 
    In the following definition, recall the definition of the height 
    of an edge in Definition \ref{def:triangular_honey}.
    
 \begin{definition}[Flow of a honeycomb]
    \label{def:flow_honeycomb}
    The \textit{flow} of a honeycomb 
    $h\in\honey_{n,d}^G$ is the map $\mathcal{L}[h] \in \Omega^1(G)$ 
    which to an oriented edge $(v,v') \in \vec E$ associates $\mathcal{L}[h](v,v') = s(v) L(\{v,v'\})$.
\end{definition}

\noindent
If $v\in \interior(V)$, $v$ corresponds to a point if $T$ and thus 
its coordinates $(v_1,v_2,v_3)$ satisfies $v_1+v_2+v_3=1$. 
If $e_1,e_2,e_3$ are the three edges adjacent to $v$, we have by 
Definition \ref{def:triangular_honey}(2), $L(e_1)+L(e_2)+L(e_3)=1$.
Hence, considering now oriented edges yields 
\begin{equation}
    \label{eq:honeycomb_divergence_1}
    \forall v \in \interior(V): \sum_{v\sim v'}\mathcal{L}[h](v,v')=s(v) \ .
\end{equation}

\noindent
Let $(\alpha, \beta, \gamma) \in \mathcal{H}_{reg}^3$ and let $G\in\mathcal{G}_d$. 
Recall that for a structure graph $G$, its set of boundary edges
$\partial E = \left\{ \{v, v'\} \mid v \in \partial V \right\}$ consists of edges having an endpoint of degree one. Since by Definition \ref{def:triangular_honey}, 
$\partial V \subset \partial T$, 
we may write $\partial E = \left( e_1, \dots, e_{3n} \right)$, where for 
$0 \leqslant l \leqslant 2$
and $1 \leqslant i \leqslant n$, the edge
$e_{\ell \, n + i}$ is adjacent to $v^{\ell n+ i}$ on $\partial_\ell T$.
Let us denote by 
$g^{\alpha,\beta,\gamma} \in \mathbb{R}^{\partial V}$ the boundary condition
given by 
$$ g^{\alpha,\beta,\gamma}_{\ell n+i} \coloneqq g^{\alpha,\beta,\gamma}(e_{\ell \, n + i}) 
= 
\left\lbrace 
\begin{aligned} 
    &\beta_i&\text{if }c(e_{\ell n+i})=0 \\ 
    &\beta_i-1&\text{if }c(e_{\ell n+i})=1
\end{aligned} \right . $$
for $\ell = 0$ and $1\leqslant i\leqslant n$ and replacing 
$\beta_i$ by $\alpha_i$ (resp. $\gamma_i$) 
when $\ell = 1$ (resp. $\ell = 2$). Remark that for $v\in \partial V$, $s(v)=1$ 
if and only if $c(e)=1$, where $e$ is the unique edge adjacent to $v$, 
see for example Figure \ref{fig:boundary_toric_honey}. 
Hence, the boundary condition translates into the condition 
\begin{equation}
    \label{eq:honeycomb_divergence_2}\mathcal{L}[h](v_{\ell n+i},v') 
    = g^{\alpha,\beta,\gamma}_{\ell n+i} \ ,
\end{equation}
where $v'$ is the unique vertex of $V$ adjacent to $v_{\ell n+i}$. 
Hence, setting $\phi_{\alpha,\beta,\gamma}(v)=s(v)$ for $v\in \interior(V)$ 
and $\phi_{\alpha,\beta,\gamma}(v_{\ell n+i})=g^{\alpha,\beta,\gamma}_{\ell n+i}$ 
for $v^{\ell n+i}\in \partial V$, \eqref{eq:honeycomb_divergence_1} 
and \eqref{eq:honeycomb_divergence_2} yields that for $h\in \honey_{n,d}^G(\alpha,\beta,\gamma)$,
$$\mathcal{L}[h] \in \mathcal{F}_{G}(\phi_{\alpha,\beta,\gamma}) \ .$$

\begin{proposition}[Honeycomb injection]
    \label{prop:parametrization_triangular_honey_first}
    Let $d\geqslant 0$ and $G\in \mathcal{G}_d$. 
    The map $\mathcal{L}: h\rightarrow \mathcal{L}[h]$ is an injective 
    map from $\honey^G$ to $\Omega^1(G)$ and 
    $$\mathcal{L} \left(\honey_{n,d}^G(\alpha,\beta,\gamma)\right) 
    \subset  \mathcal{F}_{G}(\phi_{\alpha,\beta,\gamma}) \ .$$
\end{proposition}

\begin{proof}
The fact that $\mathcal{L} \left(\honey_{n,d}^G(\alpha,\beta,\gamma)\right) 
\subset  \mathcal{F}_{G}(\phi_{\alpha,\beta,\gamma})$ is given by the previous discussion. 
Let us prove the injectivity of the map. Let $h_1,h_2\in \honey^G$ such that 
$\mathcal{L}[h_1] = \mathcal{L}[h_2]$. 
Since $G=(V,E)$ is isomorphic to the structure graph of both $h_1$ and $h_2$, 
there are two isomorphisms $\iota_i:G\rightarrow (V^i,\mathcal{E}^i)$, $i\in\{1,2\}$, 
where $\mathcal{E}^i$ is the set of segments associated to $h_i$ 
by Definition \ref{def:toric_honeycomb} and $V^i$ are the endpoints of these segments. 
Since both $h$ and $h'$ are honeycombs on the equilateral triangle, 
for which there exists a unique segment between two points, 
and $(V^1,\mathcal{E}^1)$ is isomorphic to $(V^2,\mathcal{E}^2)$, it suffices to show the equality $V^1 = V^2$. 
\\
\\
Let $v\in V$. First, suppose that $v\in \partial V$. Since the boundary 
$\partial V \simeq (v^1,\ldots,v^{3n})$ 
of $G$ is ordered, there exists $1\leqslant j\leqslant 3n$ such that $v=v^j$ and there exists a 
unique edge $e\in E$ adjacent to $v$. 
Since $(V^1, \mathcal{E}^1)$ and $(V^2,\mathcal{E}^2)$ are isomorphic to $G$ 
as colored graphs with ordered boundary, $\iota_i(v)$ is the $j$-th vertex of the 
boundary of $h_i$ and $c(\iota_i(e))=c(e)$ for $i \in \{1,2\}$. 
By Definition \ref{def:triangular_honey}, $\iota_1(v)$ and $\iota_2(v)$ belong to the 
same boundary $\partial_\ell$ of $T$ and their $(\ell+1)-$coordinates are 
\begin{align*}
\iota_1(v)_{\ell+1}=\delta_{c(\iota_1(e))=1} + (-1)^{\delta_{c(\iota_1(e))=1}} L(\iota_1(e)) 
=&\delta_{c(e)=1}+(-1)^{\delta_{c(e)=1}}L(e)\\
=&\delta_{c(\iota_2(e))=1}+(-1)^{\delta_{c(\iota_2(e))=1}} L(\iota_2(e)) = \iota_2(v)_{\ell+1}.
\end{align*}
Hence, $\iota_1(v) = \iota_2(v)$.
\\
\\
Suppose that $v\in \interior(V)$. By Definition \ref{def:triangular_honey}, 
$v$ is a trivalent vertex 
and they are three edges $e^0,e^1,e^2$ adjacent to $v$. 
Moreover, by Lemma \ref{lem:types_colors_determined_by_boundaries}, 
the type and color of $e^i$ is given by the graph structure and the type 
and color of the boundary edges. 
Suppose without loss of generality that for $1 \leqslant i \leqslant 2$, $\ell(e^i)=i$. Then,
$$\iota_1(v)=(L(e^0),L(e^1),L(e^2))=\iota_2(v) \ .$$
We deduce that $V^1=\iota_1(V) = \iota_2(V)=V^2$ and thus $h_1=h_2$. 
\end{proof}

\subsection{Dual hive}
\label{subsec:dual_hive}

Let us recall the definition of a \textit{dual hive} from 
\cite{françois2024positiveformulaproductconjugacy}. 
For $n \geqslant d \geqslant 0$, let us consider the graph 
$H_{d, n} = (R_{d, n}, E_{d, n})$ with vertices $R_{d, n}$ and edges $E_{d, n}$. 
Each vertex $v=r+se^{\pi i/3}\in H_{d,n}$ comes with a coordinate 
$(v_0,v_1,v_2)=(n+d-r-s,r,s)$. 
Each edge $e$ of $H_{d,n}$ written  $e=(v,v - e^{2\pi i\ell/3})$ with $\ell\in\{0,1,2\}$ is 
labelled $(\ell(e),h(e))\in\{0,1,2\}\times \{0,\ldots,n+d\}$ with 
\begin{equation}
    \label{eq:def_type_height_dual_hive}
    \ell(e)=\ell \text{ and } h(e)= v_{\ell} \ .
\end{equation}

\noindent
Table \ref{tab:edge_types} below summarizes the different edge types for dual hives 
and honeycombs; edges of dual hives will be in duality with edges of honeycombs with the same type. 

\begin{table}[ht]
    \centering
\begin{tabular}{c | c | c}
    Type $\ell$ & $e = \left(v, v - e^{2\pi i\ell/3}\right) \in E_{n, d}$ 
    & $e \subset x + \ed^{2i \pi (\ell+1)/3} \in \honey_{n, d}$ \\
\hline
0 &
\begin{tikzpicture}[scale=1.5, baseline=(current bounding box.center)]
    \node at (0,0) {$\bullet$};
    \node[right] at (0,0) {$v$};
    \node at (-1,0) {$\bullet$};
  \draw[black] (0,0) -- ++(-1,0);
\end{tikzpicture}
&
\begin{tikzpicture}[scale=1.5, baseline=(current bounding box.center)]
  \draw[black, thin] (0,0) -- ++(-0.5,{sqrt(3)/2});
  \node at (0,0) {$\bullet$};
  \node at (-0.5,{sqrt(3)/2}) {$\bullet$};
\end{tikzpicture}
\\
\hline
1 &
\begin{tikzpicture}[scale=1.5, baseline=(current bounding box.center)]
    \draw[black, thin] (0,0) -- ++(-0.5,{sqrt(3)/2});
  \node at (0,0) {$\bullet$};
  \node[left] at (-0.5,{sqrt(3)/2}) {$v$};
  \node at (-0.5,{sqrt(3)/2}) {$\bullet$};
\end{tikzpicture}
&
\begin{tikzpicture}[scale=1.5, baseline=(current bounding box.center)]
    \draw[black, thin] (0,0) -- ++(0.5,{sqrt(3)/2});
  \node at (0,0) {$\bullet$};
  \node at (0.5,{sqrt(3)/2}) {$\bullet$};
\end{tikzpicture}
\\
\hline
2 &
\begin{tikzpicture}[scale=1.5, baseline=(current bounding box.center)]
    \draw[black, thin] (0,0) -- ++(0.5,{sqrt(3)/2});
  \node at (0,0) {$\bullet$};
  \node[left] at (0,0) {$v$};
  \node at (0.5,{sqrt(3)/2}) {$\bullet$};
\end{tikzpicture}
&
\begin{tikzpicture}[scale=1.5, baseline=(current bounding box.center)]
      \node at (0,0) {$\bullet$};
    \node at (-1,0) {$\bullet$};
  \draw[black] (0,0) -- ++(-1,0);
\end{tikzpicture}
\end{tabular}
\vspace{0.2cm}
\caption{Edge types in dual hives and honeycombs.}
\label{tab:edge_types}
\end{table}

\begin{definition}[Color map]
    A \textit{color map} is a map $C: E_{n,d} \rightarrow \{ 0, 1, 3, m \}$ such that 
    the boundary colors around each triangular face in the clockwise order is either 
    $(0,0,0), (1,1,1), (1,0,3)$ or $(0,1,m)$ up to a cyclic rotation.
\end{definition}

\begin{definition}[Non-degenerate dual hive]
\label{def:limit_dual_hive}
    For $(\alpha, \beta, \gamma) \in \mathcal{H}_{reg}^3$, 
    such that $|\alpha| + |\beta| = |\gamma| + d$,
    the set of \textit{dual hives}, 
    denoted by $\dualhive(\alpha, \beta, \gamma)$, 
    is the set of pairs $(C, L)$ such that :
    \begin{enumerate}
        \item $C: E_{n,d} \rightarrow \{ 0, 1, 3, m \}$ is a color map,
        \item $L: E_{n,d} \rightarrow \R_{\geqslant 0} $ is the label map satisfying 
        \begin{enumerate}
            \item $L(e_1) + L(e_2) + L(e_3) = 1 $ for every triangular face of $H_{d, n}$,
            \item if $e,e'$ are edges of same type on the boundary of a same lozenge $f$,
            \begin{enumerate}
                \item $L(e)=L(e')$ if the middle edge of $f$ is colored $m$,
                \item $L(e)> L(e')$ if $h(e)> h(e')$ and the middle edge of $f$ is not colored $m$.
            \end{enumerate}
            \item The values of $L$ on $\partial E_{n,d}$ are given 
            by $(\alpha, \beta, \gamma)$ so that, sorted in decreasing height 
            of edges, see Figure \ref{fig:boundary_limit_hive} below.
            \begin{align*}
                \ell^{(0, 1)} &= (1 - \alpha_d, \dots, 1 - \alpha_1),
                &\ell^{(2, 2)} &= (\alpha_{d+1}, \dots, \alpha_{n}) \\
                \ell^{(2, 0)} &= (1 - \beta_d, \dots, 1 - \beta_1),
                &\ell^{(1, 1)} &= (\beta_{n}, \dots, \beta_{d+1}) \\
                \ell^{(1, 2)} &= (\gamma_n, \dots, \gamma_{n-d+1}), 
                &\ell^{(2, 2)} &= (1-\gamma_{n-d}, \dots, 1- \gamma_{1}).
            \end{align*}
            Moreover, for $\ell \in \{0,1,2\}$, the values of the color map $C$ on  
            $\partial^{(\ell, \ell)}$ is set to $0$ while equal to $1$ on other boundary edges. 
            We call the triple $(\alpha, \beta, \gamma)$ the \textit{boundary} of $L$, or 
            of the dual hive.
        \end{enumerate}
    \end{enumerate}
\end{definition}

\noindent
Figure \ref{fig:ex_limit_dual_hive} shows an example of a dual hive for $d=1$ and $n=3$ 
with boundary 
\begin{equation}
    \label{eq:boundary_example_dual_hive}
    (\alpha, \beta, \gamma) = \left( \left( \frac{14}{23}, \frac{7}{23}, \frac{2}{23} \right), 
    \left( \frac{18}{23}, \frac{10}{23}, \frac{3}{23} \right), 
    \left( \frac{19}{23}, \frac{10}{23}, \frac{2}{23} \right) \right) .
\end{equation}
Colors red, blue, black and greeen correspond to values $0,1,3$ and $m$ of 
the color map respectively.

\begin{figure}[H]
    \centering
    \begin{minipage}[b]{0.45\textwidth}
        \centering
        \includegraphics[width=\textwidth]{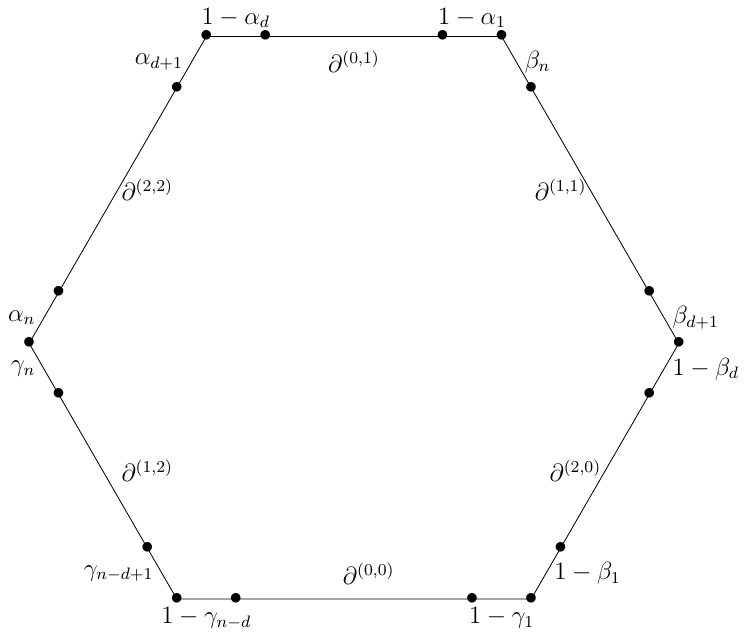}
        \caption{Boundary condition in $\dualhive(\alpha, \beta, \gamma)$.}
        \label{fig:boundary_limit_hive}
    \end{minipage}
    \hfill
    \begin{minipage}[b]{0.45\textwidth}
        \centering
        \includegraphics[width=\textwidth]{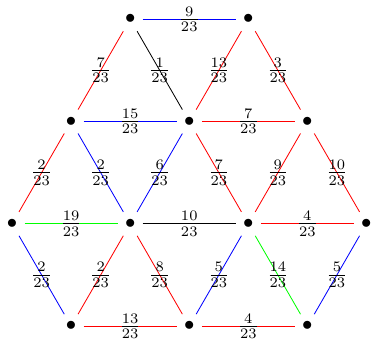}
        \caption{A dual hive with boundary condition 
        \eqref{eq:boundary_example_dual_hive}. }
        \label{fig:ex_limit_dual_hive}
    \end{minipage}
\end{figure}

\noindent
For a given color map $C$, let us denote by 
$\dualhive^C(\alpha, \beta, \gamma)$ the set of dual hives with 
boundary $(\alpha,\beta,\gamma)$ and color map $C$. 
Since an element of $\dualhive^C(\alpha,\beta,\gamma)$ is uniquely defined by its map 
$L:E_{n,d}\rightarrow \mathbb{R}_{\geqslant 0}$, the set $\dualhive^C(\alpha,\beta,\gamma)$ 
can be seen as an affine polytope of $\mathbb{R}^{E_{n,d}}$ written as 
$$\dualhive^C(\alpha,\beta,\gamma)=A^C\cap K_{n,d} \ ,$$ 
where $K_{n,d}$ is the cone of induced by (2)(b)(ii) and $A^C$ is the 
affine subspace induced by the equalities coming from (2)(a), (2)(b)(i) and (2)(c).

\subsection{From dual hive to triangular honeycomb}
\label{subsec:from_dual_hive_to_honey}

\begin{definition}[$\Gamma_{d, n}$ graph]
    Let $n \geqslant d$ be two integers. 
    The \textit{dual graph} $\Gamma_{d, n} = 
    (V^\Gamma, E^\Gamma)$ of $H_{d, n}$ is the following graph :
    \begin{itemize}
    \item there is one vertex $v_f$ for each triangular 
    face $f$ of $H_{d,n}$ and one vertex $v_{\tilde{e}}$ 
    for each outer edge $\tilde{e}$ of $H_{d,n}$ ,
    \item there is an edge $e$ between $v_f$ and $v_{f'}$ 
    (resp. between $v_f$ and $v_{\tilde{e}}$) if the faces $f$ and $f'$ 
    share an edge $\tilde{e}$ in $H_{d,n}$ 
    (resp. if $\tilde{e}$ is a boundary edge of $f$ in $H_{d,n}$).
    \end{itemize}
\end{definition}

\noindent
The map $e\mapsto \tilde{e}$ yields a bijection from $E^{\Gamma}$ 
to $E_{d,n}$ and $e$ is then said \textit{dual} to $\tilde{e}$. 
Hence, any color map $C:E_{d,n}\rightarrow \{0,1,3,m\}$ yields a 
color map, also denoted by $C$, from $E^\Gamma$ to $\{0,1,3,m\}$ 
by setting $C(e)=C(\tilde{e})$. 
Likewise, any edge of $E^{\Gamma}$ inherits the type 
$\ell(e)=\ell(\tilde{e})\in\{0,1,2\}$, the height 
$h(e)=h(\tilde{e})$ and the label 
$L(e)=L(\tilde{e})$ of its dual edge.
\\
\\
To a dual hive $H=(C,L)\in \dualhive(\alpha,\beta,\gamma)$, 
we associate a collection $\mathcal{S}(H)$ of segments 
of $T$:

\begin{enumerate}
\item For each $v\in V^{\Gamma}$ :
\begin{itemize}
\item if $v$ is adjacent to three edges 
$(e^0,e^1,e^2) \in E_\Gamma^3$ with $e^\ell$ of type $\ell$. 
We then set $x_v= \left(L(e^0),L(e^1),L(e^2)\right) \in T$,
\item if $v\in \partial V^{\Gamma}$ and $v$ is adjacent to an 
edge $e$ such that $\tilde{e}$ $\in \ell^{(i,i)}$ 
(resp. in $\ell^{(i+1,i+2)}$), 
we set $x_v=(L(e)\delta_{i, 0}, L(e) \delta_{i, 1}, 
L(e)\delta_{i, 2})$ 
(resp. $x_v=((1-L(e))\delta_{i, 0}, (1-L(e)) \delta_{i, 1}, 
(1-L(e)) \delta_{i, 2})$), 
where $\delta_{i, j} = 1 $ if $i=j$ and $0$ otherwise.
\end{itemize}
\item Then, we set 
$$\mathcal{S}(H) = \left\{[x_v,x_{v'}] \mid e=\{v,v'\} \in E^\Gamma \right\} .$$
\end{enumerate}
For $e=\{v,v'\} \in E^\Gamma$, let us denote its associated segment by 
\begin{equation}
    \label{eq:def_Phi}
    \Phi(e) = [x_v,x_{v'}] \ .
\end{equation}
As shown below, the collection of segments 
$\mathcal{S}(H)$ is almost the edge set of the 
structure graph of a honeycomb.

\begin{lemma}[Edge segments]
    \label{lem:hive_honey_1}
    Suppose that $\tilde{e} \in E_{n, d}$ is an edge of type 
    $\ell$ adjacent to a face $\tilde{v}$ of $H_{n,d}$ 
    and set $\epsilon=+1$ (resp. $\epsilon=-1$) if this 
    face is  a lower (resp. upper.) triangular face.  
    Then, either $c(e) \neq m$ and
    $$\Phi(e)\subset x_v + \epsilon e^{2 \pi i(\ell+1)/3} \mathbb{R}_{>0} \ ,$$
    or $c(e)=m$ and 
    $$\Phi(e)=\{x_v\} \ .$$
\end{lemma}

\begin{proof}
    Suppose without loss of generality that $e$ is of type $0$. 
    Then, $(x_v)_0=(x_{v'})_0=L(e)$, so that $x_0=L(e)$ for any 
    $x\in \Phi(e)=[x_v,x_{v'}]$. 
    We deduce that $\Phi(e)\subset v+\mathbb{R}e^{2\pi/3}$. 
    \\
    If $e$ is not colored $m$ 
    and is of the form $e = \{v_f, v_{f'}\}$, 
    consider the lozenge 
    of $H_{n,d}$  consisting of faces $f$ and $f'$ 
    whose middle edge is $\tilde{e}$. 
    Denote by $\tilde{f},\tilde{f}'$ the two edges of type $2$ of 
    this lozenge, with the convention that $h(\tilde{f}')>h(\tilde{f})$ 
    and $\tilde{f}$ (resp. $\tilde{f}'$) is a boundary edge of the 
    face dual of $v$ (resp. $v'$). 
    Then, by Condition (2)(ii) of Definition \ref{def:limit_dual_hive}, 
    $L(\tilde{f}')>L(\tilde{f})$ and thus $(x_{v'})_{2}>(x_v)_2$. 
    If $e$ is of the form $e = \{v_f, v_{\tilde{e}}\}$, 
    we have that $x(v_{\tilde{e}})_2 = 0$ if $\tilde{e} \in \partial^{(0, 0)}$ 
    or $x(v_{\tilde{e}})_2 = 1 - L(e)$ if $\tilde{e} \in \partial^{(0, 1)}$. 
    The two previous cases correspond to $\epsilon = -1$ and 
     $\epsilon = 1$ respectively. In both cases $\epsilon \cdot (x_{v_{\tilde{e}}})_{2} 
     > (x_v)_2$. We deduce that $e\subset x_v+\mathbb{R}_{>0}e^{2\pi/3}$. 
    A consequence of this fact is that the angle between two consecutive 
    edges adjacent to an edge $v$ is $2\pi/3$. 
    If $e$ is colored $m$ Condition (2)(b)(i) of Definition 
    \ref{def:limit_dual_hive} implies that $x_{v}=x_{v'}$ 
    and thus $\Phi(e)=\{x_v\}$.
\end{proof}

\begin{lemma}[Distinct edges give disjoint segments]
    \label{lem:hive_honey_non_crossing}
    If $e,e'\in E^\Gamma$ are distinct, then 
    $$\interior(\Phi(e))\cap \interior(\Phi(e'))=\emptyset \ .$$
\end{lemma}

\noindent
This lemma is a rephrasing in the continuous case of the statements 
of Lemma \cite[5.9]{françois2024positiveformulaproductconjugacy} 
and Lemma \cite[5.10]{françois2024positiveformulaproductconjugacy}. 
We provide here a proof which is much simpler in its continuous version.

\begin{proof}
For $\tilde{e}\in E_{n,d}$, denote by $\tilde{e}_i=\tilde{v}_i$ (resp. $\tilde{e}^i$), 
where $\tilde{v}$ is the upper-triangular face (resp. lower-triangular) 
which is delimited by $e$.
Let $e,e'$ be of same type $\ell$. First, by iterating Condition (2)(ii) 
of Definition \ref{def:limit_dual_hive}, $L(e)>L(e')$ if $e_{\ell+1}=e'_{\ell+1}$ 
and $e_{\ell}>e'_{\ell}$. Next, using Condition (2)(ii) of 
Definition \ref{def:limit_dual_hive} and the fact that $C$ is a color map, 
$L(e)>L(e')$ if $e_{\ell+1}=e'_{\ell+1}-1$ and $e_{\ell}=e'_{\ell}+1$. 
Therefore, $L(e)>L(e')$ if $e$ and $e'$ are of same type $\ell$ and 
$e_{\ell+1}\leqslant e'_{\ell+1}$, $e_{\ell}>e'_{\ell}$. 
The same reasoning yields that $L(e)\geqslant L(e')$ if $e$ and $e'$ are of same 
type $\ell$ and $e_{\ell+1}\leqslant e'_{\ell+1}$, $e_{\ell}=e'_{\ell}$ 
with equality only if all edges of type $\ell-1$ between $e$ and $e'$ 
are colored $m$.
\\
Next suppose that $e=\{v_1,v_2\}$ and $e'=\{v'_1,v'_2\}$ with $e\not =e'$, 
with $v_1,v'_1$ being dual to an upper-triangular face and $v_2,v'_2$ being 
dual to a lower-triangular face. If $v_i=v'_j$ for some $i,j\in\{1,2\}$, 
then $\interior(\Phi(e))\cap \interior(\Phi(e'))=\emptyset$ by Lemma \ref{lem:hive_honey_1}. 
\\
Otherwise, suppose without loss of generality that $(v_1)_0<(v'_1)_0$. 
Since $\sum_{j=0}^2(v_1)_j= \sum_{j=0}^2(v'_1)_j=1$, we can assume without loss 
of generality that $(v_1)_2>(v'_2)_2$. Then, by the reasoning above, 
the edge $e^2_1$ (resp. $e^2_2$) of type $2$ adjacent to $v_1$ (resp. $v_2$) 
satisfy $L(e^2_1)>L(e^2_2)$.
\\
Set $x^i:=x_{v_i}$ and $y^i=x_{v'_i}$ for $i=1,2$. 
Since $x^1_2=L(e^2_1)$ and $y^1_2=L(e^2_2)$, by the previous reasoning 
$x^1_2>y^1_2$. Doing the same with the lower triangular faces $v_2,v'_2$, 
which must be adjacent respectively to $v_1$ and $v'_1$, yield that 
$x^2_2\geqslant y^2_2$. Hence, the segments $\Phi(e)=[x^1,x^2]$ and 
$\Phi(e')=[y^1,y^2]$ can only meet at $y^2$, and $\interior(\Phi(e))\cap \interior(\Phi(e'))=\emptyset$. 
\end{proof}

\begin{definition}[Maximal chain, reduced graph]
    \label{def:skeleton}
    Let $n\geqslant d$ be integers and let 
    $C:E_{n,d}\rightarrow\{0,1,3,m\}$ be a color map. 
    \begin{itemize}
        \item  A \textit{maximal chain} of $C$ is a path 
        $\gamma = (e_1,\ldots,e_{2r+1}) \in (E_{n, d})^{2r+1}$ for some 
        $r \geqslant 0$ such that 
        $c(e_{2i})=m$, $c(e_{2i+1})=c(e_1)$ for $1 \leqslant i \leqslant r$, 
        and such that two consecutive edges share a vertex.
        We write $\gamma=\{ x, y \}$  for $x,y \in V^\Gamma$ 
        where $x$ (resp. $y$) is dual to a face $f_x$ (resp $f_y$) 
        in $H_{n, d}$ such that $e_1 \in f_x$ (resp. $e_{2r+1} \in f_y$) 
        and where $f_x$ and $f_y$ do not have any $m$ edges
        to emphasize that the path 
        goes from $x$ to $y$. 
    Moreover, the color of $\gamma$ is defined as $c(\gamma)=c(e_1)$. 
    \\
    \item The \textit{reduced graph of C} is the graph $G^C=(V^C,E^C)$ 
    defined by:
    \begin{itemize}
    \item $V^C=\interior(V^C)\cup \partial (V^C)$, where $\interior(V^C) = 
    V^{\Gamma} \setminus \{u\in V_{\Gamma} \mid \exists 
    (u,v)\in E_{\Gamma}, C(\{u,v\})=m\}$ and $\partial V^C=\partial V^{\Gamma}$,
    \item $E^C=\{ \gamma = \{x,y\} \mid \{x,y\}\text{ is a maximal chain of }C\}.$
    \end{itemize}
    The boundary vertices of $G^C$ are ordered as 
    the ones of $\Gamma$.
    \end{itemize}
\end{definition}

\noindent
Remark that the definition of the edge set is valid, 
since any vertex $u\in V^\Gamma$ adjacent to an edge colored $m$ 
cannot be the endpoint of a maximal chain of $C$. 
Moreover, by the color condition, a maximal chain with $c(e_1)=3$ 
is necessarily of length $1$. 
Note that any edge $e = \{x, y\} \in E_{n, d}$ not adjacent to an edge 
colored $m$ is a maximal chain (with $r=0$) and is thus in $E^C$.
\\
\\
The map $C\mapsto C^C$ is injective as we can recover $C$ 
from $G^C$: it suffices to color the successive edges of a 
maximal chain $\gamma$ as $c(e_{2i+1})=c(\gamma)$ and $c(e_{2i})=m$.
In the sequel, we denote by $\partial E^C$ 
the set of edges adjacent to a univalent vertex of $G^C$. 
Following Definition \ref{def:limit_dual_hive} and Definition \ref{def:skeleton}, 
we introduce a partial order $\leqslant$ on $E^C$ by completing the relation 
$e\leqslant e'$
if there exists an edge $\tilde{e}\in E_{d,n}$ 
(resp. $\tilde{e}'\in E_{d,n}$) dual to an edge in the equivalence class of 
$e$ (resp. $e'$) and such that $\tilde{e},\tilde{e}'$ are of same type, 
adjacent to the same lozenge and satisfy $h(\tilde{e}')\geqslant h(\tilde{e})$. 
\\
\\
Let $C$ be a color map and let $H \in \dualhive^C(\alpha,\beta,\gamma)$ 
be a dual hive. For any $\hat{e} \in E^C$, let us set 
$\hat{\Phi}[H](\hat{e})=\bigcup_{e\in \hat{e}}\Phi[H](e)$ and
$$\rho_C(H)=\left(\{\hat{\Phi}[H](\hat{e}) 
,\hat{e}\in E^C\},c\right)=\bigcup_{e\in E}\Phi[H](e) \ ,$$
where $c:\{\hat{\Phi}[H](\hat{e}) 
,\hat{e}\in E^C\}\rightarrow \{0,1,3\}$ with $c(\hat{\Phi}[H](\hat{e}))$ being the unique color different from $m$ in the maximal chain $\hat{e}$. Remark that $\bigcup_{\hat{e}\in E^C}\hat{\Phi}[H](\hat{e}) =\bigcup_{e\in E}\Phi[H](e)$.
\begin{lemma}[Reduced graph segments]
    The set $ \left\{\hat{\Phi}[H](\hat{e}) \mid \hat{e}\in E^C \right\}$ 
    is a set of segments of $T$.
\end{lemma}

\begin{proof}
    Let $e=\{v_1,v_2\}$ and $e'=\{v_1',v_2'\}$ be edges of $E^\Gamma$ such that 
    $\{v_2,v_1'\}\in E^\Gamma$ and $c(\{v_2,v_1'\})=m$. 
    Suppose without loss of generality that $v_2$ (resp. $v_1'$) 
    is dual to an upper (resp. lower) triangular face of $H_{d,n}$. 
    Let $\ell \in \{0, 1, 2\}$ be the type of edges $e$ and $e'$.
    Then, by Lemma \ref{lem:hive_honey_1}, $x_{v_2}=x_{v_1'} \coloneqq x_v$, 
    $\Phi(e)\subset x_{v_2}-\mathbb{R}e^{2(\ell+1)\pi/3}$ 
    (resp. $\Phi(e')\subset x_{v_2}+\mathbb{R}e^{2(\ell+1)\pi/3}$) 
    and $x_v\in \Phi(e)\cap \Phi(e')$, so that $\Phi(e)\cup \Phi(e')$ is 
    a segment corresponding to $[x_{v_1},x_{v_2'}]$. 
    Hence, if $\hat{e}=\{v,v'\}$ is a maximal chain of $C$, 
    $\bigcup_{e\in \hat{e}}\Phi(e)$ is the segment $[x_{v},x_{v'}]$.
\end{proof}

\noindent
For $G\in\mathcal{G}_d$, recall that 
$\phi_{\alpha,\beta,\gamma} \in \Omega^0(V)$ has been 
defined before Proposition \ref{prop:parametrization_triangular_honey_first}
in Section \ref{subsec:parametrization_honeycombs} and 
that the map $\mathcal{L}: \honey_{n,d}^{G} \to \Omega^1(G)$ 
has been defined in Definition \ref{def:flow_honeycomb}.

\begin{proposition}[Dual hives as honeycombs]
    \label{lem:hive_honey_final}
    Let $C:E_{n,d}\rightarrow\{0,1,3\}$ be a color map. 
    The map $\rho_C$ is a injection from 
    $\dualhive^C(\alpha,\beta,\gamma)$ to 
    $\honey_{n,d}^{G^C}(\alpha,\beta,\gamma)$ such that the map 
    $ \mathcal{L} \circ \rho_C: \dualhive^C(\alpha,\beta,\gamma) 
    \rightarrow \Omega^1(G^C)$ 
    is the restriction of an affine map with integer coefficients from 
    $\mathbb{R}^{E_{n,d}}$ to $\mathcal{F}_{G^C}(\phi_{\alpha,\beta,\gamma})$.
\end{proposition}
\begin{proof}
Let us first prove that for $H \in \dualhive^C(\alpha,\beta,\gamma)$, 
$\rho_{C}(H)$ is a triangular honeycomb. 
We first check that the two conditions of Definition \ref{def:toric_honeycomb} 
are fullfilled.
\begin{enumerate}
\item Suppose that $\hat{e}\not=\hat{e}'$ and 
$\interior(\hat{\Phi}(\hat{e}))\cap \interior(\hat{\Phi}(\hat{e}'))\not=\emptyset$. 
Let $x\in \interior(\hat{\Phi}(\hat{e}))\cap \interior(\hat{\Phi}(\hat{e}'))$. 
Since $\hat{\Phi}(\hat{e})=\bigcup_{e\in\hat{e}}\Phi(e)$, 
$\hat{\Phi}(\hat{e}')=\bigcup_{e\in\hat{e}'}\Phi(e)$ and, 
by Lemma \ref{lem:hive_honey_non_crossing}, 
$\interior(\Phi(e))\cap \interior(\Phi(e'))=\emptyset$ for $e\not=e'$, we have that $x=x_v$ 
for some $v\in  e\cap e'$ with $e\in \hat{e}$, $e'\in \hat{e}'$ not colored $m$. 
By Lemma \ref{lem:hive_honey_1} and up to switching $e$ and $e'$, 
the angle from $\Phi(e)$ to $\Phi(e')$ is $2\pi/3$. 
Since $x\in \interior(\hat{\Phi}(\hat{e}))$, $v$ is adjacent to a third edge colored 
$m$; since $C$ is a color map, $c(e)=1$ and $c(e')=0$.  
\item Suppose that $x\in \partial \hat{\Phi}(\hat{e}))\cap \partial\hat{\Phi}(\hat{e}')$. 
Then, there exists $e\in\hat{e}$, $e'\in\hat{e}'$, neither of them colored $m$, 
such that $x\in \partial \Phi(e)\cap \partial\Phi(e')$. 
Then, Lemma \ref{lem:hive_honey_1} and the fact that $C$ is a color map yields 
the second condition.
\end{enumerate}

\noindent
Hence, $\rho_C(H)$ is a honeycomb and the structure graph is given by 
$$G[\rho_C(H)] = \left( \left\{x_v,v\in V^C \right\}, 
\left\{ \partial\hat{\Phi}[H](\hat{e}),\hat{e}\in G^C \right\}\right),$$
so that $G[\rho_C(H)]$ is isomorphic to $G^C$ as colored graph with ordered boundary.  
We next turn to the conditions of being a triangular honeycomb.
\begin{enumerate}
\item Let $x$ be a vertex of $G[\rho_{C}(H)]$. 
Then, $x$ is the endpoint of a segment $\hat{\Phi}(\hat{e})=\bigcup_{e\in\hat{e}}\Phi(e)$. 
Hence, $x=x_{v}$ for some $v\in V^{\Gamma}$ which is either dual to a triangular face 
$\tilde{v}$ without edge $m$ on its boundary (for otherwise $x_{v}\in \interior(\hat{\Phi}(\hat{e}))$), 
or is equal to $v_{\tilde{e}}$ for some $\tilde{e}\in E_{d,n}$.  
In the first case, $x$ is trivalent and, by Lemma \ref{lem:hive_honey_1}, 
there are three non-trivial segments in $T$ adjacent to $x$, 
with the angle between two successive segments being equal to $2\pi/3$ : 
this implies that $x\in T\setminus \partial T$. 
In the second case, $x$ is univalent and belongs to $\partial T$ by construction.
\item Condition (2) is a direct consequence of Lemma \ref{lem:hive_honey_1}.
\item  Let $x$ be the $i$-th boundary point of $\rho_{C}(H)$ on $\partial_1T$, so that $x_0=0$. 
If $i\leqslant d$, then $x=x_{v_{\tilde{e}}}$ for the edge $\tilde{e}\in \partial^{(2,0)}$ 
such that $L(\tilde{e})=1-\beta_i$ 
and $c(e)=1$. Since $\tilde{e}$ is of type $2$, $x_2=L(\tilde{e})=1-\beta_i$, and thus 
$x_1=1-(1-\beta_i)=\beta_i \ .$
If $d+1\leqslant i\leqslant n$, then $\tilde{e}\in\partial^{(1,1)}$, $L(\tilde{e})=\beta_i$ and $c(e)=0$. 
Moreover, $\tilde{e}$ is of type $1$ and thus $x_1=L(\tilde{e})=\beta_i$. 
The cases of other boundaries are similar.
\end{enumerate}
Therefore, $\rho_{C}(H) \in \honey_{n,d}^{G^C}(\alpha,\beta,\gamma)$.
\\
\\
Let us now check that $ \mathcal{L} \circ p_C: \dualhive^C(\alpha,\beta,\gamma) 
\rightarrow \Omega^1(G^C)$ is the restriction of an affine map with integer coefficients.
Let $\hat{e}=\{v,v'\}\in E^C$ of type $\ell$ with $s(v)=1$ and $s(v')=-1$, 
and suppose without loss of generality that $v_{\ell}<v'_{\ell}$. 
Let $\mathcal{E}(\hat{e})=\{v,w\}$ be the unique edge of the maximal chain 
$\hat{e}$ adjacent to $v$. 
Then, $\mathcal{E}(\hat{e})$ is of type $\ell$ and 
$\mathcal{L}[\rho_c[H]](\hat{\Phi}[H](\hat{e})) = (x_v)_{\ell} = L(\mathcal{E}(\hat{e}))$.
Hence, $\mathcal{L} \circ \rho_{C}$ is the restriction of the linear map from 
$\mathbb{R}^{E_{n,d}}$ to $\Omega^1(G^C)$ 
mapping $(x(e))_{e\in E_{n,d}}$ to $\sum_{\hat{e}\in E^C}x(\mathcal{E}(\hat{e}))\delta_{\vec{e}}$, 
where for $e=\{v,v'\}\in E^C$ with $s(v)=1$ and $s(v')=-1$, $\vec{e}=(v,v')$. 
Remark that this map has integer coefficients in the canonical bases of both vector spaces.
\\
\\
Finally, since $h\in \honey_{n,d}^G(\alpha,\beta,\gamma)$ is uniquely determined 
by $(L(e))_{e\in G}$ by Proposition \ref{prop:parametrization_triangular_honey_first}, 
the injectivity of the map $\mathcal{L} \circ \rho_C$ will be implied by the injectivity 
of the map $\rho_C$. 
Suppose that $H_1,H_2 \in \dualhive^C(\alpha,\beta,\gamma)$ are distinct and 
denote by $L_1,L_2$ there respective label maps. 
Then, there exists $e\in E_{n,d}$ such that $L_1(e)\not=L_2(e)$. 
Denote by $\ell$ the type of $e$ and, up to using Condition (2)(a) of Definition 
\ref{def:limit_dual_hive} on a triangular face next to $e$, assume that $c(e)\not=m$. 
Let $\hat{e}=\{v,v'\}$ be the maximal chain containing $e$, with the condition that 
$v_{\ell}<v'_{\ell}$. 
Then, $L_i$ is constant on all edges $e\in \hat{e}$ not colored $m$, so that 
$L_1(\mathcal{E}(\hat{e}))=L_1(e)\not=L_2(e)=L_2(\mathcal{E}(\hat{e}))$. 
Hence, $\mathcal{L} \circ \rho_{C}(H_1)\not = \mathcal{L} \circ \rho_{C}(H_2)$, 
and $\rho_{C}$ is injective. 
\end{proof}

\begin{figure}[ht]
    \centering
    \includegraphics[width=\textwidth]{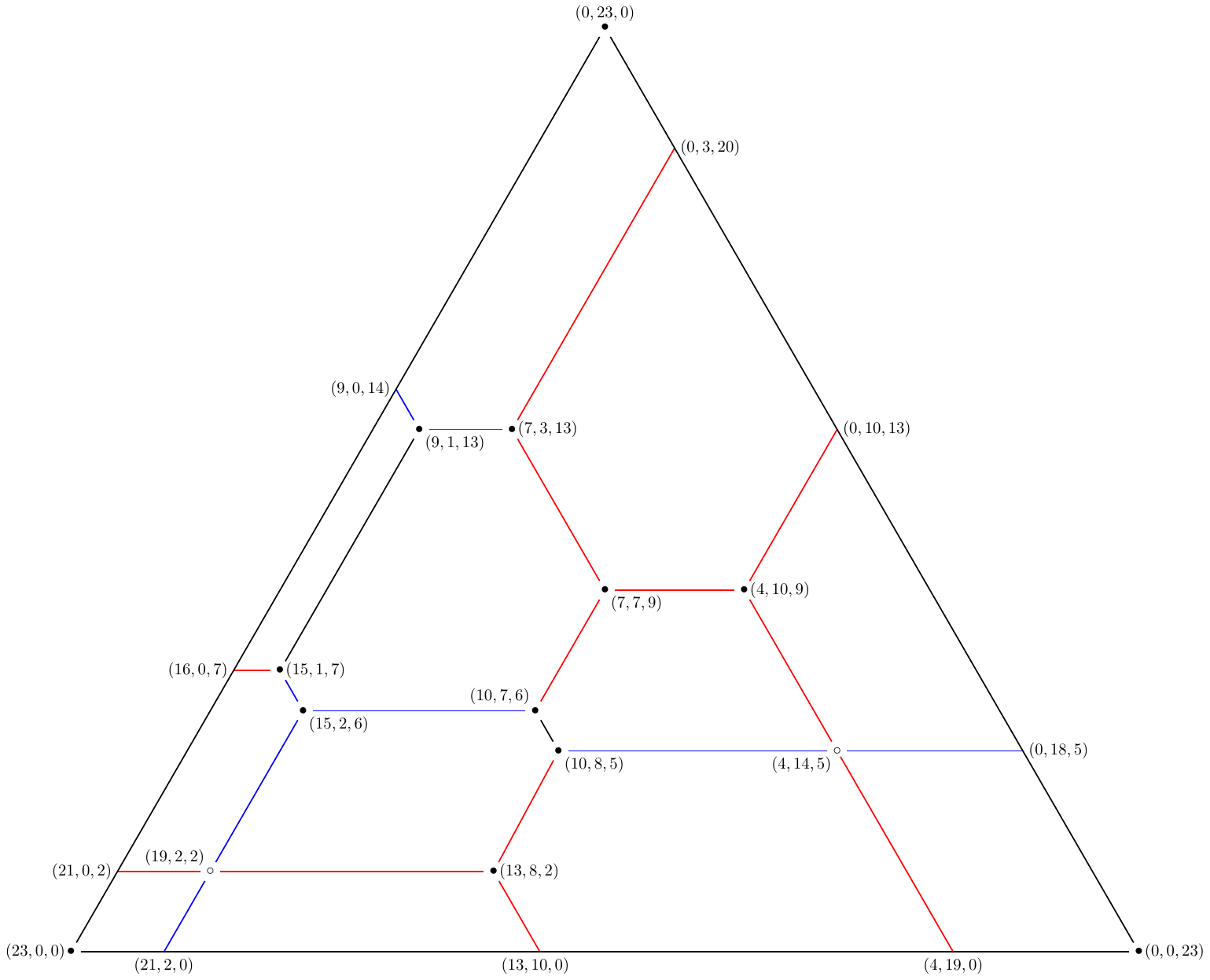}
     \caption{The triangular honeycomb corresponding to the dual hive of 
    Figure \ref{fig:ex_limit_dual_hive}. Coordinates 
    should to be multiplied by $1/23$.}
    \label{fig:ex_triangular_honey}
\end{figure}

\subsection{From triangular honeycomb to dual hive}
\label{subsec:from_honey_to_dual_hive}

Let $h\in \honey_{d,n}(\alpha,\beta,\gamma)$ be a triangular honeycomb with 
graph structure $G = (V^G, E^G) \in \mathcal{G}_d$. 
We construct a graph $\tilde{G} = (\tilde{V}, \tilde{E})$ with color and label 
$(\tilde{c},\tilde{L})$ and a map 
$\tilde{\Phi} : \tilde{E} \rightarrow \mathcal{P}(T)$ as follows : 
\begin{enumerate}
\item first consider an intermediate augmentation 
$\hat{G} = \left( \hat{V}, \hat{E} \right)$ of $G$, 
where $\hat{V}$ consists of vertices $V$ of $G$ together with 
the points which are intersection of two segment of $h$. 
For $x, y\in \hat{V}$, we have an edge $\{x,y\}\in \hat{E}$ if and only if 
$[x,y] \subset e$ for some $e\in h$ and 
$]x,y[ \, \cap \, \hat{V}=\emptyset$. Then, set $\tilde{c}(\{x,y\})=c(e)$ 
and $\tilde{L}(\{x,y\})=L(e) = x_{\ell(e)}$ if $]x,y[\subset e$. 
For $\{x,y\}\in \hat{E}$, one sets $\tilde{\Phi}(\{x,y\})=[x,y] \in \mathcal{P}(T)$.
\\
\item A vertex $v \in \hat{V}$ is of degree either $1$ or $3$ if it comes from a 
vertex of $G$ or of degree $4$ 
if it comes from a non-empty intersection $\iota(e) \cap \iota(e')$ for $e,e'\in E^G$. 
In the latter case, the four edges $\left\{ \left\{v,x_i^\pm\right\}, i=0,1 \right\}$ 
adjacent to $v$ in $\hat{G}$ are such that 
$\{v,x^\pm_i\}$ is colored $i$ and of type $\ell-i$ for some $\ell\in\{0,1,2\}$. 
In particular, the angle $\widehat{x_0^\epsilon vx_1^\epsilon}=2\pi/3$ for $\epsilon\in\{-,+\}$. 
Replace $v$ by two vertices $v^+,v^-$, add an edge $e$ to $\hat{E}$ with color $m$, 
type $\ell+1$ and label $1-L(\{v,x^{\pm}_0\})-L(\{v,x^\pm_1\})$ 
between $v^+$ and $v^-$. Set $\tilde{\Phi}(e) = \{ v \}$. 
Replace each edge $\{v,x_i^\pm\}$ by $\{v^\pm,x_i^\pm\}$, keeping the same label and color. 
Repeat the operation successively for each vertex of degree $4$. 
\begin{figure}[ht]
    \centering
    \includegraphics[scale=1]{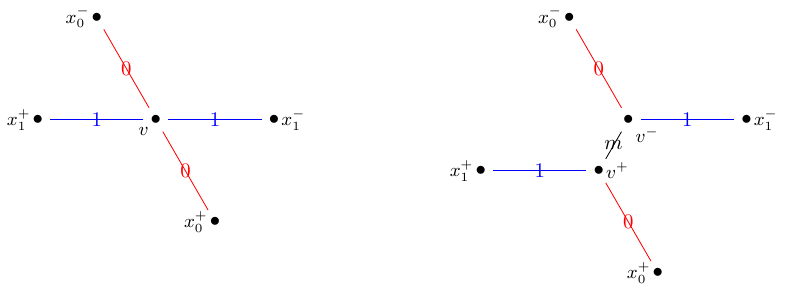}
    \caption{The augmentation of a vertex $v$ of degree $4$ to two vertices $v^\pm$ 
    of degree $3$ linked by an edge of type $\ell+1=1$. 
    The configuration is the one on the bottom right of Figure \ref{fig:ex_triangular_honey}.}
    \label{fig:m_edge}
\end{figure}
\end{enumerate}

\noindent
The resulting augmentation $\tilde{G} = (\tilde{V}, \tilde{E})$ of $G$ has univalent and trivalent vertices, 
each of the latter being adjacent to one edge of each type $\ell\in\{0,1,2\}$. 
Remark moreover that $\bigcup_{e\in \tilde{E}}e=\bigcup_{e\in E}e:=\Lambda$. 
Hence, a connected region of $\mathbb{C}\setminus \Lambda$ is a polygon with 
angle either $2\pi/3$ or $\pi/3$. 
In the latter case, the vertex $v$ of $G$ is a vertex of degree $4$ which has 
been replaced by two vertices of degree $3$ 
and an edge in $\tilde{G}$ as in Figure \ref{fig:m_edge}.
Hence, any bounded region of $\mathbb{C}\setminus \Lambda$ is bounded by $6$ edges of $\tilde{G}$.
\\
Therefore, the dual of $\tilde{G}$ is a graph $\tilde{H}$ with only triangular faces 
and inner vertices of degree $6$. 
In particular, $\tilde{H}$ is isomorphic to a subgraph of the triangular grid. 
Let us define the type (resp. color, resp. label) of an edge $e$ of $\tilde{H}$ 
as the same as the one his dual. 
In particular, each triangular face is bordered by three edges $(e^0,e^1,e^2)$, 
with $e^\ell$ of type $\ell$ and such that 
$L(e^0)+L(e^1)+L(e^2)=1$. Since the types of the boundary edges is 
$$(\underset{n-d}{\underbrace{0,\ldots,0}},\underset{d}{\underbrace{2,\ldots,2}}, 
\underset{n-d}{\underbrace{1,\ldots,1}}, 
\underset{d}{\underbrace{0,\ldots,0}},\underset{n-d}{\underbrace{2,\ldots,2}}, 
\underset{d}{\underbrace{1,\ldots,1}}) \ ,$$
$\tilde{H}$ is actually isomorphic to $H_{d,n}$. For $e\in E_{d,n}$, set 
$C(e) = \tilde{c} (\tilde{e})$ and $L(e)= \tilde{L}(\tilde{e})$

\begin{lemma}[Honeycomb to dual hive]
    \label{lem:honey_hive_final}
The pair $H=(C,L)$ is a dual hive and $\rho_{C}(H)=h$.
\end{lemma}
\begin{proof}
It is straightforward to check that the maps $C(e)=\tilde{c}(\tilde{e})$, 
$L(e)=\tilde{L}(\tilde{e})$ 
satisfy the properties (1), (2)(a), (2)(b)(i) and (c) of Definition \ref{def:limit_dual_hive}. 
\\
To verify (2)(b)(ii), let $\tilde{e},\tilde{e}'$ be opposite edges of type $\ell$ of a lozenge of 
$H_{d,n}$ with middle edge $\tilde{f}$ of type $\ell+1$ not colored $m$ and denote 
by $e,e',f$ their dual edges in $E^{\tilde{G}}$. 
Suppose that $h(\tilde{e})>h(\tilde{e}')$, where $h$ has been defined in 
\eqref{eq:def_type_height_dual_hive}. 
We want to show that $L(\tilde{e}) = \tilde{L}(e) > L(\tilde{e}') = \tilde{L}(e')$. From the 
definition above, $\tilde{L}(e) = L(e)$ where for $e \in E^G$, $L(e) = x_{\ell}$ 
has been defined in Definition \ref{def:triangular_honey}. We thus need to show that 
$x_{\ell} > x'_{\ell}$. Since $f \in E^{\tilde{G}}$ is of type $\ell + 1$, we have 
that $f \subset x + \R\ed^{2i \pi (\ell+2)/3}$ in $T$. 
The segment $f$ cannot be reduced to a point 
by the definition of a non-degenerate honeycomb and for otherwise condition $(2)$ of 
Definition \ref{def:toric_honeycomb} would not be satisfied. Therefore, the 
coordinate $\ell$ is strictly decreasing between edges $e$ and $e'$ so 
that $x_{\ell} > x'_{\ell}$. Hence, $H\in \dualhive^C(\alpha,\beta,\gamma)$.
\\
Remark that a triangular honeycomb $h=(\mathcal{E},c)$ is uniquely determined by its image 
$t(h)=\bigcup_{e\in \mathcal{E}}e$, 
since then the elements of $\mathcal{E}$ are all the segments of $\mathcal{S}$ whose 
endpoints are univalent or trivalent vertices. 
Recall that for $e = \{v, v'\} \in E^\Gamma$, $\Phi(e) = [x_v, x_{v'}]$.
Hence, to check that $\rho_C(H)=h$, it suffices to 
show that $\bigcup_{e\in E^\Gamma} \Phi(e)=\bigcup_{e\in h} e$. 
This is implied by the construction of $\tilde{\Phi}$ 
at the beginning of the section, since 
$$\bigcup_{e\in E^\Gamma}\Phi(e) = \bigcup_{e\in \tilde{E}}\tilde{\Phi}(e) 
= \bigcup_{ \{v, v'\} \in \tilde{E}} [x_v, x_{v'}]  = \bigcup_{e\in h}e \ .$$
\end{proof}

\noindent
Putting together Lemma \ref{lem:hive_honey_final} with Lemma \ref{lem:honey_hive_final} 
yields the following decomposition.

\begin{proposition}[Color map indexing]
    \label{prop:honey_hive}
    There is a partition 
    $$\honey_{n,d}=\bigsqcup_{C\text{ color map}}\honey_{n,d}^{G^C} \ ,$$
    such that, for each color map $C$, the map $\rho_{C}$ is a bijection and 
    $( \mathcal{L} \circ \rho_{C})^{-1}$ 
    is the restriction of a linear map from $\mathcal{F}_{G^C}(\phi_{\alpha,\beta,\gamma})$ to 
    $\mathbb{R}^{E_{n,d}}$ whose matrix in the canonical bases has integer coefficients.
\end{proposition}
\begin{proof}
    By Lemma \ref{lem:hive_honey_final}, 
    the map $\rho_C:\dualhive^C(\alpha,\beta,\gamma) \rightarrow 
    \honey_{n,d}^{G^C}(\alpha,\beta,\gamma)$ is injective. \\
    Let $h\in\honey_{n,d}^{G^C}(\alpha,\beta,\gamma)$ be a honeycomb. 
    Then, by Lemma \ref{lem:honey_hive_final}, 
    there exists $C'$ a color map and $H\in \dualhive^{C'}(\alpha,\beta,\gamma)$ 
    such that $\rho_{C'}(H)=h$. 
    Hence, $h\in \honey_{n,d}^{G^{C'}}(\alpha,\beta,\gamma)$. 
    Since the map $C\mapsto G^C$ is injective, $C=C'$ and $H\in \dualhive^{C}(\alpha,\beta,\gamma)$. 
    Therefore, $p_{C}$ is a bijection and 
    $\honey_{n,d}^{G^C}(\alpha,\beta,\gamma)\cap \honey_{n,d}^{G^{C'}}(\alpha,\beta,\gamma) 
    =\emptyset$ for $C\not=C'$. 
    Therefore,
    $$\honey_{n,d}=\bigsqcup_{C\text{ color map}}\honey_{n,d}^{G^C} \ .$$
    \noindent
    Finally, for a color map $C$, the map $\rho_C^{-1}$ is then obtained as follows : 
    each segment $e \in h$ of type $\ell$ corresponds to a maximal chain $\hat{e}$ 
    of type $\ell$ of $\Gamma_{n,d}$ 
    with respect to $C$. Hence, for all edge $f\in \hat{e}$ of type $\ell$, one has 
    $\rho_C^{-1}[h](f)=L(e)$. 
    Then, for any edge $f\in E_{n,d}$ colored $m$, one has $\rho_C^{-1}[h](f) = 
    1-\rho_C^{-1}[h](f_1)-\rho_C^{-1}[h](f_2)$, 
    where $f_1$ and $f_2$ are the two other edges of a triangular face bordered by $f$. 
    Hence, the matrix of the map $(\mathcal{L}\circ \rho_{C})^{-1}$ has integer 
    coordinates in the canonical bases.
\end{proof}

\noindent
From the Proposition \ref{prop:honey_hive}, any $G\in\mathcal{G}_d$ 
is of the form $G=G^C$ for some color map $C$. 
Since $V^C\subset V^\Gamma$, any vertex $v$ of $G^C$ inherits the coordinates of 
$\Gamma$ by setting 
$$(v_0,v_1,v_2)=(h(e_0),h(e_1),h(e_2)) \ ,$$ 
where $e_\ell$ is the edge of type $\ell$ adjacent to $v$ in $\Gamma_{d,n}$ 
(even if $e_\ell\not\in E^C$). For $e=\{v,w\},e'=\{v',w'\}\in E^C$ of same type, 
introduce the cover relation $e<e'$ when, up to a transposition, 
$\{w,v'\}$ is an edge of $E^C$ of type $\ell+1$ and $v'_{\ell-1}>w_{\ell-1}$. 
This cover relation translates into a cover relation in $\vec{E}^C$ by saying 
that $(v,w)<(v',w')$ is and only if $\{v,w\}<\{v',w'\}$ in the former sense.

\begin{corollary}[Parametrization of labels]
    \label{cor:paramietrization_label}
    For a color map $C$ and $G^C\in \mathcal{G}_d$,
    $$\mathcal{L} \left(\honey_{n,d}^{G^C}(\alpha,\beta,\gamma) \right) 
    = \mathcal{F}_{G^C}(\phi_{\alpha,\beta,\gamma})\cap K_{G^C} \ ,$$
    where $K_{G^C}\subset \Omega^1(G^C)$ is the cone defined as 
    $$K_{G^C} = \left\{\omega \in \Omega^1(G^C) \mid 
    \vert \omega(e)\vert<\vert \omega(e')\vert\,\text{ if } e<e' \right\}.$$ 
\end{corollary}

\begin{proof}
    Let us define $\Psi: \Omega^1(G) \to \R^{E_{n, d}}$ by 
    $$\Psi[\omega](e)=\left\lbrace\begin{aligned} \vert \omega(\vec{\hat{e}})\vert 
        \quad &\text{ if }e\in \hat{e}, c(e)\not=m\\
    1-\vert \omega(\vec{\hat{e_1}})\vert-\vert \omega(\vec{\hat{e_2}})\vert 
    \quad &\text{ if }c(e)=m,\,e_1\in\hat{e}_1,e_2\in\hat{e}_2 ,\,(e_1,e_2,e) 
    \text{ triangular face of }H_{n,d}\end{aligned}\right..$$
    By Proposition \ref{prop:honey_hive}, we have that  
    $\mathcal{L} \left(\honey_{n,d}^{G^C}(\alpha,\beta,\gamma) \right) 
    = \Psi^{-1} \left(\dualhive^C(\alpha,\beta,\gamma)\right)$.

    \noindent
    Then, remark that $\dualhive^C(\alpha,\beta,\gamma) 
    = \mathcal{H}^C(\alpha,\beta,\gamma) \cap \mathcal{K}_{<},$ 
    where $\mathcal{H}^C(\alpha,\beta,\gamma)\subset \mathbb{R}^{E_{n,d}}$ 
    is the vector subspace determined by the conditions 
    (2)(a), (2)(b)(i) and (2)(c) of Definition \ref{def:limit_dual_hive} 
    and $\mathcal{K}_{<}$ is the cone given by 
    $$\mathcal{K}_{<} = \left\{(H(e))_{e\in E_{n,d}} \mid L(e)< L(e')\text{ if } e ,\,e' 
    \text{ satisfy condition (2)(b)(ii) of Definition \ref{def:limit_dual_hive}} \right\}.$$
    Hence, 
    $$\mathcal{L}\left(\honey_{n,d}^{G^C}(\alpha,\beta,\gamma)\right) = 
    \Psi^{-1} \left(\mathcal{H}^C(\alpha,\beta,\gamma) \cap \mathcal{K}_{\leqslant} \right)
    =\Psi^{-1}\left(\mathcal{H}^C(\alpha,\beta,\gamma)\right) \cap 
    \Psi^{-1}\left(\mathcal{K}_{<}\right).$$
    One then checks that $\Psi^{-1}(\mathcal{H}^C(\alpha,\beta,\gamma)) = 
    \mathcal{F}_{G^C}(\alpha,\beta,\gamma)$ 
    and $\Psi^{-1}(\mathcal{K}_{<})=K_{G^C}$. 
\end{proof}

\begin{remark}[Number of vertices and edges]\label{rmk:graph_vertices_edges}
By Proposition \ref{prop:honey_hive} and the construction of $G^C$ from a coloured map $C$, we deduce that all the graphs $G^C$ have the same number $n^2+3n$ of vertices and the same number $\frac{3n(n+1)}{2}$ of edges : indeed, vertices correspond to either triangles of the dual hive model which are not neighbours to an edge coloured $m$ or to boundary edges, and edges correspond to edges of the dual hive model which are not coloured $m$. Since it has been proven in \cite{françois2024positiveformulaproductconjugacy} that there are always $d(n-d)$ edges coloured $m$ in a dual hive, the result is deduced.
\end{remark}

\subsection{Volume of flat connections}
\label{subsec:main_th_three_holed_sphere}

We can now combine the results of \autocite{françois2024positiveformulaproductconjugacy} 
with the ones of the previous section to prove Theorem \ref{th:Z_g_p_0} in the 
case of the three-holed sphere, that is, for $(g, p) = (0, 3)$.
Let us denote by $\Sigma_0^3$ the sphere with three generic marked points removed. 
The moduli space of flat $\SU(n)$ connections can be described as 
$$M_{0, 3}(\alpha,\beta,\gamma) = 
\{(U_1,U_2,U_3)\in \mathcal{O}_{\alpha} \times 
\mathcal{O}_{\beta}\times \mathcal{O}_{\gamma} \mid U_1U_2U_3=Id_{\SU(n)}\}/\SU(n) \ ,$$
where $\SU(n)$ acts diagonally by conjugation on each factor. 
Its volume has been computed in \autocite{françois2024positiveformulaproductconjugacy} 
using two equivalent 
models named \textit{toric hives} and \textit{dual hives}. 
Using the results of this section, we present a reformulation of this results 
in terms of triangular honeycombs. For $G \in \mathcal{G}_d$, let us set  
\begin{equation}
    \label{eq:def_volume_honey_with_L}
    \Vol\left[\honey_{n,d}^G(\alpha,\beta,\gamma)\right] \coloneqq 
    \Vol \left[ \mathcal{L}( \honey_{n,d}^G(\alpha,\beta,\gamma)) \right] \ ,
\end{equation}
where $\mathcal{L} : \honey_{n,d}^G(\alpha,\beta,\gamma) \to 
\mathcal{F}_G(\phi_{\alpha,\beta,\gamma})$ was defined in Definition 
\ref{def:flow_honeycomb} and $\Vol$ is the volume form defined 
in \eqref{eq:definition_volume_form}.

\begin{theorem}[Volume of flat $\U(n)$-connections on the three-holed sphere]
    \label{th:volume_flat_connection_0_3}
    Let $n\geqslant 3$ and consider the canonical volume form on $\U(n)$. 
    Then, for $\alpha,\beta,\gamma\in \mathcal{H}_{reg}$, 
    \begin{equation*}
        Z_{0, 3}(\alpha,\beta,\gamma) \coloneqq 
        \Vol \left[ M_{0, 3}(\alpha,\beta,\gamma) \right] \not=0 
    \end{equation*}
    only if $\sum_{i=1}^n\alpha_i+\sum_{i=1}^n\beta_i+\sum_{i=1}^n\gamma_i=n+d$ 
    for some $d\in\mathbb{N}$, in which case, 
    if $\widetilde{\gamma}=(1-\gamma_n,\ldots,1-\gamma_1)$,
    \begin{equation*}
        Z_{0, 3}(\alpha,\beta,\gamma) = 
        \frac{2^{(n+1) [2]}(2\pi)^{(n-1)(n-2)}} 
        {n! \Delta(\alpha) \Delta(\beta) \Delta(\gamma)} 
        \sum_{G\in\mathcal{G}_d} 
        \Vol \left[\honey_{n,d}^G(\alpha,\beta,\widetilde{\gamma})\right] \ ,
    \end{equation*}
    where for $\alpha \in \mathcal{H}_{reg}$, $\Delta(\alpha) 
    = 2^{n(n-1)/2} \prod_{i < j} \sin (\pi(\alpha_i - \alpha_j))$.
\end{theorem}

\noindent
Before proving this theorem, let us recall three results 
from \autocite{françois2024positiveformulaproductconjugacy}. 
\begin{enumerate}
\item For any pair $(C,C')$ of color maps, there exists a linear 
isomorphism $Rot[C\rightarrow C']$ 
from $\dualhive^C(\alpha,\beta,\gamma)$ to 
$\dualhive^{C'}(\alpha,\beta,\gamma)$ with integer coefficients 
on the canonical bases and such that 
$Rot[C\rightarrow C]=Id$ and $Rot[C\rightarrow C']^{-1}=Rot[C'\rightarrow C]$.
\item There exists a color map $C_0$ and a subset $S\subset E_{n,d}$ 
such that $p_S: \dualhive^{C_0}(\alpha,\beta,\gamma) \to \R^S$ which to a label map 
$L: E_{n, d} \to \R$ associates $(L(e))_{e \in S}$ is an isomorphism such that both the map and its inverse have integer coefficients in the canonical bases indexed by $E_{n,d}$ and $S$. 
\item We have the formula 
$$Z_{0, 3}(\alpha,\beta,\gamma) = 
    \frac{2^{(n+1) [2]}(2\pi)^{(n-1)(n-2)}} 
    {n! \Delta(\alpha) \Delta(\beta) \Delta(\gamma)} 
    \sum_{C} 
    \Vol_S\left[Rot[C\rightarrow C_0] \left(\dualhive^C(\alpha,\beta,\gamma)\right) \right],$$
    where the sum is over color maps $C: E_{n, d} \to \{0, 1, 3, m\}$ 
    and where $\Vol_S$ is the Lebesgue measure of dimension $|S|$.
\end{enumerate}
In the following proof, let us call an {\it integral} linear map $f:\mathbb{R}^E\rightarrow \mathbb{R}^F$ a linear map with integer coefficients in the canonical bases. Likewise, an integral isomorphism is an invertible linear map such that both a map and its inverse has integers coefficients in the canonical bases.
\begin{proof}[Proof of Theorem \ref{th:volume_flat_connection_0_3}] 
Let $C$ be a color map. By Proposition \ref{prop:BijectionTrees}, 
there exists a set $R \subset E$ such the restriction map $\varphi_{R} : 
\mathcal{F}_{G^C}(\phi_{\alpha,\beta,\gamma}))\rightarrow \Omega^1(R)$
 is an integral isomorphism. 
Denote by $i_{R}:\Omega^1(R)\rightarrow \Omega^1(G)$ the corresponding inverse map, 
which has thus affine with integer coefficients in the canonical basis 
and is a bijection from $\Omega^1(R)$ to $\mathcal{F}_{G^C}(\phi_{\alpha,\beta,\gamma})$. 
Since, by Proposition \ref{prop:honey_hive} and Corollary \ref{cor:paramietrization_label}, 
$\Psi^C:\Omega^1(G^C)\rightarrow \mathbb{E}^{E_{n,d}}$ is an affine integral map 
and a bijection from $\mathcal{F}_{G^C}(\phi_{\alpha,\beta,\gamma}))\cap K_{G^C}$ to $\dualhive^C
(\alpha,\beta,\gamma)$ and by (1) above, $Rot[C\rightarrow C_0]$ is an integral isomorphism from $\dualhive^C(\alpha,\beta,\gamma)$ 
to $\dualhive^{C_0}(\alpha,\beta,\gamma)$. We deduce that 
$$p_{S}\circ Rot[C\rightarrow C_0]\circ \Psi^C\circ i_R: \Omega^1(R) \rightarrow \R^S$$
is an integral isomorphism. 
Likewise, since by (2) above $p_S^{-1}:\mathbb{R}^S\rightarrow 
\dualhive^{C_0}(\alpha,\beta,\gamma)$ 
is an integral affine isomorphism and $\mathcal{L} \circ \rho_{C}$ is an 
integral affine isomorphism from 
$\dualhive^{C}(\alpha,\beta,\gamma)$ to $\mathcal{F}_{G^{C}}(\phi_{\alpha,\beta,\gamma})$,
$$ F \coloneqq \left(p_{S}\circ Rot[C\rightarrow C_0]\circ \Psi^C\circ i_R \right)^{-1} 
= \varphi_{R}\circ (\mathcal{L} \circ \rho_{C}) 
\circ Rot[C_0\rightarrow C] \circ p_S^{-1}: \mathbb{R}^S\rightarrow \Omega^1(R)$$
is an integral isomorphism. We deduce that  
its determinant as an isomorphism from $\Omega^1(R)$ to $\R^S$ has modulus one.
Hence,
\begin{align*}
\Vol_S\left[Rot[C\rightarrow C_0] \left(\dualhive^C(\alpha,\beta,\gamma)\right) \right] = 
& \Leb \left[u\in\mathbb{R}^S \mid Rot[C_0\rightarrow C]\circ p_{S}^{-1}(u) 
\in \dualhive^C(\alpha,\beta,\gamma) \right]
\\
=& \Leb\left[u\in\mathbb{R}^S \mid F(u) 
\in \varphi_R \left(\honey_{n,d}^C(\alpha,\beta,\gamma)\right) \right]\\
=& \Leb\left[z\in\Omega^1(R) \mid z\in \varphi_R\left(\honey_{n,d}^C(\alpha,\beta,\gamma)\right) \right]\\
=& \Vol \left[\honey_{n,d}^{G^C}(\alpha,\beta,\gamma)\right],
\end{align*}
where we used Proposition \ref{prop:divergence_volume} on the last equality. Hence, by (3),
\begin{align*}
Z_{0, 3}(\alpha,\beta,\gamma) &=
    \frac{2^{(n+1) [2]}(2\pi)^{(n-1)(n-2)}} 
    {n! \Delta(\alpha) \Delta(\beta) \Delta(\gamma)} 
    \sum_{C:E_{n,d}\rightarrow \{0,1,3,m\} \text{ color map}} 
     \Vol_S\left[Rot[C\rightarrow C_0] \left(\dualhive^C(\alpha,\beta,\gamma)\right) \right]\\
    &=\frac{2^{(n+1) [2]}(2\pi)^{(n-1)(n-2)}} 
    {n! \Delta(\alpha) \Delta(\beta) \Delta(\gamma)} 
    \sum_{C:E_{n,d}\rightarrow \{0,1,3,m\} \text{ color map}} 
    \Vol \left[\honey_{n,d}^{\Gamma^C}(\alpha,\beta,\gamma)\right]\\
    &=\frac{2^{(n+1) [2]}(2\pi)^{(n-1)(n-2)}} 
    {n! \Delta(\alpha) \Delta(\beta) \Delta(\gamma)} 
    \sum_{G\in\mathcal{G}_d} 
    \Vol\left[\honey_{n,d}^{G}(\alpha,\beta,\gamma)\right],
\end{align*}
where we used Proposition \ref{prop:honey_hive} for the last equality.
\end{proof}

\begin{remark}[Volume correspondance]
    \label{rmk:volume_correspondance}
    Recall that in \eqref{eq:parametrization_P}, we introduced the parametrization map 
    \begin{equation*}
        P: \honey^G(\alpha, \beta, \gamma) \to \R^E \ ,
    \end{equation*}
    which to a honecomb $h$ associates $P[h] = (|\mathcal{L}[h](e)|)_{e \in E}$. 
    For $S \subset E$, let us consider the map 
    \begin{align*}
        \eta_S: \Omega^1(S) & \rightarrow \R^S \\
        w & \mapsto (|w(e)|)_{e \in S} \ .
    \end{align*}
    Then, the pullback $(\eta_S)^* \diff \Vol$ of $\diff \Vol$ defined 
    in \eqref{eq:volume_R_E} on $\R^S$ viewed as a subspace of $\R^E$ satisfies 
    \begin{equation}
        (\eta_S)^* \diff \Vol = \Leb_{\Omega^1(S)} = (\varphi_S)_* \diff \Vol' \ ,
    \end{equation}
    where $\diff \Vol'$ is the volume form defined 
in \eqref{eq:definition_volume_form} and where we used \eqref{eq:pushforward_phi_S} 
for the last equality.
\end{remark}

\section{Sieving of honeycombs}
\label{sec:sieving_honeycombs}

Let $\mathcal{T}$ be an oriented surface with boundary obtained 
by gluing $N$ equilateral triangles $T^1,\ldots,T^N$ by pairs along some of their boundaries in an orientation reversing way, 
such that $p$ edges $L_1,\ldots,L_{p}$ of the equilateral triangles are not glued together. For each $1\leqslant i\leqslant N$, let $f_i:T\rightarrow \mathcal{T}$ be an orientation-preserving isometry from $T$ to $T^i$ (viewed as a subset of $\mathcal{T}$). For $1\leqslant j\leqslant p$, let $1 \leqslant s_j \leqslant N $ and 
$ 0 \leqslant \ell_j \leqslant 2$ 
be such that $L_j=\partial_{\ell_j}T^{s_j}$. Each edge $L_j$ has a natural orientation $\ell_j$ coming from the unique equilateral it belongs to.  

\noindent
Recall that a honeycomb on $\mathcal{T}$ is defined as a sequence $(h^i)_{1\leqslant i\leqslant N}$ of triangular honeycombs such that $t_i(h^i)_{\vert \partial_j T^i}=t_{i'}(h^{i'})_{\vert \partial_{j'}T^{i'}}$ whenever $\partial_jT^i$ is glued to $\partial_{j'}T^{i'}$, where for a triangular honeycomb $h=(\mathcal{E},c)$, $t_i(h)=\bigcup_{e\in\mathcal{E}}f_i(e)$. 

Let $h=(h^i)_{1\leqslant i\leqslant N}$ be such a honeycomb and denote by $G_i=(V_i,E_i)$ the structure graph of $h^i$ for $1\leqslant i\leqslant N$. Then, the structure graph $G(h)$ of $h$, as defined after 
Definition \ref{def:(g,p)_honey}, can be describe as 
$$G(h)=\left(\bigcup_{i=1}^NG_i\right)_{S*S'}$$ 
in the sense of Definition \ref{def:sieving}, 
where $S,S'\subset \bigcup_{i=1}^m\partial V_i$ are the vertices 
which are identified together. 
Hence, if $\mathcal{G}_\mathcal{T}$ denotes the set of structure graphs appearing in 
$\honey_{\mathcal{T}}$, there is an injective map
$$i:\mathcal{G}_{\mathcal{T}}\rightarrow  \mathcal{G}^N,$$
where $\mathcal{G}$ is the set of structure graphs appearing in the case $N=1$ and $i(G)$ is the tuple $(G_1,\ldots,G_N)$ such that $G=\left(\bigcup_{i=1}^NG_i\right)_{S*S'}$.
Moreover, $G(h)$
has $pn$ univalent vertices, 
$n$ of them being on each boundary component 
$L_j,\, 1\leqslant j\leqslant p$, and $\frac{(3N-p)n}{2}$ bivalent vertices, 
$n$ of them being on a same boundary of a triangle while being 
in the interior of the surface. 
For $1 \leqslant j \leqslant p$, let us denote by $(v^j_m)_{1\leqslant m\leqslant n}$ 
the univalent boundary vertices on $L_{j}$ ranked decreasingly with respect to 
their $(\ell_j+1)$-coordinate.

In Section \ref{subsec:parametrization_honeycombs}, 
we defined for each $G_i=(V_i,E_i)\in \mathcal{G}_{d_i}$ 
a map $s_i:V_i\rightarrow \{-1,1\}$ such that for each edge 
$\{v,v'\}\in E_i$, $s_i(v)s_i(v')=-1$. 
Let us denote by $\interior_3(V)$ the set of vertices of degree 
$3$ of $G$. By the construction of $G$ in terms of 
the graphs $G_i$ and the fact that all vertices of 
$\interior(V_i)$ have degree $3$, 
$$\interior_3(V)=\bigsqcup_{1\leqslant i\leqslant N}\interior(V_i) \ .$$
We can therefore extend the maps $s_i$ defined on each 
$G_i$ to a map $s:\interior_3(V) \to \{-1,1\}$.

\begin{definition}[Label map of a honeycomb]
    \label{def:map_L_arbitrary_honeycomb}
Let $G=(V,E)\in\mathcal{G}_{\mathcal{T}}$ and $h=(h_i)_{1\leqslant i\leqslant N}\in \honey_{\mathcal{T}}^G$, with $h_i=(\mathcal{E}_i,c_i)$ for $1\leqslant i\leqslant N$. Let $\phi:G(h)\rightarrow G$ be the corresponding graph isomorphism. Then, the \textit{label map} $\mathcal{L}[h]$ is the element of $\Omega^1(G)$ whose value at $(v,v') \in \vec{E}$ such that $v\in \interior_3(V)$ and $\{v,v'\}=\phi(\partial e)$ for $e\in \mathcal{E}_i$ is
$$\mathcal{L}[h](v,v')=s(v)L(e) \ ,$$
where $L$ is the coordinate 
map defined in Definition \ref{def:triangular_honey}. (2) for $h_i$.
\end{definition}
Remark that each edge $e \in E$ of $G$ is adjacent to at 
least one trivalent vertex, and if $e=\{v,v'\}$ with 
$v,v'\in \interior_3(V)$, then necessarily there exists 
$1\leqslant i\leqslant N$ such that $v,v'\in V_i$. 
Therefore $s(v)s(v')=-1$, so that $\mathcal{L}[h] \in \Omega^1(G)$. Hence, the latter construction yields a well-defined map 
$$\mathcal{L}:\honey_{\mathcal{T}}^G 
\rightarrow \Omega^1(G).$$

\begin{figure}[H]
    \centering
    \includegraphics[scale=0.6]{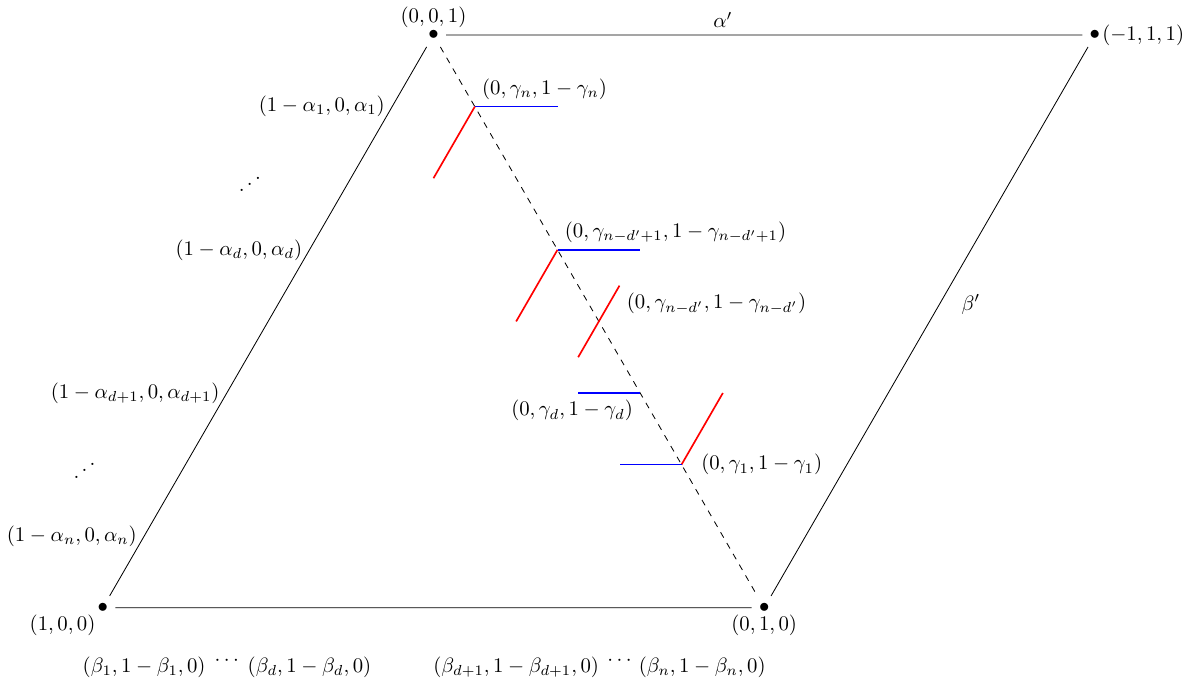}
    \caption{Line segments on the common edge of $h$ and the rotation of $h'$ 
    in $\R^3_{\Sigma=1}$. 
    Here, $d < d'$ which can be read from the orientation of line segments.}
    \label{fig:common_boundary} 
\end{figure}

\noindent
Let $\gamma^{(1)},\ldots,\gamma^{(p)}\in\mathcal{H}_{reg}$, 
and recall that $\honey_{\mathcal{T}}^G(\gamma^{(1)}, \ldots,\gamma^{(p)})$ 
denotes the set of honeycombs $h$ on $\mathcal{T}$ having boundary condition 
$$(v^j_s)_{\ell_j}=\gamma^{(j)}_s.$$
Let us define $\phi_{\gamma^{(1)},\ldots,\gamma^{(p)}}: 
V\rightarrow\mathbb{R}$ by 
\begin{equation}
    \label{eq:def_phi_p}
    \phi_{\gamma^{(1)},\ldots,\gamma^{(p)}}(v)=\left\lbrace\begin{aligned}
    &s(v)&\text{ if }& \deg v=3\\
    &1-c(v,v')-c(v,v'')&\text{ if } 
    & \deg v=2, v\sim v', v\sim v'', v'\not=v''\\
    &\gamma^{(j)}_s &\text{ if } & v=v^j_s, v\sim v'  \text{ and } c(v,v')=0 \\
    &\gamma^{(j)}_s - 1 & \text{ if } & v=v^j_s, v\sim v' \text{ and } c(v,v')=1
    \end{aligned} \right..
\end{equation}

\begin{lemma}[Flow of honeycomb]
    \label{lem:differential_gluing_honey}
    For any $G\in\mathcal{G}_{\mathcal{T}}$, $\mathcal{L}$ 
    is injective and for $\gamma^{(1)},\ldots,\gamma^{(p)} 
    \in\mathcal{H}_{reg}$, 
    $$\mathcal{L}(\honey_{\mathcal{T}}^G(\gamma^{(1)}, 
    \ldots,\gamma^{(p)})) = 
    \mathcal{F}_{G}(\phi_{\gamma^{(1)},\ldots,\gamma^{(p)}}) 
    \cap K_G \ ,$$
    where $i(G)=(G_1,\ldots,G_N)$, $\Omega^1(G)$ 
    is identified with $\prod_{i=1}^N\Omega^1(G_i)$ and 
    $$K_G=K_{G_1}\times \dots\times K_{G_N},$$
    with $K_{G_i}$ defined in Corollary \ref{cor:paramietrization_label}.
\end{lemma}

Let us introduce some notations to prove this lemma. Set $n(\mathcal{T})=\frac{3N-p}{2}$ and denote 
by $B_{1},\ldots,B_{n(\mathcal{T})}$ the segments of $\mathcal{T}$ 
corresponding to boundary of triangles identified together. 
Each $B_j$ is then adjacent to two triangles $T^{r^{(1)}_j}$ and $T^{r^{(2)}_j}$ 
and there exists $\ell^{(1)}_j,\ell^{(2)}_j$ such that 
$B_j=\partial_{\ell^{(1)}_j}T^{r^{(1)}_j}=\partial_{\ell^{(2)}_j}T^{r^{(2)}_j}$.
\noindent

\begin{proof}
    Let $G = (V, E)\in\mathcal{G}_T$. 
    First, by definition, 
    $$\honey_{\mathcal{T}}^G= \left\{(h^1,\ldots,h^N) \in 
    \honey_{T^{1}}^{G_1}\times\ldots\times \honey_{T^N}^{G_N} 
    \mid h^{r_j^{(1)}}_{\vert B_j } = 
    h^{r_j^{(2)}}_{\vert B_j}, 1\leqslant j\leqslant n(\mathcal{T}) \right\}.$$
    Next, recall that  
    $$V=\bigcup_{i=1}^NV_i \, / \left\langle 
    \partial_{\ell_j}G^{r_j^{(1)}}=\partial_{\ell'_j}G^{r_j^{(2)}} 
    , 1\leqslant j\leqslant n(\mathcal{T}) \right\rangle$$ 
    and that $E$ is the image of $\bigcup_{i=1}^N E_i$ 
    in this quotient. 
    Since any element of $\bigcup_{i=1}^N E_i$ has at 
    most one endpoint in $\bigcup_{j=1}^{n(\mathcal{T})} 
    \partial_{\ell_j}G^{r_j^{(1)}} \cup \partial_{\ell'_j}G^{r_j^{(2)}}$, 
    there is a canonical identification 
    $\Omega^1(G)=\prod_{j=1}^N\Omega(G^j)$. 
    Moreover, by the definition of $\mathcal{L}: 
    \honey_G^{\mathcal{T}}\to \Omega^1(G)$ from 
    Definition \ref{def:map_L_arbitrary_honeycomb}, 
    for $h\in \honey_G^{\mathcal{T}}$
    $$\mathcal{L}[h]=\left(\prod_{i=1}^N\mathcal{L}\right)(h^1,\ldots,h^N)$$
    under the previous identification. 
    The injectivity of 
    the map $\mathcal{L}$ from 
    Proposition \ref{prop:parametrization_triangular_honey_first} 
    yield the injectivity of $\mathcal{L}$.
    \\
    \\
    It remains to describe the image 
    $\mathcal{L}\left(\honey_{\mathcal{T}}^G
    (\gamma^{(1)},\ldots,\gamma^{(p)})\right)$. 
    By the previous reasoning, 
    this amounts to describe the image throught
    $\prod_{i=1}^N\mathcal{L}$ of the set
    $$\left\{(h^1,\ldots,h^N) \in 
    \honey_{T^{1}}^{G_1}\times\ldots\times \honey_{T^N}^{G_N} 
    \mid h^{r_j^{(1)}}_{\vert B_j } = 
    h^{r_j^{(2)}}_{\vert B_j}, 1\leqslant j\leqslant n(\mathcal{T}) \right\}.$$
    Let us first describe how the condition 
    $\partial_{\ell_j}h^{s_j}=\gamma^{(j)}$ for 
    $1\leqslant j\leqslant p$ translates through $\mathcal{L}$. 
    Let $v^j_m\in \partial_{\ell_j}G_{s_j}$ and 
    $x\in \partial_{\ell_j}h^{s_j}$ the point such that 
    $\iota(v^j_m) = x$. Since $h^{s_j}\in \honey_{G_{s_j}}$, 
    like in Section \ref{subsec:parametrization_honeycombs}, 
    the condition $x_{\ell_j}=\gamma^{(j)}_m$ is equivalent 
    to the condition 
    $$\mathcal{L}[h](v^j_m,v)=
    \gamma^{(j)}_m-c(v^j_m,v) = \phi_{\gamma^{(1)},\ldots,\gamma^{(p)}}(v) \ ,$$
    where $v$ is the unique element of $V_{s_{j}}$ 
    such that $v\sim v^j_m$.
    \\
    \\
    Let us now consider the condition $h^{r_j^{(1)}}_{\vert B_j } = 
    h^{r_j^{(2)}}_{\vert B_j}$. 
    Let $v\in \partial_{\ell^{(1)}_j}G_{r^{(1)}_j}$ 
    and $v'\in \partial_{\ell^{(2)}_j}G_{r^{(2)}_j}$ 
    such that $v\sim v'$ in $G$, and let 
    $x=\iota_{T^{r^{(1)}_j}}(v), x'=\iota_{T^{r^{(2)}_j}}(v')$ 
    the corresponding image in $h_{\vert T^{r^{(1)}_j}}$ 
    and $h_{\vert T^{r^{(2)}_j}}$. 
    Using that the edges of both triangles are identified in 
    the order-reversing way, the condition 
    $h^{r_j^{(1)}}_{\vert B_j } = 
    h^{r_j^{(2)}}_{\vert B_j}$ implies that 
    $$x_{\ell^{(1)}_j-1}=1-x'_{\ell^{(2)}_j-1} \ .$$
    Let $w\in V_{r^{(1)}}$ and $w'\in V_{r^{(2)}}$ 
    such that $v\sim_{G_{r^{(1)}}}w$ and $v'\sim_{G_{r^{(2)}}}w'$. 
    Then, following \eqref{eq:honeycomb_divergence_2}, 
    the previous equality is equivalent to 
    $\mathcal{L}[h^{r^{(1)}}](v,w)+c(v,w) = 
    1-\mathcal{L}[h^{r^{(1)}}](v,w)-c(v',w')$, 
    which yields
    $$\mathcal{L}[h](\bar{v},w)+\mathcal{L}[h](\bar{v},w') = 
    1-c(v,w)-c(v',w')= \phi_{\gamma^{(1)},\ldots,\gamma^{(p)}}(v) \ .$$
\end{proof}

\noindent
In order to state the volume formula for $(g,p)$ honeycombs, 
let us introduce for $r\in[0,1)$ the notation 
$$\mathcal{H}_{reg}^{r} = \left\{\gamma\in\mathcal{H}_{reg} \mid 
\sum_{i=1}^n\gamma_i=r\mod \mathbb{Z}\right\}.$$
This set is a union of affine polytopes of $\mathcal{H}_{reg}$ 
of dimension $n-1$, and one can check that the volume $d_Ju$ 
on each affine polytope induced by the projection on $\mathbb{R}^J$ 
is independent of $J$ for any $J\subset\{1,\ldots,n\}$ of 
cardinal $n-1$. We simply denote by $\diff u$ this volume form.

\begin{proposition}
    \label{prop:formula_gluing_honey}
    Suppose that $\mathcal{T}'$ is obtained by 
    gluing $\mathcal{T}$ and $T^{N+1}$ along the boundaries 
    $L_p$ and $\partial_0 T^{N+1}$. 
    Then, for all $G\in \mathcal{T}'$ and 
    $\gamma^1,\ldots,\gamma^{p+1}\in \mathcal{H}_{reg}$ such that 
    $\sum_{i=1}^{p+1}\vert \gamma_i\vert\in \mathbb{N}$,
    \begin{align*}
    \sum_{G\in\mathcal{G}_{\mathcal{T}'}} & 
    \Vol \left[ \mathcal{L}\left(\honey_{\mathcal{T}'}^{G}(\gamma^1,\ldots,\gamma^{p+1})\right) \right]\\
    =& \int_{\mathcal{H}_{reg}^{\theta} } 
    \sum_{(G_1,G_2)\in\mathcal{G}\times\mathcal{G}_{\mathcal{T}}} 
    \Vol \left[ \mathcal{L}\left(\honey_{\mathcal{T}}^{G_2}(\gamma^1,\ldots,\gamma^{p-1},u)\right) \right] 
    \Vol \left[ \mathcal{L}\left(\honey^{G_1}(\tilde{u},\gamma^p,\gamma^{p+1})\right) \right] \diff u,
    \end{align*}
    where $\tilde{u}=(1-u_n,\ldots,1-u_1)$ and $\theta=-\sum_{i=1}^{p-1}\vert \gamma^i\vert\mod \mathbb{Z}$.
\end{proposition}

\begin{proof}
Remark that from \eqref{eq:def_phi_p},
$$\sum_{v\in V}\phi_{\gamma^{(1)},\ldots,\gamma^{(p)}}(v) = 
\sum_{j=1}^p\sum_{i=1}^n\gamma_s^{(j)}+\sum_{i=1}^N 
\left(\sum_{v\in \interior(V_i)}s(v)-\sum_{v\in\partial V_i, v\sim v'} 
c(v,v')\right).$$
Hence, in order for $\phi_{\gamma^{(1)},\ldots,\gamma^{(p)}}$ 
to yields a non-empty set $\mathcal{F}_{G}(\phi_{\gamma^{(1)}, 
\ldots,\gamma^{(p)}})$, 
by Proposition \ref{prop:BijectionTrees} it is necessary that 
$\sum_{i=1}^p\vert\gamma^{(i)}\vert\in \mathbb{N}$. 
Hence, $\honey_{\mathcal{T}}^{G_2}(\gamma^1,\ldots,\gamma^{p-1},u)$ 
is non-empty if and only if $u\in\mathcal{H}_{reg}^\theta$ with 
$\theta-\sum_{i=1}^{p-1}\vert \gamma^i\vert\mod \mathbb{Z}$. 
Using Proposition \ref{prop:divergence_product} (2) gives
\begin{align*}
\int_{\mathcal{H}_{reg}^{\theta} } 
    & \Vol \left[ \mathcal{L}\left(\honey_{\mathcal{T}}^{G_2}
    (\gamma^1,\ldots,\gamma^{p-1},u)\right) \right] 
    \Vol \left[ \mathcal{L}\left(\honey^{G_1} 
    (\tilde{u},\gamma^p,\gamma^{p+1})\right) \right] \diff u \\
&= \Vol\left[(K_{G_1}\times K_{G_2})\cap 
\mathcal{F}_{(G_1 \cup G_2)_{S_1 * S_2}} 
\left(\phi_{\gamma^{(1)},\ldots,\gamma^{(p+1)}}\right)\right] \ .
\end{align*}
By Lemma \ref{lem:differential_gluing_honey}, 
the latter volume is exactly
\begin{equation*}
    \Vol \left[\mathcal{L}\left( \honey_{\mathcal{T}}^G \left(\gamma^{(1)}, 
    \ldots,\gamma^{(p+1)}\right) \right) \right] \ .
\end{equation*}
Therefore, 
\begin{align*}
\sum_{(G_1,G_2)\in\mathcal{G}\times\mathcal{G}_{\mathcal{T}}} 
&\int_{\mathcal{H}_{reg}^{\theta} } 
     \Vol \left[ \mathcal{L}\left(\honey_{\mathcal{T}}^{G_2}
    (\gamma^1,\ldots,\gamma^{p-1},u)\right) \right] 
    \Vol \left[ \mathcal{L}\left(\honey^{G_1} 
    \left(\tilde{u},\gamma^p,\gamma^{p+1}\right)\right) \right] \diff u\\
&= \sum_{G\in\mathcal{G}_{\mathcal{T}'}} 
\Vol \left[\mathcal{L}\left( \honey_{\mathcal{T}}^G \left(\gamma^{(1)}, 
    \ldots,\gamma^{(p+1)}\right) \right) \right] \ .
\end{align*}
\end{proof}

\noindent
A similar reasoning using Proposition \ref{prop:divergence_contraction} 
yields the following proposition.
\begin{proposition}
    \label{prop:formula_contracting_honey}
Suppose that $\mathcal{T}'$ is obtained by gluing two boundaries of $\mathcal{T}$. 
Then, for each $G\in \mathcal{G}_{\mathcal{T}'}$, $\honey_{\mathcal{T}'}^{G}$ admits 
a volume form with, for $\gamma^1,\ldots,\gamma^{p-2}\in \mathcal{H}_{reg}$ 
with $\sum_{i=1}^{p-2}\vert \gamma^i\vert\in\mathbb{Z}$,
\begin{align*}
\sum_{G\in\mathcal{G}_{\mathcal{T}'}} 
\Vol \left[\mathcal{L}\left(\honey_{\mathcal{T}'}^G \left(\gamma^{(1)}, 
    \ldots,\gamma^{(p-2)}\right)\right)\right] = 
    \int_{\mathcal{H}_{reg}} 
    \sum_{G\in\mathcal{G}_{T}} 
    \Vol \left[\mathcal{L}\left(\honey_{\mathcal{T}}^G\left(\gamma^{(1)}, 
    \ldots,\gamma^{(p-2)}, u , \tilde{u}\right)\right)\right] \diff u \ .
\end{align*}
\end{proposition}

\section{Proof of Theorem \ref{th:Z_g_p_0} : 
volume of flat $\U(n)$-connections on a compact surface}
\label{sec:proof_of_th_flat_connections}

\noindent
The goal of this section is the proof of Theorem \ref{th:Z_g_p_0}, which gives a 
volume expression for the volume $M_{g,n}(\alpha_1,\ldots,\alpha_p)$ of flat $\SU(n)$--connections on a surface $\mathcal{M}$ of 
genus $g$ with $p$ boundary components for $\alpha_1,\ldots,\alpha_p\in\mathcal{H}_{reg}$. 
Recall that $\mathcal{H}_{reg}^0$ denotes the set of regular conjugacy classes of $\SU(n)$.

\subsection{Parametrizations of conjugacy classes and volume form} 

Let us consider the standard parametrization of conjugacy classes in $\SU(n)$ 
given by
$$\mathcal{A}= \left\{t_1\geqslant\ldots\geqslant t_n \mid \sum_{i=1}^nt_i=0, \ t_1-t_n\leqslant 1 \right\}.$$
The set $\mathcal{A}$ is called an alcove of type $A_{n-1}$. 
Remark that $\mathcal{A}$ is a polytope of dimension $n-1$ in $\mathbb{R}^n$, 
and for any $R\subset \{1,\ldots,n\}$ of cardinal $n-1$, the projection 
$p_{R}:\mathcal{A}\rightarrow \mathbb{R}^R$ yields a non-zero volume form 
$p_{R}^{*}d\ell_{\mathbb{R}^R}$ on $\mathcal{A}$. 
This volume form is independent of $R$ and denoted by $\diff t$ in the sequel. 
Choosing for example $R=\{1,\ldots n-1\}$, we have
\begin{align*}
Vol(\mathcal{A})=&\int_{\mathbb{R}^{n-1}}\indic_{t_1\geqslant\ldots\geqslant t_{n-1},\,t_1+\sum_{i=1}^{n-1}t_i\leqslant 1}\prod_{i=1}^{n-1} \diff t_i\\
=&\int_{\mathbb{R}^{n-1}}\indic_{1\geqslant u_1\geqslant\ldots\geqslant u_{n-1}\geqslant 0}\frac{1}{n}\prod_{i=1}^{n-1} \diff u_i=\frac{1}{n!} \ ,
\end{align*}
where we did the change of variable $\eta:(t_i)_{1\leqslant i\leqslant n-1}\mapsto (t_i+\sum_{j=1}^{n-1}t_j)_{1\leqslant i\leqslant n-1}$ with $Jac(\eta(t))=n$.

Denote by $\mathcal{A}_{reg}$ the subset of $\mathcal{A}$ consisting of tuples $(t_1>\ldots>t_n)$ with $t_1-t_n<1$. The set $\mathcal{A}_{reg}$ parametrizes then the regular conjugacy classes of $\SU(n)$, as $\mathcal{H}^0_{reg}$ does. There is moreover a piecewise affine bijection 
$\phi:\mathcal{A}_{reg}\rightarrow\mathcal{H}^0_{reg}$ whose value on $(t_1>\ldots> t_i> 0>t_{i+1}> \dots> t_n)$ is 
\begin{equation}\label{eq:change_parametrization_conjugacy}
\phi(t_1,\ldots,t_n) =(1+t_{i+1},\ldots,1+t_{n},t_1,\ldots,t_{i}) \ .
\end{equation}
One has $Jac(\phi(t))=1$ for all $t\in\mathcal{A}_{reg}$, and thus $\phi$ is volume preserving. One then checks that $\phi^* \diff \theta= \diff t$.

\subsection{Contraction formula on moduli spaces of flat connections} 
\label{subsec:contraction_formulas_flat_co}

Let $\widehat{\mathcal{M}}$ be a (possibly disconnected) oriented surface with $(p+2)$ boundary components 
$L_1,\ldots,L_p,L_{p+1}, L_{p+2}$. Let $\mathcal{M}$ be the oriented surface obtained by gluing $L_{p+1}$ 
and $L_{p+2}$ in an orientation reversing way and suppose that $\mathcal{M}$ is connected. 
Then, if we denote by $M(\mathcal{M},\alpha_1,\ldots,\alpha_p)$ the moduli space of flat $\SU(n)$-connections on $\mathcal{M}$, 
the following formula from \cite{Meinrenken_Woodward} holds for the its volume.

\begin{theorem}[{\autocite[Prop. 5.4]{Meinrenken_Woodward}}]\label{thm:contraction_formula_connection}
Suppose that $\alpha_1,\ldots,\alpha_p\in \mathcal{H}_{reg}^0$. If $M(\mathcal{M},\alpha_1,\ldots,\alpha_p)$ 
contains at least one connection whose stabilizer is $Z(\SU(n))$, then
\begin{align*}
\Vol(M(\mathcal{M},\alpha_1,\ldots,\alpha_p))=\frac{1}{k} 
\int_{\mathcal{A}}\Vol(M(\widehat{\mathcal{M}},\alpha_1,\ldots,\alpha_p,\phi(t),\phi(-t)))\Delta(t)^2 \diff t \ ,
\end{align*}
where $k=1$ if $\widehat{\mathcal{M}}$ is connected and $\#Z(\SU(n))=n$ otherwise, 
$dt$ is the previous volume form on $\mathcal{A}$ and 
$\Delta(t)=2^{n(n-1)/2}\prod_{1\leqslant i<j\leqslant n}\sin\left(\pi(t_i-t_j)\right)^2$.
\end{theorem}

\noindent
Note that we slightly adapted the result of \cite{Meinrenken_Woodward}, which was stated for the torsion volume of \cite{witten1992two}, 
to translate it in the symplectic picture : this only consists in adding the term $\Delta(t)$. 
Let us first remark that considering $\U(n)$-valued connection instead of $\SU(n)$-valued connection does not change the volume. 
Indeed, suppose that $\mathcal{M}$ corresponds to a surface of genus $g$ with $p$ boundary components. 
Then, for any $\alpha_1,\ldots,\alpha_{p}\in\mathcal{H}_{reg}$,
\begin{align*}
&M_{\U(n)}(\mathcal{M},\alpha_1,\ldots,\alpha_{p})\\
&\simeq \left\{\left((U_i)_{1\leqslant i\leqslant 2g},C_1,\ldots,C_{p}\right) \in \U(n)^{2g}\times\mathcal{O}_{\alpha_1}\times\dots\times 
\mathcal{O}_{\alpha_{p}}\Big\vert\prod_{i=1}^g[U_{2i-1},U_{2i}]=\prod_{i=1}^{p}C_i\right\}/ \U(n) \ ,
\end{align*} 
where $\U(n)$ acts diagonally by conjugation. 
Since $\det\prod_{i=1}^g[U_{2i-1},U_{2i}]=1$, the latter set is non-empty only if $\sum_{i=1}^p\vert \alpha_i\vert\in \mathbb{N}$. 
Next, remark that any conjugacy class of $\SU(n)$ is also a conjugacy class of $\U(n)$ and there is a natural action of $\mathbb{R}$ on $\mathcal{H}_{reg}$ given by 
$$t\cdot (\theta_1>\ldots >\theta_n)=std(\theta_i+t\mod \mathbb{Z}),$$
where $std(x_1,\ldots,x_n)$ denotes the standardization $(x_{i_1}>x_{i_2}>\ldots>x_{i_n})$. For $\alpha\in\mathcal{H}_{reg}$, set $t(\alpha)=\vert \alpha\vert\mod \mathbb{Z}$ 
and $\hat{\alpha}=(-t_{\alpha}/n)\cdot \alpha$, so that $\hat{\alpha}\in \mathcal{H}^0_{reg}$.
Since $Z(\U(n))$ acts trivially by conjugation, when $\alpha_{1},\ldots \alpha_{p}\in \mathcal{H}_{reg}$ we have
\begin{align*}
&M_{\U(n)}(\mathcal{M},\alpha_1,\ldots,\alpha_{p})\\
&\simeq \left\{((U_i)_{1\leqslant i\leqslant 2g},C_1,\ldots,C_{p})\in \U(n)^{2g}\times\mathcal{O}_{\alpha_1}\times\dots\times 
\mathcal{O}_{\alpha_{p}}\Big\vert\prod_{i=1}^g[U_{2i-1},U_{2i}]=\prod_{i=1}^{p}C_i\right\}/ \SU(n)\\
&\simeq \mathbb{T}^{2g}\times\left\{((U_i)_{1\leqslant i\leqslant 2g},C_1,\ldots,C_{p})\in \SU(n)^{2g}\times\mathcal{O}_{\hat{\alpha}_1}\times\dots\times 
\mathcal{O}_{\hat{\alpha}_{p}}\Big\vert\prod_{i=1}^g[U_{2i-1},U_{2i}]=\prod_{i=1}^{p}C_i\right\}/ \SU(n)\\
&\simeq \mathbb{T}^{2g}\times M_{\SU(n)}(\mathcal{M},\hat{\alpha}_1,\ldots,\hat{\alpha}_{p}),
\end{align*} 
so that with the convention that $\Vol(\mathbb{T})=1$, 
\begin{equation}\label{eq:SU_n_equal_Un}
\Vol \left[M_{\U(n)}\left(\mathcal{M},\alpha_1,\ldots,\alpha_{p}\right)\right] = 
\Vol \left[M_{\SU(n)}(\mathcal{M}, \hat{\alpha}_1,\ldots,\hat{\alpha}_{p})\right] .
\end{equation}
We then derive the following proposition. Recall that for $\theta = (\theta_1 > \dots > \theta_n) \in \mathcal{H}_{reg}$, 
we denote by $\tilde{\theta}$ the element of $\mathcal{H}_{reg}$ given by 
\begin{equation*}
    \tilde{\theta} = (1 - \theta_n > \dots > 1 - \theta_1) \ .
\end{equation*}

\begin{proposition}\label{prop:volume_formula_flat}
Suppose that $\alpha_1,\ldots,\alpha_p\in \mathcal{H}_{reg}^0$. If either $p\geqslant 3$ or $(p,g(\mathcal{M}))=(1, 1)$ 
or $g(\mathcal{M})\geqslant 2$, then, if $\widehat{\mathcal{M}}$ is disconnected, 
\begin{align*}
\Vol(M(\mathcal{M},\alpha_1,\ldots,\alpha_p)) = 
\frac{1}{n}\int_{\mathcal{H}_{reg}^0} 
\Vol(M(\widehat{\mathcal{M}},\alpha_1,\ldots,\alpha_p,\theta,\tilde{\theta}))\Delta(\theta)^2 \diff \theta \ ,
\end{align*}
and, if $\widehat{\mathcal{M}}$ is connected,
\begin{align*}
\Vol(M(\mathcal{M},\alpha_1,\ldots,\alpha_p))=\int_{\mathcal{H}_{reg}}\Vol(M(\widehat{\mathcal{M}}, 
\alpha_1,\ldots,\alpha_p,\theta,\tilde{\theta}))\Delta(\theta)^2 \diff \theta \ .
\end{align*}
\end{proposition}

\begin{proof}
First, by \autocite[Thm. 5.20]{bismut5symplectic}, $M_{g,n}(\alpha_1,\ldots,\alpha_p)$ contains at least one element 
for which the stabiliser under the diagonal action of $\SU(n)$ is $Z(\SU(n))$ if $\alpha_1,\ldots,\alpha_p\in\mathcal{H}_{reg}^0$ 
and either $p\geqslant 3$, $p=1$ and $g(\mathcal{M})=1$, or $g(\mathcal{M})\geqslant 2$.
Remark that the volume form $\diff t$ on $\mathcal{A}$ introduced in the previous subsection is such that 
$$\Vol\left(\left\{(t_1,\ldots,t_n)\in \mathbb{R}^n \, \Big\vert \, \sum_{i=1}^nt_i=0, \max_{1\leqslant i,j\leqslant n}t_i-t_j<1\right\}\right)=1 \ .$$
By \eqref{eq:change_parametrization_conjugacy}, we have that $\phi(-t)=\widetilde{\phi(t)}$ and $Jac(\phi(t))=1$ for all $t\in\mathcal{A}$. 
Hence, doing the change of variable $\theta=\phi(t)$ yields
$$\Vol(M(\mathcal{M},\alpha_1,\ldots,\alpha_p))=\frac{1}{k}\int_{\mathcal{H}_{reg}^0}
\Vol(M(\widehat{\mathcal{M}},\alpha_1,\ldots,\alpha_p,\theta,\widetilde{\theta}))\Delta(\theta)^2 \diff \theta,$$
where $k=1$ if $\widehat{\mathcal{M}}$ is connected and $k=n$ otherwise. 
It remains to replace the integration on $\mathcal{H}^0_{reg}$ by the integration on $\mathcal{H}_{reg}$ in the case where $\widehat{\mathcal{M}}$ is connected. 
For all $t\in [0,1/n)$, $\alpha_1,\ldots,\alpha_p\in \mathcal{H}_{reg}^0$ and $\theta\in \mathcal{H}_{reg}^0$, by \eqref{eq:SU_n_equal_Un}
$$\Vol\left[ M_{\U(n)}(\widehat{\mathcal{M}},\alpha_1,\ldots,\alpha_p,t\cdot\theta,\widetilde{t\cdot\theta})\right]  = 
\Vol \left[M_{\SU(n)}\left(\widehat{\mathcal{M}},\alpha_1,\ldots,\alpha_p,\theta,\widetilde{\theta}\right)\right] . $$
Therefore, if $\hat{\mathcal{M}}$ is connected,
\begin{align*}
&\frac{1}{k}\int_{\mathcal{H}_{reg}^0}\Vol(M(\widehat{\mathcal{M}},\alpha_1,\ldots,\alpha_p,\theta,\widetilde{\theta}))\Delta(\theta)^2 \diff \theta\\
&\hspace{2cm}=n\int_{0}^{1/n}\left(\int_{\mathcal{H}_{reg}^0}\Vol(M_{\U(n)}(\widehat{\mathcal{M}},\alpha_1,\ldots,\alpha_p,t\cdot\theta, 
\widetilde{t\cdot\theta}))\Delta(\theta)^2d\theta\right) \diff t\\
&\hspace{2cm}=\int_{\mathcal{H}_{reg}}\Vol(M_{\U(n)}(\widehat{\mathcal{M}},\alpha_1,\ldots,\alpha_p,u,\widetilde{u}))\Delta(u)^2 \diff u \ ,
\end{align*}
where we use that the change of variable $(u_1,\ldots,u_n)=
(\theta_1+t,\ldots,\theta_{n-1}+t,-\sum_{i=1}^{n-1}\theta_i+t)=:\phi(\theta_1,\ldots,\theta_{n-1},t)$ yields $Jac(\phi)=n$.
\end{proof}

\subsection{Proof of Theorem \ref{th:Z_g_p_0}}
\label{subsec:proof_th_flat_co}

The proof of Theorem \ref{th:Z_g_p_0} is then a deduction from the previous results and the previous construction on differential structures.

\begin{proof}[Proof of Theorem \ref{th:Z_g_p_0}]
The proof is done by induction on $N=3g+p$, where $N\geqslant 3$. If $N=3$, the condition 
$p + 2g-2 \geqslant 1$ implies that $p=3$ and $g=0$. The surface is therefore the three-holed sphere and 
the result is given by Theorem \ref{th:volume_flat_connection_0_3}. 
Suppose $N>3$ and let $S$ be a surface of genus $g$ with $p$ points removed. 
Let $\mathcal{T}$ be a surface constructed in Section \ref{subsec:(g,p)_honey}. 
Then $\mathcal{T}$ is obtained either by gluing two edges of a connected surface $\mathcal{T}'$ or by gluing one edge of a connected surface 
$\mathcal{T}'$ to the edge of an equilateral triangle $T$. 
\\
In the first case, $\mathcal{T}'$ is a flat surface associated to a surface $\mathcal{M}'$ with genus $g-1$ and $p+2$ points removed. 
Let $\alpha_1,\ldots,\alpha_p\in\mathcal{H}_{reg}$ with $\sum_{i=1}^p\vert\alpha_{i}\vert\in \mathbb{N}$. 
By applying \eqref{eq:SU_n_equal_Un} and Theorem \ref{prop:volume_formula_flat}, we have 
\begin{align*}
Z_{g,p}(\alpha_1,\dots,\alpha_p)=Z_{g,p}(\hat{\alpha}_1,\dots,\hat{\alpha}_p) = 
&\int_{\mathcal{H}_{reg}}\Vol(M(\mathcal{M}',\hat{\alpha}_1,\ldots,\hat{\alpha}_p,\theta,\tilde{\theta}))\Delta(\theta)^2 \diff \theta \\
=&\int_{\mathcal{H}_{reg}}\Vol(M(\mathcal{M}',\alpha_1,\ldots,\alpha_p,\theta,\tilde{\theta}))\Delta(\theta)^2\diff \theta \ .
\end{align*}
Let us denote by $c_{g, p} = \frac{c_{0,3}^{N}}{n^{2g+p-3}}$ where $c_{0, 3}$ is given in Theorem \ref{th:Z_g_p_0}. 
Since $3(g-1)+p+2<N$, by induction
$$\Vol(M(\mathcal{M}',\alpha_1,\ldots,\alpha_p,\theta,\tilde{\theta})) = 
\frac{c_{g-1,p+2}}{\Delta(\theta)\Delta(\tilde{\theta}) \prod_{j=1}^p\Delta(\alpha_j)}\sum_{\substack{G \in \mathcal{G}^{(g-1,p+2)}}}
        \Vol \left[ \honey^G_{\mathcal{T'}}(\alpha_1, 
        \dots, \alpha_p,\theta,\tilde{\theta}) \right],$$
and thus by Proposition \ref{prop:formula_contracting_honey},
\begin{align*}
Z_{g,p}(\alpha_1,\dots,\alpha_p)=
&\frac{c_{g-1,p+2}}{\prod_{j=1}^{p}\Delta(\alpha_j)}
\int_{\mathcal{H}_{reg}}\sum_{\substack{G \in \mathcal{G}^{(g-1,p+2)}}}
        \Vol \left[ \honey^G_{\mathcal{T'}}(\alpha_1, 
        \dots, \alpha_p,\theta,\tilde{\theta})\right]\frac{\Delta(\theta)^2}{\Delta(\theta)\Delta(\tilde{\theta})} \diff \theta\\
        =&\frac{c_{g-1,p+2}}{\prod_{j=1}^{p}\Delta(\alpha_j)}\sum_{\substack{G \in \mathcal{G}^{(g,p)}}}
        \Vol \left[ \honey^G_{\mathcal{T}}(\alpha_1, 
        \dots, \alpha_p) \right],
 \end{align*}
where we used the fact that $\Delta(\theta)=\Delta(\tilde{\theta})$ on the second equality. 
Since $c_{g-1, p+2} = c_{g, p}$, one gets the result.
In the second case, $\mathcal{T}$ is obtained by gluing a 
surface $\mathcal{M}'$ with genus $g$ and $p-1$ points removed and a triangle $T$ associated to a three-holed sphere. 
Let $\widehat{M}=\mathcal{M'}\cup T$ be the corresponding disconnected surface. Then, by \eqref{eq:SU_n_equal_Un} and Proposition \ref{prop:volume_formula_flat},
\begin{align*}
Z_{g,p}(\alpha_1,\dots,\alpha_p)=
&Z_{g,p}(\hat{\alpha}_1,\dots,\hat{\alpha}_p)\\
=&\frac{1}{n}\int_{\mathcal{H}_{reg}^0}\Vol(M(\mathcal{M}', \hat{\alpha}_1,\ldots,\hat{\alpha}_{p-2},\theta)) 
\Vol(M(T,\tilde{\theta},\hat{\alpha}_{p-1},\hat{\alpha}_p)) \diff \theta.
\end{align*}
Set $s=-\sum_{i=1}^{p-2}\vert \alpha\vert_i$. By \eqref{eq:SU_n_equal_Un},
$$\Vol(M(\mathcal{M}',\hat{\alpha}_1,\ldots,\hat{\alpha}_{p-2},\theta)) = 
\Vol(M(\mathcal{M}',\alpha_1,\ldots,\alpha_{p-2},\theta+s))$$
and
$$ \Vol(M(T,\tilde{\theta},\hat{\alpha}_{p-1},\hat{\alpha}_p))=\Vol(M(T,\tilde{\theta}-s,\alpha_{p-1},\alpha_p)).$$
Since the map $\theta\mapsto \theta-s$ is volume preserving,
\begin{align*}
&\int_{\mathcal{H}_{reg}^0}\Vol(M(\mathcal{M}',\hat{\alpha}_1,\ldots,\hat{\alpha}_{p-2},\theta)) 
\Vol(M(T,\tilde{\theta},\hat{\alpha}_{p-1},\hat{\alpha}_p)) \diff \theta\\
&\hspace{3cm}=\int_{\mathcal{H}_{reg}^{s}}\Vol(M(\mathcal{M}',\alpha_1,\ldots,\alpha_{p-2},\theta))
\Vol(M(T,\tilde{\theta},\alpha_{p-1},\alpha_p))\diff \theta.
\end{align*}
By induction,
$$\Vol(M(\mathcal{M}',\alpha_1,\ldots,\alpha_{p-2},\theta,)) = 
\frac{c_{g,p-1}}{\Delta(\theta)\prod_{i=1}^{p-2}\Delta(\alpha_i)}\sum_{\substack{G \in \mathcal{G}^{(g,p-1)}}}
        \Vol \left[ \honey^G(\alpha_1, 
        \dots, \alpha_{p-2},\theta) \right],$$
 and
 $$\Vol(M(T,\tilde{\theta},\alpha_{p-1},\alpha_p)) = 
 \frac{c_{0,3}}{\Delta(\tilde{\theta})\Delta(\alpha_{p-1}),\Delta(\alpha_{p})}\sum_{\substack{G \in \mathcal{G}^{(0,3)}}}
        \Vol \left[ \honey^G(\tilde{\theta}, \alpha_{p-1},\alpha_{p}) \right],$$
 and thus, by Proposition \ref{prop:formula_gluing_honey},
\begin{align*}
Z_{g,p}(\alpha_1,\dots,\alpha_p)
 =&\frac{c_{g,p-1}c_{0,3}}{\prod_{i=1}^{p}\Delta(\alpha_{i})}\frac{1}{n}\int_{\mathcal{H}_{reg}^s}\sum_{\substack{G_1 \in \mathcal{G}^{(g,p-1)},G_2\in \mathcal{G}^{(0,3)}}}
        \Vol \left[ \honey^{G_1}(\alpha_1, 
        \dots, \alpha_{p-2},\theta)\right]\\
        &\hspace{5cm}\Vol \left[ \honey^{G_2}(\tilde{\theta}, 
        \alpha_{p-1},\alpha_{p})\right]\frac{\Delta(\theta)^2}{\Delta(\theta)\Delta(\tilde{\theta})}d\theta\\
        =&\frac{c_{g,p}}{\prod_{i=1}^{p}\Delta(\alpha_{i})}\sum_{\substack{G \in \mathcal{G}^{(g,p)}}}
        \Vol \left[ \honey^G(\alpha_1, 
        \dots, \alpha_p) \right] 
\end{align*}
since $c_{g,p}=\frac{c_{g,p-1}c_{0,3}}{n}$.
\end{proof}

\section{Yang-Mills partition function on compact oriented surfaces}
\label{sec:proof_positive_area}

The goal of this section is to prove Corollary \ref{cor:yang_mills_g_p}, which provides an explicit volume formula 
for the marginal Yang–Mills partition function on an oriented surface of genus $g$ 
with prescribed non-degenerate holonomies (up to conjugation) 
along a finite collection of disjoint loops. 
As shown in \cite{levy2003yang}, this partition function depends 
only on the prescribed conjugacy classes and on the areas of the 
connected components delimited by the loops. 

To each disjoint loops configuration $\mathcal{L}$, recall the construction of the oriented graph $\mathcal{T}(\mathcal{L})=(V,E)$ from Section \ref{subsec:Yang-mills_result} :
\begin{itemize}
\item the set $V$ of vertices of $\mathcal{T}(\mathcal{L})$ is the set of connected components of $S\setminus \bigcup_{i=1}^p\Gamma_i$. 
Each vertex $v\in V$ is labelled $(A_v,g_v)$ where $A_v$ is the area of the corresponding connected component and $g_v$ is its genus.
\item For $v_1,v_2\in V$, there is a directed edge $e$ from $v_1$ to $v_2$ 
if $v_1$ and $v_2$ are boundary components of a loop, and we write $v=s(e)$. We label the corresponding oriented edge $e$ by $\alpha_e$, with the condition that $\alpha_{\bar{e}}=1-\alpha_e$ if $\bar{e}$ is the oriented edge $(v_2,v_1)$.
\end{itemize}

\noindent
In this section, 
for $T \geqslant 0$ and $x,y \in \mathcal{H}$, let us denote by $k_T(x,y)$ the kernel of the unitary Dyson Brownian motion on $\mathcal{H}$ at time $T$. This kernel has an explicit expression for $x,y\in\mathcal{H}_{reg}$ as 
\begin{equation}\label{eq:kernel_brownian}
k_{T}(x,y)=\frac{\Delta(y)}{\Delta(x)}\det(p_T(2\pi x_i,2\pi y_j),1\leqslant i,j\leqslant n),
\end{equation}
where $p_{T}$ is the heat kernel of the Brownian motion on the circle $\mathbb{R}/2\pi \mathbb{Z}$.

\begin{lemma}[Partition function of a cylinder]
    Let $T > 0$, $\alpha_1\in\mathcal{H},\alpha_2 \in \mathcal{H}_{reg}$. Then, the Yang-Mills partition function associated to a cylinder of volume $T$ with holonomies $\alpha_1,\alpha_2$ on its boundaries is
    \begin{equation*}
        Z_{0, 2, T}(\alpha_1, \alpha_2) =\frac{1}{\Delta(\alpha_2)^2} k_T(\alpha_1,\alpha_2).
    \end{equation*}
\end{lemma}

\begin{proof}
    This is a consequence of \cite[Proposition 4.2.4 and $(5.3)$]{levy2003yang}.
\end{proof}

\begin{proposition}[Volume formula for Yang-Mills partition function]\label{prop:volume_formula_yang_mills}
    Let $ g \geqslant 0,p\geq 1$ be integers, $(\alpha_1, \dots, \alpha_p) \in \mathcal{H}^p$ 
    and $T > 0$. Assume that loops $\gamma=\{\gamma_1,\ldots,\gamma_p\}$ associated to $\alpha_1, \dots, \alpha_p$ enclose contractible domains $D_1,\ldots,D_p$ and that $S\setminus \bigcup_{i=1}^pD_i$ has volume $T$. Then, the Yang--Mills partition function on $S\setminus \bigcup_{i=1}^pD_i$ with holonomy $\alpha_i$ around $\gamma_i$ is given by 
    \begin{equation*}
        Z_{g,p, T}(\alpha_1, \dots, \alpha_p) 
        =\int_{\mathcal{H}} Z_{g, p}(u_1, \dots, u_p) \, 
        \prod_{i=1}^pk_{\frac{T}{p}}(1-u,\alpha_i)  \prod_{\ell=1}^p\diff u_i  \ .
    \end{equation*}
\end{proposition}

\begin{proof}
Recall that the partition function is invariant by area preserving diffeomorphisms. Let us introduce a curve $\gamma$ enclosing the curve $\gamma_p$ such that the area of the domain enclosed by $\gamma_p$ and $\gamma$ is $T'$, with $0< T'< T$. Then, by the Markov property of the Yang-Mills partition function, see \cite[Proposition 5.1.2]{levy2003yang} or \cite[Eq. (2.63)]{witten1992two},
$$Z_{g, p, T}(\alpha_1, \dots, \alpha_p)=\int_{\mathcal{H}} Z_{g, p, T-T'}(\alpha_1, \dots, \alpha_{p-1}, u) \, 
        Z_{0, 2, T'}(1-u, \alpha_p) \Delta(u)^2\diff u \ .$$
Choosing $T'=T(1-\frac{1}{p})$ and doing the same for the $p-1$ other curves yields
$$Z_{g, p, T}(\alpha_1, \dots, \alpha_p)=\int_{\mathcal{H}} Z_{g, p}(u_1, \dots, u_p) \, 
        \prod_{i=1}^pk_{\frac{T}{p}}(1-u_i, \alpha_i)\prod_{i=1}^p \diff u_i \ ,$$
where $Z_{g, p}(u_1, \dots, u_p)=Z_{g, p,0}(u_1, \dots, u_p)$ is the 
the volume of flat connections with holonomies $u_1,\ldots,u_p$ on the boundary, see \cite{sengupta2003volume}. 
\end{proof}
\begin{proof}[Proof of Corollary \ref{cor:yang_mills_g_p}]

The proof is done by induction on the number $p$ of loops on $S$. Let $\alpha_p$ be the holonomy around the oriented loop $\gamma_p$ of $S$, and denote by $\tilde{S}$ the surface obtained by cutting $S$ along $\gamma_p$. 

If $\tilde{S}=S_1\sqcup S_2$ is disconnected, with $\gamma_1$ being positively oriented on $S_1$, then by  \cite[Proposition 5.1.3]{levy2003yang}.
$$\operatorname{YM}_{\mathcal{L}}(\alpha_1,\ldots,\alpha_p)=\operatorname{YM}_{\mathcal{L}_{1}}(\alpha_{i_1},\ldots,{i_s},\alpha_p)\operatorname{YM}_{\mathcal{L}_2}(\alpha_{j_1},\ldots,\alpha_{j_s},1-\alpha_p),$$
where $\mathcal{L}_1$ is the loop configuration given by $S_1$ together with the loops $\gamma_{i_1},\ldots,\gamma_{i_r},\gamma_{p}$ which lie on $S_1$, and $\mathcal{L}_2$ is the loop configuration given by $S_2$ and the remaining loops $\gamma_{i_1},\ldots,\gamma_{i_r},\gamma_{p}$ lying on $S_2$. 

If $\tilde{S}$ is connected and $\tilde{\mathcal{L}}$ is the loop configuration on $\tilde{S}$ obtained by keeping the loops $\gamma_1,\ldots,\gamma_{p-1}$ and the two copies $\gamma_p^+,\gamma_p^-$ obtained by cutting along $\gamma_p$, then by  \cite[Proposition 5.4.3]{levy2003yang},
$$\operatorname{YM}_{\mathcal{L}}(\alpha_1,\ldots,\alpha_p)=\operatorname{YM}_{\tilde{\mathcal{L}}}(\alpha_1,\ldots,\alpha_p,1-\alpha_p).$$
Iterating on all the loops yields 
$$\operatorname{YM}_{\mathcal{L}}(\alpha_1,\ldots,\alpha_p) 
        = 
        \prod_{v\in V}Z_{g_v, d_v,T}(\alpha_{e_1^v}, \dots, \alpha_{e_{d_v}^v}) ,$$
        where for each $v\in V$, $e_1^v, \dots, e_{d_v}^v$ are the oriented edges starting from $v$ in the graph $\mathcal{T}(\mathcal{L})$ constructed from $\mathcal{L}$ above. Using Proposition \ref{prop:volume_formula_yang_mills} yields the result.
\end{proof}
\section{Acknowledgments}
Q.F would like to thank Thibaut Lemoine for useful discussions on this paper and related notions. P.T is supported by the Agence Nationale de la Recherche funding ANR CORTIPOM 21-CE40-001.

\printbibliography

\end{document}